\documentclass[a4paper]{article}

\usepackage{amsfonts}
\usepackage{amssymb,amsmath,amsthm,mathrsfs}
\usepackage{geometry}
\usepackage{enumerate}
\usepackage{bbm}
\usepackage{graphicx}
\usepackage{bbm}
\usepackage{xcolor}

\usepackage[colorlinks,
            linkcolor=blue,
            anchorcolor=green,
            citecolor=blue
            ]{hyperref}

\numberwithin{equation}{section}
\newtheorem{theorem}{Theorem}[section]
\newtheorem{lemma}[theorem]{Lemma}
\newtheorem{proposition}[theorem]{Proposition}

\newtheorem{definition}[theorem]{Definition}
\newtheorem{remark}[theorem]{Remark}
\newtheorem{prob}[theorem]{\bf Problem}

\newcommand{\Vol}{{\rm vol}}
\newcommand{\KK}{\mathfrak K}
\newcommand{\TBRW}{\mathscr{T}^{\rm B}}
\newcommand{\TWN}{\mathscr{T}^{\rm W}}
\newcommand{\T}{\mathscr{T}}

\title{Percolation of thick points of the log-correlated Gaussian field in high dimensions}
\author{Jian Ding\thanks{Peking University}\qquad Ewain Gwynne\thanks{University of Chicago}\qquad Zijie Zhuang\thanks{University of Pennsylvania}}

\begin{document}

\maketitle

\begin{abstract}
    We prove that the set of thick points of the log-correlated Gaussian field contains an unbounded path in sufficiently high dimensions. This contrasts with the two-dimensional case, where Aru, Papon, and Powell (2023) showed that the set of thick points is totally disconnected. This result has an interesting implication for the exponential metric of the log-correlated Gaussian field: in sufficiently high dimensions, when the parameter $\xi$ is large, the set-to-set distance exponent (if it exists) is negative. This suggests that a new phase may emerge for the exponential metric, which does not appear in two dimensions. In addition, we establish similar results for the set of thick points of the branching random walk. As an intermediate result, we also prove that the critical probability for fractal percolation converges to 0 as $d \to \infty$.
\end{abstract}

\setcounter{tocdepth}{1}
\tableofcontents

\section{Introduction}

\subsection{Thick points of the log-correlated Gaussian field}

Consider $\mathbb{R}^d$ with $d \geq 2$. Let $\phi$ be the whole-plane log-correlated Gaussian field, which is a random generalized function on $\mathbb{R}^d$ defined up to a global additive constant. The covariance of $\phi$ is given by
\begin{equation}
    \label{eq:cov-lgf}
    {\rm Cov}[(\phi, f_1), (\phi, f_2)] = \int_{\mathbb{R}^d \times \mathbb{R}^d} f_1(x) f_2(y) \log \frac{1}{|x-y|} dx dy,
\end{equation}
where $f_1$ and $f_2$ are smooth and compactly supported functions with average 0. For $d=2$, the log-correlated Gaussian field is the same as the whole-plane Gaussian free field (GFF). However, for $d \geq 3$, the GFF is not the same as the log-correlated Gaussian field. We refer to~\cite{lgf-survey, fgf-survey} for basic properties of this process. Throughout the paper, we choose the global additive constant of $\phi$ such that the average of $\phi$ over the unit sphere is 0; see \cite[Section 11.1]{fgf-survey} for the well-definedness of the spherical average of $\phi$. In fact, this particular choice of normalization does not affect the set of thick points defined below, as it only shifts $\phi$ by some random constant.

The thick points of $\phi$ are formally those places where the value of $\phi$ is exceptionally large. Since $\phi$ is not defined pointwise, a regularization procedure is required to define the thick points, as we now elaborate. Let $\rho$ be a smooth, compactly supported and non-negative function on $\mathbb{R}^d$ with $\int_{\mathbb{R}^d} \rho(x) dx = 1$. For integers $n \geq 0$, define the approximation $\phi_n$ of $\phi$ as
$$
\phi_n = \phi * \rho_n \quad \mbox{with} \quad \rho_n(x):=2^{nd} \rho(2^n x) \quad \mbox{for } x \in \mathbb{R}^d.
$$
In other words, $\phi_n(x) = \int_{\mathbb{R}^d} \phi(y) 2^{nd}\rho(2^n(x - y)) dy$ for all $x \in \mathbb{R}^d$. For $\alpha > 0$, define the $\alpha$-thick point set
\begin{equation}\label{eq:lgf-thick}
\T_\alpha := \{ x \in \mathbb{R}^d : \liminf_{n \rightarrow \infty} \frac{1}{n}\phi_n(x) \geq \alpha \sqrt{\log 2} \}.
\end{equation}
We normalize by the factor $\sqrt{\log 2}$ since in our choice of approximation, ${\rm Var}[\phi_n(x)] = n \log 2 + O(1)$. In fact, the sets of thick points are almost surely the same for different choices of $\rho$; see Lemma~\ref{lem:approx-thick} for the proof. In~\cite{hmp-thick-pts, CH-thick-point}, it was shown that the Hausdorff dimension of $\T_\alpha$ is almost surely equal to $d - \frac{\alpha^2}{2}$ when $0 < \alpha \leq \sqrt{2d} $. Moreover, when $\alpha > \sqrt{2d}$, we have $\T_\alpha = \emptyset$ almost surely. See also~\cite{Chen-thick} for studies about thick points for power-law-correlated Gaussian fields.

In this paper, we focus on the geometric properties of $\T_\alpha$, specifically its percolative properties. In two dimensions, where \(\phi\) is the Gaussian free field, it was proved in~\cite{APP-disconnect} that the set of thick points $\T_\alpha$ is totally disconnected for all \(\alpha > 0\), i.e., any connected component is a single point.\footnote{The result in~\cite{APP-disconnect} was stated for the thick point set defined using the circle average process $\{\widetilde h_n \}_{n \geq 1}$ with the condition $\lim_{n \to \infty} \widetilde h_n(x)/n = \alpha$. However, their argument can be easily extended to our case. Indeed, their result is based on comparing the thick point set of the GFF with that of the weighted nesting CLE$_4$ field. The thick point sets of the later are totally disconnected under both the lim and liminf conditions. Moreover, difference choices of approximation will not impact the thick point set; see the estimates in Lemma~\ref{lem:sec4-1}.} However, in sufficiently high dimensions, the set of thick points may contain an unbounded path.

\begin{theorem}\label{thm:lgf-thick}
    For any fixed $\alpha>0$, there exists a constant $C > 0$ depending on $\alpha$ such that for all $d \geq C$, the set $\T_\alpha$ contains an unbounded path almost surely.
\end{theorem}

We can also show that $\T_\alpha$ contains a path crossing any box.

\begin{theorem}\label{thm:crossing}
    Fix $\alpha>0$. The following holds almost surely for all sufficiently large $d$. For every box $B = [a_1,b_1] \times \ldots \times [a_d,b_d] \subset \mathbb{R}^d$ and any two opposite faces of $B$, the set $\T_\alpha \cap B$ contains a path connecting these two faces.
\end{theorem}

\begin{remark}
    \begin{enumerate}
        \item (Level surfaces of the log-correlated Gaussian field) In two dimensions, one can define the level lines of the Gaussian free field (GFF), which are SLE$_4$ curves~\cite{ss-contour}. To be precise, a level line is a continuous curve almost surely determined by the GFF. Furthermore, conditional on the level line, the restrictions of the GFF to the two sides of the curve are conditionally independent GFFs with constant boundary values, see e.g.~\cite{ss-contour, pw-gff-notes}. One may hope to extend the definition of level lines to higher dimensions where they would become level surfaces, as discussed in~\cite{lgf-survey}. However, our Theorem~\ref{thm:crossing} suggests that in sufficiently high dimensions, the level surface of the log-correlated Gaussian field might not be a nice object. In particular, we expect that the level surface cannot intersect the thick point set. Thus, by Theorem~\ref{thm:crossing}, it does not separate the space unlike in the two-dimensional case. 
        
        \item (Different definitions of thick points) One can use an alternative condition $\lim_{n \rightarrow \infty} \frac{1}{n}\phi_n(x) = \alpha \sqrt{\log 2}$ instead of the condition in~\eqref{eq:lgf-thick} to define the thick point set. This alternative condition was considered in many other papers, e.g.,~\cite{hmp-thick-pts, CH-thick-point, APP-disconnect}, but we adopt the condition as in~\eqref{eq:lgf-thick} for technical reasons. While the sets of thick points defined by these two conditions are different, we expect the same result to hold for the set of thick points under the alternative condition, though additional technical work would be required.

        \item (General log-correlated Gaussian fields). By the definition of thick points~\eqref{eq:lgf-thick}, Theorems~\ref{thm:lgf-thick} and~\ref{thm:crossing} extend to random fields that differ from the log-correlated Gaussian field in~\eqref{eq:cov-lgf} by a random continuous function. By \cite[Theorem A]{JSW-decomposition}, this includes the Gaussian log-correlated field with kernel $K(x,y) = \log \frac{1}{|x-y|} + g(x,y)$, where $g \in H_{\rm loc}^{d+\epsilon}(\mathbb{R}^d \times \mathbb{R}^d)$ for some $\epsilon>0$.
    \end{enumerate}
\end{remark}

\subsection{Thick points of other fields}

In this paper, we also consider other log-correlated Gaussian fields and prove that their thick point sets contain an unbounded path when the dimension is sufficiently high.

\subsubsection{Branching random walk and fractal percolation}\label{subsec:BRW}

We consider the branching random walk (BRW) defined as follows. For integers $j \geq 0$, let $\mathcal{B}_j$ be the set of disjoint $d$-dimensional boxes of the form $[0, 2^{-j})^d + 2^{-j} \mathbb{Z}^d$ that cover $\mathbb{R}^d$. For $x \in \mathbb{R}^d$, let $\mathcal{B}_j(x)$ denote the unique box in $\mathcal{B}_j$ that contains $x$. Let $\{ a_{j, B} \}_{j \geq 0, B \in \mathcal{B}_j}$ be i.i.d.\ Gaussian random variables with mean 0 and variance $\log 2$. For $n \geq 1$, the branching random walk $\mathcal{R}_n$ is defined as follows:
\begin{equation}\label{eq:def-brw}
\mathcal{R}_n(x) = \sum_{j = 0}^n a_{j, \mathcal{B}_j(x)} \quad \mbox{for all } x \in \mathbb{R}^d.
\end{equation}
For $\alpha>0$, define the thick point set
\begin{equation}\label{eq:brw-thick}
    \TBRW_\alpha := \{ x \in \mathbb{R}^d: \liminf_{n \rightarrow \infty} \frac{1}{n}\mathcal{R}_n(x) \geq \alpha \}.
\end{equation}

We have an analog of Theorem~\ref{thm:lgf-thick} for the branching random walk. 

\begin{theorem}\label{thm:BRW-thick}
    For any fixed $\alpha>0$, there exists a constant $C > 0$ depending on $\alpha$ such that for all $d \geq C$, the set $\TBRW_\alpha$ contains an unbounded path almost surely.
\end{theorem}

By the structure of the BRW, there is a simple way to find (part of) the thick point set as follows. At scale 0, we divide $\mathbb{R}^d$ into boxes in $\mathcal{B}_0$. For each box, we sample a Gaussian random variable (corresponding to $a_{0,B}$) and retain the box if and only if this random variable is at least $\alpha$. This corresponds to retaining each box independently with probability $p=\mathbb{P}[a_{0,B} \geq \alpha]$. At scale 1, we further divide each retained box into $2^d$ sub-boxes in $\mathcal{B}_1$, sample a Gaussian random variable (corresponding to $a_{1,B}$) for each sub-box, and retain it if and only if the random variable is at least $\alpha$, and so on. Let $A_\infty$ be the intersection of the retained sets at each scale. For all $x \in A_\infty$, we have $a_{j, \mathcal{B}_j(x)} \geq \alpha$ for all $j \geq 0$. Therefore, $\mathcal{R}_n(x) \geq \alpha n$ for all $n \geq 1$, which implies that $A_\infty \subset \TBRW_\alpha$. In fact, in Section~\ref{sec:brw}, we will prove that $A_\infty$ almost surely contains an unbounded path in sufficiently high dimensions.

This set $A_\infty$ is reminiscent of the fractal percolation studied in~\cite{mandelbrot-book, CCD-fractal}. Indeed, $A_\infty$ is almost the same as the retained set in fractal percolation in~\cite{CCD-fractal} with $N=2$ and open probability $p$ except that in the above procedure we retain the set $[0,2^{-j})^d$, while in fractal percolation, the set $[0,2^{-j}]^d$ is retained. Let $A_\infty'$ denote the retained set obtained in the fractal percolation, and thus, $A_\infty' \supset A_\infty$. Then, we have an analogous result for the fractal percolation. Define 
$$
p_c(d) = \inf \{p : \mathbb{P}_p[ \mbox{there exists an unbounded path in }A_\infty'] > 0 \}.
$$

\begin{theorem}
    We have $\lim_{d \to \infty} p_c(d) = 0$.
\end{theorem}

\begin{proof}
    For any fixed $p>0$, we will show in Section~\ref{sec:brw} that in sufficiently large dimensions, the set $A_\infty$ almost surely contains an unbounded path. Since $A_\infty' \supset A_\infty$, we obtain the theorem.
\end{proof}

It was shown in~\cite[Theorem 3]{CCD-fractal} that in $d = 2$, for $p \geq p_c$, the retained set $A_\infty'$ has a unique unbounded connected component, while for $p < p_c$, the retained set $A_\infty'$ is almost surely totally disconnected. To our knowledge, the situation in higher dimensions is not well understood. We refer to ~\cite{Chayes-fractal-3d} and Section [b] of~\cite{Chayes-notes} for more discussions. It would also be interesting to determine the correct order of $p_c(d)$ as $d$ tends to infinity. For Bernoulli site percolation on $\mathbb{Z}^d$, each vertex has $2d$ neighbors. Thus, a simple counting argument shows that the critical probability is at least $\frac{1}{2d}$; in fact, it is known that the critical probability is $\frac{1}{2d}(1+o_d(1))$ as $d \to \infty$~\cite{Kesten-high}. In fractal percolation, each $d$-dimensional box of the form $2^{-j} \mathbb{Z}^d + [0,2^{-j}]^d$ intersects $3^d-1$ other such boxes, so the same counting argument gives a lower bound $\frac{1}{3^d-1}$ for $p_c(d)$. We expect that $p_c(d)$ might decay exponentially in $d$, but the approach in this paper seems too coarse to establish this.


\subsubsection{The white noise field}

As explained in Section 4.1.1 of~\cite{lgf-survey}, the whole-plane log-correlated Gaussian field in~\eqref{eq:cov-lgf} can be obtained via the space-time white noise. We also have an analog of Theorem~\ref{thm:lgf-thick} for the thick point set defined using white noise (see Theorem~\ref{thm:WN-thick} below).

We first introduce some notation. Let $B_r(x) := \{ z: |z - x|_2 < r\}$ for $x \in \mathbb{R}^d$ and $r>0$, and let $\Vol(B_r(x))$ be the $d$-dimensional Lebesgue measure of $B_r(x)$. Define the function
$$
\KK(x) = \Vol(B_1(0))^{-\frac{1}{2}} 1_{ |x|_2 < 1 } \quad \mbox{for $x \in \mathbb{R}^d$}.
$$
Then, $\KK$ is a radially symmetric function and satisfies $\int_{\mathbb{R}^d} \KK(x)^2 dx = 1$. Let $W$ denote a space-time white noise on $\mathbb{R}^d \times (0,\infty)$. For integers $n \geq 1$ and $x \in \mathbb{R}^d$, let
\begin{equation} \label{eq:def-WN}
\begin{aligned}
   h_n(x) &:= \int_{\mathbb{R}^d} \int_{2^{-n}}^1 \KK \big(\frac{y-x}{t}\big) t^{-\frac{d+1}{2}} W(dy,dt);\\
   h(x) &:= \int_{\mathbb{R}^d} \int_0^1 \KK \big(\frac{y-x}{t}\big) t^{-\frac{d+1}{2}} W(dy,dt).\\
\end{aligned}
\end{equation}
By the Kolmogorov continuity theorem, there is a modification such that $\{h_n\}$ is a sequence of random continuous functions that approximate $h$. 
As explained in Section 4.1.1 of~\cite{lgf-survey} (see Appendix B of~\cite{afs-metric-ball} for the proof)\footnote{While Appendix B of~\cite{afs-metric-ball} was explained in two dimensions, the same proof works for $d \geq 3$. Their argument was also only stated for smooth functions, but as explained in Lemma~\ref{lem:lgf-approx}, the same result holds for our $\KK$.}, there exists a coupling of $W(dy,dt)$ and $\phi$ such that $h - \phi$ is a random continuous function. For $\alpha>0$, define the thick point set
\begin{equation}\label{eq:WN-thick}
\TWN_\alpha := \{ x \in \mathbb{R}^d : \liminf_{n \rightarrow \infty} \frac{1}{n} h_n(x) \geq \alpha \sqrt{\log 2} \}.
\end{equation}

We have an analog of Theorem~\ref{thm:lgf-thick} for $\TWN_\alpha$.

\begin{theorem}\label{thm:WN-thick}
     For any fixed $\alpha>0$, there exists a constant $C > 0$ depending on $\alpha$ such that for all $d \geq C$, the set $\TWN_\alpha$ contains an unbounded path almost surely.
\end{theorem}

In fact, we first prove Theorem~\ref{thm:WN-thick} and then show that, under the coupling of $W(dy,dt)$ and $\phi$, we have $\T_\alpha = \TWN_\alpha$ almost surely (Lemma~\ref{lem:approx-thick}). Theorem~\ref{thm:lgf-thick} then follows as a consequence of Theorem~\ref{thm:WN-thick}. Theorem~\ref{thm:crossing} also follows from the analogous result for $\TWN_\alpha$.

We make no attempt to optimize the value of $C$ obtained by our approach. Our methods can in principle lead to an explicit bound for $C$. However, we expect that any such bound will be very large and far from optimal. For example, our proof of Theorem~\ref{thm:BRW-thick} shows that the theorem is true with $C = C(\alpha) = \mathbb{P}[N(0, \log 2) \geq \alpha]^{-2000}$, where $N(0, \log 2)$ denotes a Gaussian random variable with mean 0 and variance $\log 2$. Similar comments apply for our other main theorem statements.


\subsection{The exponential metric of the log-correlated Gaussian field}

In this subsection, we discuss an implication for the exponential metric of the log-correlated Gaussian field (see Theorem~\ref{thm:exponential-metric} below). This metric is formally obtained by reweighting the Euclidean metric tensor by $e^{\xi \phi}$ for a parameter $\xi > 0$. Since $\phi$ is not defined pointwise, a normalization procedure is required to define this metric. We use $\{ h_n\}_{n \geq 1}$ defined in~\eqref{eq:def-WN} to approximate $\phi$.\footnote{Since $h$ and $\phi$ differ only by a random continuous function, if the metric $e^{\xi h}$ can be defined, the metric $e^{\xi \phi}$ can be obtained by reweighting $e^{\xi h}$ with the continuous function $e^{\xi (\phi - h)}$. Furthermore, it is expected that different methods of approximation will not change the resulting metric.} For integers $n \geq 1$, let $D_n$ be the exponential metric associated with $h_n$. Namely,
\begin{equation}\label{eq:def-exponential-metric}
D_n(z,w) := \inf_{P: z \rightarrow w} \int_0^1 e^{\xi h_n(P(t))} |P'(t)|dt \quad \mbox{for all } z,w \in \mathbb{R}^d,
\end{equation}
where the infimum is taken over all the piecewise continuously differentiable paths $P:[0,1] \rightarrow \mathbb{R}^d$ connecting $z$ and $w$. The exponential metric of the log-correlated field is obtained by taking the limit of $\{D_n(\cdot,\cdot)\}_{n\geq 1}$ after appropriate normalization. The topology of the convergence and the behavior of the limiting metric may depend on $\xi$. 

An important quantity characterizing the behavior of the exponential metric is the distance exponent, which depends on $\xi$. The point-to-point distance exponent $Q = Q(\xi)$ satisfies
\begin{equation}\label{eq:point-to-point}
    {\rm Med}(D_n(x, y)) = 2^{-(1-\xi Q)n + o(n)} \quad \mbox{as $n \rightarrow \infty$},
\end{equation}
where $x,y$ are any two fixed points in $\mathbb{R}^d$, and ${\rm Med}(X)$ denotes the median of the random variable $X$. It is known that $Q(\xi)$ exists and is a continuous, non-increasing function of $\xi$ (see~\cite[Proposition 1.1]{dgz-exponential-metric}, and also \cite{dg-supercritical-lfpp, dg-lqg-dim}).\footnote{The result in~\cite{dgz-exponential-metric} was proved under the assumption that $\KK$ is smooth, but their main results (Proposition 1.1 and Theorem 1.2) extend easily to our setting of $\KK$. For $d = 2$, many papers such as~\cite{dg-supercritical-lfpp, dg-lqg-dim} considered another approximation of the Gaussian free field via convolution with the heat kernel (as explained in~\cite[Section 3.1]{cg-support-thm}, this corresponds to taking $\KK(x) = \sqrt{\frac{2}{\pi}} e^{-|x|_2^2}$). It is straightforward to see that this approximation yields the same point-to-point distance exponent $Q(\xi)$ using e.g.\ the estimates in Lemma~\ref{lem:sec4-1}. Moreover, these two approximations also yield the same set-to-set distance exponent (if it exists) as defined below.} Moreover, it is straightforward to see that $Q(\xi)\geq 0$ for all $\xi >0$, as $h_n(x)$ is of order $O(\sqrt{n})$ with high probability around the point $x$. We introduce the threshold
$$
\xi_c := {\rm sup} \{\xi : Q(\xi) > \sqrt{2d} \} \in (0,\infty).
$$
As explained in ~\cite[Remark 1.3]{dgz-exponential-metric}, when $Q(\xi) < \sqrt{2d}$, the limiting metric (if it exists) should have singular points, i.e., points that have infinite distance to any other point. Here, the value $\sqrt{2d}$ arises from the fact that the maximum of $h_n$ in a unit box is approximately $ (\sqrt{2d} \log 2 + o(1))n$ with high probability~\cite{Mad15}. Furthermore, when $Q(\xi) > \sqrt{2d}$, it was shown in~\cite{dddf-lfpp, dgz-exponential-metric} that every possible subsequential limit of $D_n$, appropriately renormalized, is continuous and induces the Euclidean topology, see below for details. 

In light of the point-to-point distance exponent~\eqref{eq:point-to-point}, one can also consider the set-to-set distance exponent $\widetilde Q = \widetilde Q(\xi)$ which satisfies
\begin{equation}\label{eq:set-to-set}
    {\rm Med}(D_n(K_1, K_2)) = 2^{-(1-\xi \widetilde Q)n + o(n)} \quad \mbox{as $n \rightarrow \infty$},
\end{equation}
where $K_1$ and $K_2$ are any two fixed disjoint compact subsets of $\mathbb{R}^d$ with non-empty interiors. For $d = 2$, the exponent $\widetilde Q(\xi)$ exists and always equals $Q(\xi)$~\cite{dg-supercritical-lfpp, dg-lqg-dim}. For $d \geq 3$ and $\xi \in (0,\xi_c)$, as a consequence of~\cite[Theorem 1.2]{dgz-exponential-metric}, we know that $\widetilde Q(\xi)$ exists and equals $Q(\xi)$. However, for $d \geq 3$ and $\xi \geq \xi_c$, the existence of $\widetilde Q(\xi)$ has not been proved yet.

In two dimensions, there has been extensive research on the exponential metric and its properties. In particular, it corresponds to the so-called Liouville quantum gravity (LQG) metric, which is one of the central topics in random planar geometry. We briefly recall its construction, which was completed in~\cite{dddf-lfpp, gm-uniqueness, dg-supercritical-lfpp, dg-uniqueness}. Although the results in these papers were proved for a different approximation of the Gaussian free field via convolution with the heat kernel, the same results are expected to hold for the approximation in~\eqref{eq:def-exponential-metric}. In two dimensions, it is known that $Q(\xi) = \widetilde Q(\xi) > 0$ for all $\xi>0$~\cite{dg-supercritical-lfpp, lfpp-pos}, and the following hold:
\begin{enumerate}
    \item In the subcritical region $0 < \xi < \xi_c$ (or equivalently, $Q(\xi) = \widetilde Q(\xi) > 2$), after appropriate normalization, the metrics converge with respect to the topology of uniform convergence on compact subsets of $\mathbb{R}^2 \times \mathbb{R}^2$, and the exponential metric is continuous and induces the Euclidean topology~\cite{dddf-lfpp, gm-uniqueness}.

    \item In the supercritical region $\xi > \xi_c$ (or equivalently, $Q(\xi) = \widetilde Q(\xi) \in (0, 2)$), after appropriate normalization, the metrics converge with respect to the topology on lower semi-continuous functions on $\mathbb{R}^2 \times \mathbb{R}^2$, and the exponential metric is a metric on $\mathbb{R}^2$ though it is allowed to take on infinite values~\cite{dg-supercritical-lfpp, dg-uniqueness}. Indeed, there is an uncountable, totally disconnected, zero Lebesgue measure set of singular points $z$ that satisfy $D(z,w) = \infty$ for every $w \neq z$. Roughly speaking, the singular points correspond to the points of thickness greater than $Q(\xi)$; see \cite[Proposition 1.11]{pfeffer-supercritical-lqg}.

    \item In the critical case $\xi = \xi_c$ (or equivalently, $Q(\xi) = \widetilde Q(\xi) = 2$), after appropriate normalization, the metrics converge with respect to the topology on lower semi-continuous functions on $\mathbb{R}^2 \times \mathbb{R}^2$ (in fact, we expect that the convergence holds w.r.t.\ the uniform topology) and the exponential metric is continuous and induces the Euclidean topology~\cite{dg-supercritical-lfpp, dg-uniqueness, dg-critical-lqg}.
\end{enumerate}
We refer to~\cite{ddg-metric-survey} for more discussions on the two-dimensional case and its relation to other objects such as random planar maps.

In a recent work~\cite{dgz-exponential-metric}, we initiated the study of the exponential metric in dimensions at least three. In particular, we proved that in the subcritical region $\xi \in (0, \xi_c)$ (or equivalently, $Q(\xi) = \widetilde Q(\xi) > \sqrt{2d}$), after appropriate normalization, $\{D_n(\cdot, \cdot)\}_{n \geq 1}$ is tight with respect to the topology of uniform convergence on compact subsets of $\mathbb{R}^d \times \mathbb{R}^d$. Moreover, any subsequential limit is a continuous metric that induces the Euclidean topology. Although the uniqueness of the subsequential limit is still open, we expect that it should follow from similar arguments to~\cite{gm-uniqueness, dg-uniqueness}. We also expect that in the region $Q(\xi) = \widetilde Q(\xi) \in (0, \sqrt{2d})$, after appropriate normalization, $\{D_n(\cdot, \cdot)\}_{n \geq 1}$ converges to a limiting metric that is allowed to take on infinite values and has properties similar to the supercritical LQG metric in the two-dimensional case.

Next, we discuss the implication of Theorem~\ref{thm:WN-thick} for the exponential metric. Based on an intermediate result from the proof of Theorem~\ref{thm:WN-thick} (Proposition~\ref{prop:WN-thick-variant}), we establish the following. In the proof, instead of $\alpha$-thick points, we consider the percolation of $(-\alpha)$-thick points, which allows us to connect two sets with a path of extremely low $h_n$ values.

\begin{theorem}\label{thm:exponential-metric}
For any fixed $A>0$, there exist a constant $C>0$ and a function $L: \mathbb{N} \to (0, \infty)$ with $\lim_{x \to \infty} L(x) = 0$ such that for all $d \geq C$, $\xi \geq L(d)$, and any two disjoint compact sets $K_1, K_2 \subset \mathbb{R}^d$ with non-empty interiors:
$$
\mathbb{P}\big[D_n(K_1, K_2) \leq 2^{- A \xi n}\big] \geq 1 - o_n(1) \quad \mbox{as } n \rightarrow \infty,
$$
where the $o_n(1)$ term may depend on $A, d, \xi, K_1, K_2$. 
\end{theorem}

By definition, it always holds that $Q(\xi) \geq \widetilde Q(\xi)$. As explained above, we always have $Q(\xi) \geq 0$. The condition $Q(\xi) = \widetilde Q(\xi)$ suggests that the set-to-set distance is comparable to the distance between typical points (as is the case in $d=2$~\cite{dddf-lfpp, dg-supercritical-lfpp} and in the subcritical region $\xi \in (0,\xi_c)$ for $d \geq 3$~\cite{dgz-exponential-metric}), and thus the limiting metric could be well defined between typical points. However, if $Q(\xi) > \widetilde Q(\xi)$, then the set-to-set distance is much smaller than the distance between typical points, which in turn implies that the limiting metric could be well defined only between atypical points---a case that has not been investigated (see Remark~\ref{rmk:super-super} below for discussion).

As a consequence of Theorem~\ref{thm:exponential-metric} and~\eqref{eq:set-to-set}, it follows that in sufficiently high dimensions, for all sufficiently large $\xi$, if the set-to-set distance exponent $\widetilde Q(\xi)$ exists, then $\widetilde Q(\xi) < 0$ and thus we are in the regime  $Q(\xi) > \widetilde Q(\xi)$.  This also contrasts with the two-dimensional case where $\widetilde Q(\xi)$ is always positive and equals $Q(\xi)$~\cite{lfpp-pos, dg-supercritical-lfpp}. We also expect that similar arguments could show that $Q(\xi) = 0$ when both $d$ and $\xi$ are sufficiently large, although additional technical work is required. This partly answers Problem 7.2 of~\cite{dgz-exponential-metric} and suggests that a new phase may appear for the exponential metric.

\begin{remark}[Super-supercritical exponential metric]\label{rmk:super-super}
Theorem~\ref{thm:exponential-metric} suggests that there might be a new phase for the exponential metric in sufficiently high dimensions where $\widetilde Q(\xi) < 0$ and $Q(\xi) = 0$. In this phase, it is natural to normalize $D_n(\cdot, \cdot)$ in~\eqref{eq:def-exponential-metric} by the median of the set-to-set distance. (If we normalize by the median of the point-to-point distance, then the limiting metric would have distance 0 between any two disjoint boxes, and would thus always equal 0.) In this case, we conjecture that if a limit exists, it should satisfy the following properties:
\begin{itemize}
    \item The distances between any two disjoint compact subsets with non-empty interiors are positive and finite.
    
    \item For $\alpha>0$, similar to~\eqref{eq:lgf-thick}, we can define the set of $(-\alpha)$-thick points, which satisfy the condition that $\limsup_{n\rightarrow \infty} \frac{1}{n} h_n(x) \leq -\alpha \sqrt{\log 2}$. We expect that there exists a constant $\alpha>0$ depending on $\xi$ such that all points in the complement of the set of $(-\alpha)$-thick points are singular points, i.e., the only non-singular points are the $(-\alpha)$-thick points.
     
    \item Not every pair of non-singular points has finite distance between them, so the metric space is not connected, and is essentially an internal metric on the set of $(-\alpha)$-thick points.
\end{itemize}
\end{remark}

\subsection{Overview of the proof}

To prove Theorems~\ref{thm:lgf-thick}--\ref{thm:exponential-metric}, it suffices to find a continuous path in $\mathbb{R}^d$ where the value of the corresponding field is exceptionally high. We will use the first and second moment estimates to show that in sufficiently high dimensions, with high probability, such a path exists on $8^{-n} \mathbb{Z}^d$ (or $\frac{1}{\lfloor \sqrt{d} \rfloor} 8^{-n} \mathbb{Z}^d$) for each $n \geq 1$. By taking a subsequential limit, we then obtain the desired continuous path.

In Section~\ref{sec:brw}, we first study the branching random walk (BRW) and prove Theorem~\ref{thm:BRW-thick}. Due to the straightforward structure of the BRW, this proof is simpler than that of Theorem~\ref{thm:WN-thick}, but it already captures the main idea. We will construct a set of paths on the rescaled lattice $8^{-n} \mathbb{Z}^d$ such that two paths chosen uniformly from this set typically have few intersections. The construction is by first selecting a set of oriented paths on $\mathbb{Z}^d$ and then inductively refining it on lattices $8^{-j} \mathbb{Z}^d$ for each $j \geq 1$. (The number 8 is fairly arbitrary --- it is just a dyadic number chosen to be large enough so that the refinements of disjoint paths will be far away from each other.) A path is called good if all increments of the BRW (recall~\eqref{eq:def-brw}) are lower-bounded by $\alpha$ for all vertices on the path, and thus, all vertices on a good path satisfy $\mathcal{R}_m(x) \geq \alpha m $ for all $1 \leq m \leq n$. By comparing the first and second moments of the number of good paths, we will show that in sufficiently high dimensions, there is a high probability of having a good path. Then, by taking a subsequential limit of the good paths, we obtain a continuous path connecting to infinity, where every point is an $\alpha$-thick point.

In Section~\ref{sec:WN}, we study the white noise field and prove Theorems~\ref{thm:WN-thick} and~\ref{thm:exponential-metric}. The proof strategy is similar to that of Theorem~\ref{thm:BRW-thick} but it is more technically involved due to the correlations in the white noise field. There are two main differences: first, we construct the path on $\frac{1}{\lfloor \sqrt{d} \rfloor} 8^{-n} \mathbb{Z}^d$, where $\frac{1}{\lfloor \sqrt{d} \rfloor}$ is the correlation length of the white noise field in high dimensions (Lemma~\ref{lem:est-correlation}). Second, in addition to the condition that $h_m(x) \geq \alpha m$ for all $1 \leq m \leq n$, we will introduce an auxiliary condition to make the second moment estimate more tractable. We refer to the beginning of Section~\ref{sec:WN} and Section~\ref{subsec:good-WN} for a more detailed proof outline.

In Section~\ref{sec:approx}, we show that different approximations of the log-correlated Gaussian field yield the same thick point sets, and prove Theorems~\ref{thm:lgf-thick} and~\ref{thm:crossing}. Finally, in Section~\ref{sec:open}, we list several open problems.

\subsection{Basic notations}\label{subsec:notation}

\subsubsection*{Numbers}

We write $\mathbb{N} = \{1,2,\ldots \}$. Unless otherwise specified, all logarithms in this paper are taken with respect to the base $e$. For $a \in \mathbb{R}$, we use $\lfloor a \rfloor$ to denote the largest integer not greater than $a$.

Throughout the paper, let $d \geq 2$ be the dimension of the space and let the integers
\begin{equation}\label{eq:def-dimenion-constant}
\mathfrak q = \lfloor \frac{d}{10} \rfloor \quad \mbox{and} \quad \mathfrak r = \lfloor \sqrt d \rfloor.
\end{equation}
Constants such as $c, c', C, C'$ may change from place to place, while constants with subscripts like $c_1,C_1$ remain fixed throughout the article. The dependence of constants on additional variables including the dimension $d$ will be indicated at their first occurrence. 

\subsubsection*{Subsets of the space}

We consider the space $\mathbb{R}^d$ with $d \geq 2$. For $z \in \mathbb{R}^d$, we write $z = (z_1,\ldots,z_d)$ for its coordinates. For $1 \leq i \leq d$, let $\mathbf e_i$ be the $i$-th standard basis vector. We use the notation $|\cdot|_1$, $|\cdot|_2$, and $|\cdot|_\infty$ to represent the $l^1$-, $l^2$-, and $l^\infty$-norms, respectively. We denote the corresponding distances by $\mathfrak d_1$, $\mathfrak d_2$, and $\mathfrak d_\infty$. Note that $\mathfrak d_\infty \leq \mathfrak d_2 \leq \mathfrak d_1$. For a set $A \subset \mathbb{R}^d$ and $r>0$, let
\begin{equation*}
    B_r(A)= \{z \in \mathbb{R}^d: \mathfrak d_2(z,A) <r \}.
\end{equation*}
We use $\Vol(A)$ to denote the volume of $A$ with respect to the Lebesgue measure. 

For a subset $A \subset \mathbb{Z}^d$, we use $|A|$ to denote its cardinality. In this paper, we will consider paths on $\mathbb{Z}^d$ or on the rescaled lattice. Unless otherwise specified, all paths are assumed to be self-avoiding and nearest-neighbor. For a self-avoiding discrete path $P$, its length refers to the number of vertices it contains, i.e., $|P|$. We use $x \in P$ to represent a vertex in $P$ and use $e \in P$ to refer to an edge joining a pair of neighboring vertices in $P$. We may also use the notation $x \in e$ to indicate that the vertex $x$ is an endpoint of the edge $e$. For an edge $e$ in $\mathbb{Z}^d$, we use $e \cap P \neq \emptyset$ to indicate that $e$ intersects $P$, i.e., some endpoint of $e$ is a vertex in $P$. We use $\overline{P}$ to denote the continuous path obtained by connecting neighboring vertices in $P$. 

\medskip
\noindent\textbf{Acknowledgements.} 
We thank the anonymous referees for their careful reading and many helpful comments. J.D. is partially supported by NSFC Key Program Project No.12231002 and an Explorer Prize. E.G. was partially supported by NSF grant DMS-2245832. Z.Z.\ is partially supported by NSF grant DMS-1953848.

\section{Branching random walk}
\label{sec:brw}

In this section, we study the branching random walk and prove Theorem~\ref{thm:BRW-thick}. 

The proof strategy is as follows. We first construct a set of discrete paths on the rescaled lattice that go from $\{z : |z|_1 = 1 + \frac27 \}$ to far
away such that when the dimension is sufficiently high, two paths chosen uniformly from this set typically have few intersections. Specifically, for integers $n \geq 0$ and $M \geq 10$, we define $\mathcal{P}^{n,M}$ in Section~\ref{subsec:BRW-paths} consisting of self-avoiding paths on $8^{-n}\mathbb{Z}^d$ that connect $\{z : |z|_1 = 1 + \frac27\}$ to $\{z : |z|_1 = M - \frac27 \}$. For an integer $k \geq 0$, we call a path $P \in \mathcal{P}^{n,M}$ \textbf{$k$-good} if 
\begin{equation}\label{eq:brw-good}
a_{m, \mathcal{B}_m(x)} \geq \alpha \quad \mbox{for all $m \in [k, n] \cap \mathbb{Z}$ and $x \in P$}.
\end{equation}
Recall the definition of $\{a_{j,B}\}$ and $\mathcal{R}_n$ from Section~\ref{sec:brw}. This condition ensures that $\mathcal{R}_m(x) - \mathcal{R}_k(x) \geq \alpha (m-k)$ for all $m \in [k, n] \cap \mathbb{Z}$. Next, we count the number of good paths in $\mathcal{P}^{n,M}$. By comparing the first and second moments of the number of good paths, we prove the following proposition.

\begin{proposition}\label{prop:BRW-thick-point}
    For any fixed $\alpha>0$, the following holds for all sufficiently large $d$. There exists a function $\epsilon : \mathbb{N} \rightarrow (0,1)$ depending on $\alpha$ and $d$ with $\lim_{k \rightarrow \infty} \epsilon(k) = 0$ such that for all integers $n \geq k \geq 1$ and $M \geq 10$, with probability at least $1-\epsilon(k)$, there exists a $k$-good path in $\mathcal{P}^{n,M}$.
\end{proposition} 
\noindent Finally, by taking a subsequential limit of the good paths in $\mathcal{P}^{n,M}$, we obtain a continuous path connecting to infinity and all points on it are $\alpha$-thick points, thereby proving Theorem~\ref{thm:BRW-thick}.

The structure of this section is as follows. In Section~\ref{subsec:BRW-paths}, we construct the set of paths $\mathcal{P}^{n,M}$ and collect its properties. In Section~\ref{subsec:BRW-estimate}, we use first and second moment estimates to prove Proposition~\ref{prop:BRW-thick-point}. In Section~\ref{subsec:sec2-final-proof}, we prove Theorem~\ref{thm:BRW-thick}.

\subsection{Construction of paths}\label{subsec:BRW-paths}

In this subsection, for integers $n \geq 0$ and $M \geq 10$, we construct a set of self-avoiding paths $\mathcal{P}^{n,M}$ on $8^{-n} \mathbb{Z}^d$. Each path in $\mathcal{P}^{n,M}$ connects $\{z : |z|_1 = 1 + \frac27\}$ to $\{z : |z|_1 = M - \frac27\}$ (Lemma~\ref{lem:basic-property-mathcalP}), and two paths chosen uniformly from $\mathcal{P}^{n,M}$ will typically have few intersections (Lemma~\ref{lem:intersect-BRW}). We also collect some other properties of these paths that will be used in the proof of Proposition~\ref{prop:BRW-thick-point}.

The construction involves two steps: first we let $\mathcal{P}^{0,M}$ consist of oriented paths on $\mathbb{Z}^d$ connecting $\{z : |z|_1 = 1\}$ to $\{z : |z|_1 = M\}$, and then we inductively refine $\mathcal{P}^{0,M}$ on $8^{-n} \mathbb{Z}^d$ for each $n \geq 1$. Recall that for $1 \leq i \leq d$, $ \mathbf e_i$ is the $i$-th standard basis vector.

\begin{definition}[$\mathcal{P}^{0, M}$]\label{def:P0-brw}
    For an integer $M \geq 10$, let $\mathcal{P}^{0, M}$ consist of the paths of the form $( \mathbf e_{i_1},  \mathbf e_{i_1} +  \mathbf e_{i_2}, \ldots, \sum_{k=1}^m  \mathbf e_{i_k})$, where $(i_1, \ldots, i_M)$ is any sequence from $\{1, \ldots, d\}^M$. 
\end{definition}

Then, $\mathcal{P}^{0, M}$ consists of oriented paths on $\mathbb{Z}^d$ from $\{z : |z|_1 = 1\}$ to $\{z : |z|_1 = M\}$ with length $M$. The cardinality $|\mathcal{P}^{0, M}| = d^M$. The following lemma is deduced from the fact that random walk is transient in $d \geq 3$.

\begin{lemma}\label{lem:high-d-intersect}
    For any $d \geq 3$, there exists $c_1 = c_1(d) \in (0,1)$ such that for all $M \geq 10$ and $P, Q$ independently chosen uniformly from $\mathcal{P}^{0, M}$, the random variable $|P \cap Q|$ is stochastically dominated by a geometric random variable $G$ with success probability $c_1$, i.e., $\mathbb{P}[G \geq k] = c_1^k$ for all integers $k \geq 0$.
\end{lemma}

\begin{proof}
    Consider two independent oriented random walks $X_1$ and $X_2$ starting from 0 (we will slightly abuse the notation below by denoting $X_i$ also as the range of the walk). That is, for $i \in \{1,2\}$ and each $n \geq 0$, $X_i(n+1) - X_i(n)$ is independently chosen from $\{\mathbf e_1, \ldots, \mathbf e_d\}$. It suffices to show that $|X_1 \cap X_2 \setminus \{0\}|$ is stochastically dominated by a geometric random variable as $|P \cap Q|$ is stochastically dominated by $|X_1 \cap X_2 \setminus \{0\}|$.
    
    Since $|X_1(n)|_1 = |X_2(n)|_1 = n$, we have $|X_1 \cap X_2 \setminus \{0\}| = |\{n \geq 1 : X_1(n) = X_2(n)\}|$. Now we show that $\mathbb{P}[X_1(n) \neq X_2(n)\mbox{ for all $n\geq 1$}] > 0$. Note that $X_1(n) - X_2(n)$ can be viewed as a random walk where each step has the law $\mathbf e_i - \mathbf e_j$ with $i,j$ chosen uniformly from $\{1,\ldots, d\}^2$. This random walk is transient in $d \geq 3$, see e.g., \cite[Theorem 4.1.1]{lawler-limic-walks}. Therefore, $\mathbb{P}[X_1(n) \neq X_2(n)\mbox{ for all $n\geq 1$}] > 0$. Combined with the Markov property, this yields that $|X_1 \cap X_2 \setminus \{0\}|$ is stochastically dominated by a geometric random variable.
\end{proof}

Next, we inductively refine the paths in $\mathcal{P}^{0,M}$ on $ 8^{-n} \mathbb{Z}^d$ for each $n \geq 1$.
To this end, we need a geometric lemma to construct a collection of disjoint paths in a tube. A subset $T \subset \mathbb{Z}^d$ is called a \textbf{tube} if there exists $x \in 8 \mathbb{Z}^d$ and $y = x + 8 \epsilon \mathbf e_i$ for some $\epsilon \in \{ -1, + 1\}$ and $1 \leq i \leq d$ such that $T = \{z \in \mathbb{Z}^d: \mathfrak d_1(z, \ell) \leq 2\}$, where $\ell$ is the straight line connecting $x$ and $y$ (see Figure~\ref{fig:subpath-tube}). We will denote the tube as $T_{x,y}$ to emphasize its dependence on $x$ and $y$. In particular, $T_{x,y} = T_{y,x}$. The notation of a tube and the following lemma
naturally extend to any rescaled lattice.

\begin{figure}[h]
\centering
\includegraphics[scale=0.8]{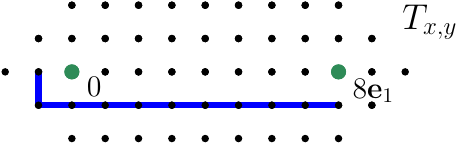}
\caption{A self-avoiding path of length 11 in a tube $T_{x,y}$ with $x = 0$ and $y = 8 \mathbf e_1$ (the two green points). Its starting point has $|\cdot|_1$-distance 1 to $x$, and its ending point has $|\cdot|_1$-distance 1 to $y$.}
\label{fig:subpath-tube}
\end{figure}

\begin{lemma}\label{lem:path-tube}
Recall that $\mathfrak q = \lfloor \frac{d}{10} \rfloor$ and consider a tube $T_{x,y}$ on $\mathbb{Z}^d$. For any vertex $z \in T_{x,y}$ with $|x-z|_1 = 1$, there exists a set of $\mathfrak q$ self-avoiding paths $\mathcal{T}_{x,y}^z$ in $T_{x,y}$ such that each path has length 11, starts from $z$, and ends at a vertex with $|\cdot|_1$-distance 1 to $y$; see Figure~\ref{fig:subpath-tube}. Moreover, different paths in $\mathcal{T}_{x,y}^z$ intersect only at their starting point $z$. 

If we do not fix the starting point $z$, there exists a set of $\mathfrak q$ disjoint self-avoiding paths $\mathcal{T}_{x,y}$ such that each path has length 11, starts at a vertex with $|\cdot|_1$-distance 1 to $x$, and ends at a vertex with $|\cdot|_1$-distance 1 to $y$.

We further require that 
\begin{equation}\label{eq:path-condition-brw}
\begin{aligned}
&\mbox{all paths in $\mathcal{T}_{x,y}$ (resp.\ $\mathcal{T}_{x,y}^z$) start in $B_x$ (resp.\ $B_x \cup B_z$),}\\
&\mbox{stay in $B_x \cup B_y$ (resp.\ $B_x \cup B_y \cup B_z$), and end in $B_y$},
\end{aligned}
\end{equation}
where $B_\bullet := \bullet + [0,8)^d$ for $\bullet \in \{x,y,z\}$. 
\end{lemma}

\begin{proof}
    Without loss of generality, assume that $x = 0$. First, we construct $\mathcal{T}_{x,y}$. By the symmetry of coordinates, it suffices to consider the following two cases:
    \begin{enumerate}
        \item $y = 8 \mathbf e_1$: We take $\mathcal{T}_{x,y} = \{ ( \mathbf e_{2t}, \mathbf e_1 + \mathbf e_{2t}, 2\mathbf e_1 + \mathbf e_{2t}, \ldots, 8\mathbf e_1 + \mathbf e_{2t}, 8\mathbf e_1 + \mathbf e_{2t} + \mathbf e_{2t+1}, 8\mathbf e_1 + \mathbf e_{2t+1} ): 1 \leq t \leq \mathfrak q \}$.

        \item $y = -8  \mathbf e_1$: We take $\mathcal{T}_{x,y} = \{ (  \mathbf e_{2t},  -  \mathbf e_1 +  \mathbf e_{2t}, - 2 \mathbf e_1 +  \mathbf e_{2t}, \ldots, - 8 \mathbf e_1 +  \mathbf e_{2t}, - 8 \mathbf e_1 +  \mathbf e_{2t} +  \mathbf e_{2t+1}, - 8 \mathbf e_1 +  \mathbf e_{2t+1}): 1 \leq t \leq \mathfrak q \}$.
    \end{enumerate}
    We can check that $\mathcal{T}_{x,y}$ satisfies~\eqref{eq:path-condition-brw}. 

    Next, we construct $\mathcal{T}_{x,y}^z$. By definition, we have $z = \pm  \mathbf e_i$ for some $1 \leq i \leq d$. Without loss of generality, assume that $i=1$. When $z =  \mathbf e_1$ (in particular, $B_z = B_x$), there are four cases for $y$:
    \begin{enumerate}
            \item $y = 8  \mathbf e_1$: Take $\mathcal{T}_{x,y}^z = \{ (  \mathbf e_1,  \mathbf e_1 +  \mathbf e_{2t}, 2 \mathbf e_1 +  \mathbf e_{2t}, \ldots, 8 \mathbf e_1 +  \mathbf e_{2t}, 8 \mathbf e_1 +  \mathbf e_{2t} +  \mathbf e_{2t+1}, 8 \mathbf e_1 +  \mathbf e_{2t+1}): 1 \leq t \leq \mathfrak q \}$.
            \item $y = -8  \mathbf e_1$: Take $\mathcal{T}_{x,y}^z = \{ (  \mathbf e_1,  \mathbf e_1 +  \mathbf e_t,  \mathbf e_t, - \mathbf e_1 +  \mathbf e_t, \ldots, -8 \mathbf e_1 +  \mathbf e_t): 2 \leq t \leq \mathfrak q + 1 \}$.
            
            \item $y = 8  \mathbf e_j$ with $j \neq 1$: Without loss of generality, assume that $j = 2$. We take $\mathcal{T}_{x,y}^z = \{ (  \mathbf e_1,  \mathbf e_1 +  \mathbf e_t,  \mathbf e_1 +  \mathbf e_2 +  \mathbf e_t,  \mathbf e_1 + 2 \mathbf e_2 +  \mathbf e_t, \ldots,  \mathbf e_1 + 8 \mathbf e_2 +  \mathbf e_t,  8 \mathbf e_2 +  \mathbf e_t ): 3 \leq t \leq \mathfrak q + 2 \}$.

            \item $y = -8  \mathbf e_j$ with $j \neq 1$: Without loss of generality, assume that $j = 2$. We take $\mathcal{T}_{x,y}^z = \{ (  \mathbf e_1,  \mathbf e_1 +  \mathbf e_t,  \mathbf e_1 -  \mathbf e_2 +  \mathbf e_t,  \mathbf e_1 - 2 \mathbf e_2 +  \mathbf e_t, \ldots,  \mathbf e_1 - 8 \mathbf e_2 +  \mathbf e_t,  - 8 \mathbf e_2 +  \mathbf e_t ): 3 \leq t \leq \mathfrak q + 2 \}$.
    
        \end{enumerate}
    When $z = -  \mathbf e_1$ (in particular, $B_z \neq B_x$), we again consider the four cases for $y$:
    \begin{enumerate}
            \item $y = 8  \mathbf e_1$: Take $\mathcal{T}_{x,y}^z = \{ ( -  \mathbf e_1, -  \mathbf e_1 +  \mathbf e_t,  \mathbf e_t,  \mathbf e_1 +  \mathbf e_t, \ldots, 8 \mathbf e_1 +  \mathbf e_t): 2 \leq t \leq \mathfrak q + 1\}$.
            \item $y = -8  \mathbf e_1$: $\mathcal{T}_{x,y}^z = \{ ( -  \mathbf e_1, -  \mathbf e_1 +  \mathbf e_{2t}, - 2 \mathbf e_1 +  \mathbf e_{2t}, \ldots, - 8 \mathbf e_1 +  \mathbf e_{2t}, - 8 \mathbf e_1 +  \mathbf e_{2t} +  \mathbf e_{2t+1}, - 8 \mathbf e_1 +  \mathbf e_{2t+1}): 1 \leq t \leq \mathfrak q \}$.
            
            \item $y = 8  \mathbf e_j$ with $j \neq 1$: Without loss of generality, assume that $j = 2$. We take $\mathcal{T}_{x,y}^z = \{ ( -  \mathbf e_1, -  \mathbf e_1 +  \mathbf e_t,  \mathbf e_t,  \mathbf e_2 +  \mathbf e_t, 2  \mathbf e_2 +  \mathbf e_t, \ldots, 8 \mathbf e_2 +  \mathbf e_t ): 3 \leq t \leq \mathfrak q + 2 \}$.

            \item $y = -8  \mathbf e_j$ with $j \neq 1$: Without loss of generality, assume that $j = 2$. We take $\mathcal{T}_{x,y}^z = \{ ( -  \mathbf e_1, -  \mathbf e_1 +  \mathbf e_t,  \mathbf e_t, - \mathbf e_2 +  \mathbf e_t, - 2 \mathbf e_2 +  \mathbf e_t, \ldots, - 8 \mathbf e_2 +  \mathbf e_t ): 3 \leq t \leq \mathfrak q + 2 \}$.
    
    \end{enumerate}
    We can check that all these $\mathcal{T}_{x,y}^z$ satisfy~\eqref{eq:path-condition-brw}. 
\end{proof}

Now we give the definition of the refinements of a self-avoiding path on $ 8^{-j} \mathbb{Z}^d$. We collect some basic properties in Lemma~\ref{lem:property-refine}.

\begin{definition}[Refinements]\label{def:refine}
    Let $j \geq 0$ be an integer. For two neighboring vertices $x$ and $y$ in $ 8^{-j} \mathbb{Z}^d$, let $T_{x,y}$ be the tube defined with respect to them on $  8^{-j-1} \mathbb{Z}^d$. The refinements of the edge $(x,y)$ are self-avoiding paths in $T_{x,y}$ with length $11$ whose starting and ending points have $|\cdot|_1$-distance $ 8^{-j-1}$ to $x$ and $y$, respectively; see Figure~\ref{fig:subpath-tube}.
    For a self-avoiding path $P = (x_1, \ldots, x_L)$ on $ 8^{-j} \mathbb{Z}^d$, a path on $ 8^{-j-1} \mathbb{Z}^d$ is called its refinement if 
    \begin{enumerate}
        \item It consists chronologically of the refinements of $(x_i,x_{i+1})$ for $1 \leq i \leq L - 1$;
        \item For all $1 \leq i \leq L-2$, the end point of the refinement of $(x_i, x_{i+1})$ coincides with the start point of the refinement $(x_{i+1}, x_{i+2})$, and moreover, these two refinements do not intersect at other vertices.
    \end{enumerate}
    See Figure~\ref{fig:refinement-tube} for an illustration.

\begin{figure}[h]
\centering
\includegraphics[scale=0.8]{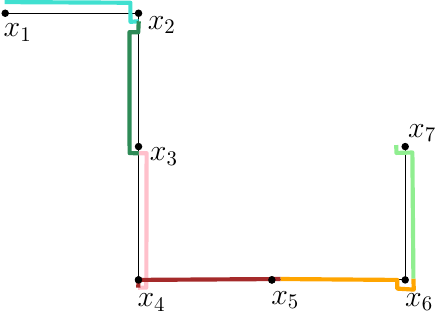}
\caption{The colored paths are the refinements of $(x_i, x_{i+1})$ for $1 \leq i \leq 6$. In this figure, the refinements of neighboring edges intersect only at their endpoints and do not intersect at other vertices (though they may appear to intersect at other vertices due to the two-dimensional illustration).}
\label{fig:refinement-tube}
\end{figure}
\end{definition}

The following lemma is a direct consequence of the definition. Recall that for a discrete path $P$ on $ 8^{-j} \mathbb{Z}^d$, $\overline P$ is the continuous path obtained by connecting the neighboring vertices in $P$ with straight lines.

\begin{lemma}\label{lem:property-refine}
    The following hold for all $j \geq 0$:
    \begin{enumerate}
        \item Each refinement of a self-avoiding path on $ 8^{-j} \mathbb{Z}^d$ is a self-avoiding path on $ 8^{-j-1} \mathbb{Z}^d$. \label{lem3.8-claim1}
        
        \item If $P$ has length $L$, then its refinement has length $10(L-1)+1$. \label{lem3.8-claim2}

        \item If a path $P$ on $ 8^{-j-1} \mathbb{Z}^d$ is a refinement of a path $Q$ on $ 8^{-j} \mathbb{Z}^d$, then $Q$ is unique and given by concatenating all the vertices in $ 8^{-j}\mathbb{Z}^d$ that are $8^{-j-1}$ close to $P$ in terms of $|\cdot|_1$-distance. \label{lem3.8-claim3}

        \item If $P$ is a refinement of a path $Q$ on $ 8^{-j}\mathbb{Z}^d$, then we have $\overline P \subset \{z : \mathfrak d_1(z, \overline Q) \leq 2 \cdot 8^{-j-1}\} $ and $\overline Q \subset \{z : \mathfrak d_1(z, \overline P) \leq 2 \cdot 8^{-j-1}\} $.\label{lem3.8-claim4}
    \end{enumerate} 
\end{lemma}

\begin{remark}
    As mentioned in Section~\ref{subsec:notation}, we will only consider self-avoiding paths in this paper unless otherwise specified. One important fact we will frequently use is that a self-avoiding path can cross each vertex at most once, which allows us to obtain many bounds independent of $d$.
\end{remark}

Next, we inductively refine the paths in $\mathcal{P}^{0,M}$. For each $j \geq 0$, the set $\mathcal{P}^{j + 1, M}$ will consist of self-avoiding paths on $8^{-j - 1} \mathbb{Z}^d$ that are some refinements of the paths in $\mathcal{P}^{j, M}$. By Lemma~\ref{lem:property-refine}, all of the paths in $\mathcal{P}^{j, M}$ have the same length $\mathcal L_j$ defined as follows:
\begin{equation}\label{eq:def-mathcal-l}
\mathcal L_0 =  M \quad \mbox{and} \quad \mathcal L_{i+1} = 10(\mathcal L_i - 1) + 1 \quad \mbox{for all } i \geq 0.
\end{equation}

\begin{definition}[$\mathcal{P}^{n, M}$]\label{def:Pn-brw}
    Let an integer $M \geq 10$. We assume that $d$ is sufficiently large such that $\mathfrak q \geq 100$. Recall $\mathcal{P}^{0, M}$ from Definition~\ref{def:P0-brw}. Suppose that $\mathcal{P}^{j, M}$ has been defined for some integer $j \geq 0$. For each path $P = (x_1, \ldots, x_{\mathcal L_j}) \in \mathcal{P}^{j, M}$ on $8^{-j} \mathbb{Z}^d$, the set $\mathcal{P}^{j + 1,M}$ contains the following $(\mathfrak q - 10)^{\mathcal L_j - 1}$ refinements of $P$:
    \begin{enumerate}
    \item Recall $\mathcal{T}_{x_1,x_2}$ from Lemma~\ref{lem:path-tube}. We deterministically choose a subset of cardinality $(\mathfrak q - 10)$ from $\mathcal{T}_{x_1,x_2}$ and then choose a refinement of $(x_1,x_2)$ from this subset;

    \item Suppose the refinements of $(x_1,x_2),(x_2,x_3),\ldots, (x_{i-1}, x_i)$ have been defined for $i < \mathcal L_j $. Let $z$ be the last point of the refinement of $(x_{i-1}, x_i)$. Since at most 10 paths in $\mathcal{T}_{x_i, x_{i+1}}^z$ intersect the refinement of $(x_{i-1}, x_i)$ at a point other than $z$,\footnote{This is due to the fact that $\mathcal{T}_{x_i, x_{i+1}}^z$ contains self-avoiding paths in $T_{x_i, x_{i+1}}$ intersecting only at their starting point $z$ (Lemma~\ref{lem:path-tube}). Furthermore, a fixed refinement of $(x_{i-1}, x_i)$ occupies at most 10 vertices in $T_{x_i, x_{i+1}}$ other than $z$.} we can deterministically choose a subset of $\mathcal{T}_{x_i, x_{i+1}}^z$ with cardinality $(\mathfrak q - 10)$ such that these paths do not intersect the refinement of $(x_{i-1}, x_i)$ except at $z$, and then choose a refinement of $(x_i,x_{i + 1})$ from this subset;
    
    \item By concatenating the refinements of $(x_i, x_{i+1})$ for $1 \leq i \leq \mathcal L_j - 1$, we obtain $(\mathfrak q - 10)^{\mathcal L_j - 1}$ refinements of $P$ on $ 8^{-j - 1} \mathbb{Z}^d$, each with length $\mathcal L_{j+1}$.

\end{enumerate}
By induction, this defines $\mathcal{P}^{n,M}$ for all $n \geq 0$. In addition, for different $M_1 > M_2 \geq 10$, if two paths $P \in  \mathcal{P}^{j,M_1}$ and $Q \in \mathcal{P}^{j,M_2}$ are the same up to length $|Q| = \mathcal L_j(M_2)$ (the number in~\eqref{eq:def-mathcal-l} starting from $\mathcal L_0 = M_2$), we require that their refinements in $\mathcal{P}^{j+1,M_1}$ and $\mathcal{P}^{j+1,M_2}$ are also the same up to length $\mathcal L_{j+1}(M_2)$.
\end{definition}

\begin{remark}\label{rmk:refine}
    In Section~\ref{sec:WN}, we will consider the same definition of refinements on the rescaled lattice $\frac{1}{\lfloor \sqrt{d} \rfloor} 8^{-j} \mathbb{Z}^d$ for $j \geq 0$. The choice of refinements will be the same except that we start with a set of paths on $\frac{1}{\lfloor \sqrt{d} \rfloor} \mathbb{Z}^d$. All the results in Section~\ref{subsec:BRW-paths} can be extended to that case, which we will recall in Section~\ref{subsec:path-WN}.
\end{remark}

By Definition~\ref{def:Pn-brw}, we have $|\mathcal{P}^{j + 1,M}| = (\mathfrak q - 10)^{\mathcal L_j - 1} \times |\mathcal{P}^{j,M}|$. For each path $P \in \mathcal{P}^{n,M}$, by Claim~\ref{lem3.8-claim3} in Lemma~\ref{lem:property-refine}, there exists a unique sequence of paths $P_j \in \mathcal{P}^{j,M}$ for $0 \leq j \leq n$ such that $P_n = P$ and $P_{j+1}$ is a refinement of $P_j$ for all $0 \leq j \leq n-1$. We will often use $P_j$ to denote this sequence of paths, and call $P_j$ the corresponding path of $P$ in $\mathcal{P}^{j,M}$. 

The set of paths $\mathcal{P}^{n,M}$ also satisfies a restriction property with respect to $M$.

\begin{lemma}\label{lem:restrict-M-brw}
    For all $M_1 \geq M_2 \geq 10$ and $n \geq 0$, if we restrict the paths in $\mathcal{P}^{n,M_1}$ before length $\mathcal L_n(M_2)$, then we obtain $\mathcal{P}^{n,M_2}$.
\end{lemma}

\begin{proof}
    The case $n = 0$ follows from Definition~\ref{def:P0-brw}. The case $n \geq 1$ follows from the case $n = 0$ and Definition~\ref{def:Pn-brw} with the following induction argument. Suppose that $\mathcal{P}^{j,M_1}$ and $\mathcal{P}^{j,M_2}$ have been defined and satisfy the restriction property. Recall from Definition~\ref{def:Pn-brw} that when defining $\mathcal{P}^{j+1,M_1}$ and $\mathcal{P}^{j+1,M_2}$, we required that for two paths $P \in  \mathcal{P}^{j,M_1}$ and $Q \in \mathcal{P}^{j,M_2}$ that are the same up to length $|Q| = \mathcal L_j(M_2)$, their refinements in $\mathcal{P}^{j+1,M_1}$ and $\mathcal{P}^{j+1,M_2}$ are also the same up to length $\mathcal L_{j+1}(M_2)$. Therefore, if we restrict the paths in $\mathcal{P}^{j+1,M_1}$ before length $\mathcal L_{j+1}(M_2)$, we obtain $\mathcal{P}^{j+1,M_2}$. By induction, this proves the lemma.
\end{proof}

Since each path in $\mathcal{P}^{j,M}$ has the same number of refinements in $\mathcal{P}^{j+1,M}$, we see that choosing a path uniformly from $\mathcal{P}^{n,M}$ is equivalent to first choosing a path uniformly from $\mathcal{P}^{0,M}$, and then uniformly choosing one of its refinements as defined in Definition~\ref{def:Pn-brw} at each scale. We use this fact to prove the following lemma, which implies that when $d$ is sufficiently large, two paths chosen uniformly from $\mathcal{P}^{n,M}$ have few intersections.

\begin{lemma}\label{lem:intersect-BRW}
    The following holds for all $n \geq 0$ and $M \geq 10$. For two paths $P, Q$ chosen uniformly from $\mathcal{P}^{n, M}$ and $0 \leq j \leq n$, let $P_j$ and $Q_j$ be their corresponding paths in $\mathcal{P}^{j,M}$ and let $Y_j = |P_j \cap Q_j|$. Then for all $1 \leq i \leq n$, given $(Y_0, Y_1, \ldots Y_{i-1})$, we have
    \begin{equation*}
    Y_i\mbox{ is stochastically dominated by } 11(Z_1 + Z_2 + \ldots + Z_{2 Y_{i-1}}),\end{equation*}
    where $(Z_l)_{l \geq 1}$ are i.i.d.\ Bernoulli random variables with probability $\frac{50}{\mathfrak q - 10}$ of being 1. 
\end{lemma}

\begin{proof}
    Since $(P,Q)$ is uniformly chosen from $\mathcal{P}^{n, M} \times \mathcal{P}^{n, M}$, given $(P_0, P_1, \ldots, P_{i-1})$ and $(Q_0, Q_1, \ldots, Q_{i-1})$, the paths $P_i$ and $Q_i$ are independently chosen from the refinements of $P_{i-1}$ and $Q_{i-1}$ as defined in Definition~\ref{def:Pn-brw}, respectively. Let $P_{i-1} = (x_1, x_2, \ldots, x_{\mathcal L_{i-1}})$ and suppose $Q_i$ has been chosen. For each $1 \leq k \leq \mathcal L_{i-1} - 1$ such that $(x_k, x_{k+1})$ does not intersect $Q_{i-1}$, any refinement of $(x_k, x_{k+1})$ will not intersect $Q_i$. Let $\{ k_l \}_{1 \leq l \leq m}$ be those $k$ such that $(x_k, x_{k+1})$ intersects $Q_{i-1}$, arranged in an increasing order. Since each edge $(x_{k_l}, x_{k_l +1})$ contains an endpoint in $P_{i-1} \cap Q_{i-1}$ and each vertex is crossed at most once by $P_{i-1}$, we know that $m \leq 2 |P_{i-1} \cap Q_{i-1}| = 2Y_{i-1}$.
    For $1 \leq l \leq m$, let $Z_l'$ be the indicator function that the refinement of $(x_{k_l}, x_{k_l + 1})$ in $P_i$ intersects $Q_i$ at a vertex other than the starting point. When $k_l=1$ (in particular, $l = 1$), we instead let $Z_l'$ be the indicator function that the refinement of $(x_{k_l}, x_{k_l + 1}) = (x_1, x_2)$ in $P_i$ intersects $Q_i$. Then, we have $Y_i = |P_i \cap Q_i| \leq 11 \sum_{l=1}^{m} Z_l'$. 
    
    Next, we show that for each $1 \leq l \leq m$, conditioned on the refinements of $\{(x_p, x_{p+1})\}_{1 \leq p \leq k_l-1}$ in $P_i$, \begin{equation}\label{eq:lem2.10-claim}
        \begin{aligned}
            &\mbox{the refinement of $(x_{k_l}, x_{k_l+1})$ has probability at least $(1 - \frac{50}{\mathfrak q - 10})$ of not intersecting $Q_i$}\\
            &\mbox{except at the starting point, i.e., $\mathbb{P}[Z_l' = 0|Y_0,Y_1,\ldots, Y_{i-1}, Z_1', Z_2',\ldots, Z_{l-1}'] \geq 1 - \frac{50}{\mathfrak q - 10}$}.
        \end{aligned}
    \end{equation}When $k_l = 1$ (in particular, $l=1$), this property follows from the fact that the paths in $\mathcal{T}_{x_1, x_2}$ are disjoint (Lemma~\ref{lem:path-tube}). Now we consider the case $k_l \geq 2$. Let $T_{x_{k_l}, x_{k_l+1}}$ be the tube defined with respect to $x_{k_l}$ and $x_{k_l+1}$ on $8^{-i-1} \mathbb{Z}^d$, and let $z$ be the end point of the refinement of $(x_{k_l - 1}, x_{k_l})$ in $P_i$, which is also the starting point of the refinement of $(x_{k_l}, x_{k_l+1})$. We claim that $Q_i$ occupies at most $4 \times 11 \leq 50$ vertices in $T_{x_{k_l}, x_{k_l+1}}$. This is because these vertices must belong to the refinements of the edges in $Q_{i-1}$ that intersect $(x_{k_l}, x_{k_l+1})$, and moreover, there are at most $4$ such edges and each edge is refined into 11 vertices. Combining this claim with the fact that the paths in $\mathcal{T}_{x_i, x_{i+1}}^z$ are disjoint except at $z$ (Lemma~\ref{lem:path-tube}), we see that at most $50$ paths in $\mathcal{T}_{x_{k_l}, x_{k_l+1}}^z$ intersect $Q_i$ at a point other than $z$, which proves~\eqref{eq:lem2.10-claim} (recall from Definition~\ref{def:Pn-brw} that the refinement of $(x_{k_l}, x_{k_l+1})$ is chosen from a subset of $\mathsf q -10$ subpaths in $\mathcal{T}_{x_{k_l}, x_{k_l+1}}^z$). Therefore, conditional on $(Y_0,Y_1,\ldots, Y_{i-1})$, the sequence $(Z_1', \ldots, Z_m')$ is stochastically dominated by a sequence of i.i.d.\ Bernoulli random variables with probability $\frac{50}{\mathfrak q - 10}$ of being 1. Combined with the inequalities $m \leq 2Y_{i-1}$ and $Y_i \leq 11 \sum_{l=1}^{m} Z_l'$, this yields the lemma. \qedhere
\end{proof}

The following lemma presents some basic properties of refinements. We refer to the refinement of a refinement as a refinement at scale $2$, etc.
\begin{lemma}\label{lem:basic-property-mathcalP}
    The following hold for all $j \geq 0$ and $d \geq 2$:
    \begin{enumerate}

        \item \label{property-1-cal} For an edge $(x,y)$ in $  8^{-j} \mathbb{Z}^d$, let $P$ be a refinement of $(x,y)$ at any scale. Then we have $$\overline P \subset \{ z : \mathfrak d_1(z, \overline{(x,y)} )  \leq \frac{2}{7} 8^{-j} \} \quad \mbox{and} \quad \overline{(x,y)} \subset \{ z : \mathfrak d_1(z, \overline{P}) \leq \frac{2}{7} 8^{-j} \}.$$

        \item \label{property-0-cal} All of the paths in $\mathcal{P}^{j, M}$ connect $\{z : |z|_1 = 1 + \frac{2}{7} \}$ and $\{z : |z|_1 = M - \frac{2}{7} \} $.

        \item \label{property-2-cal} For all paths $P \in \mathcal{P}^{j,M}$ and all its refinements at any scale $Q$ in $\cup_{k \geq j} \mathcal{P}^{k,M} $, we have $Q \subset P + [0,8^{-j})^d$. Moreover, for all vertices $x \in P$, we have $Q \cap (x + [0,8^{-j})^d) \neq \emptyset$.
 
    \end{enumerate}
    
\end{lemma}

\begin{proof}
    Claim~\ref{property-1-cal} directly follows from Claim~\ref{lem3.8-claim4} in Lemma~\ref{lem:property-refine} and the fact that $\sum_{i=j}^\infty 2 \cdot 8^{-i-1} = \frac{2}{7} 8^{-j}$. Claim~\ref{property-0-cal} follows from Claim~\ref{property-1-scr} and the fact that the paths in $\mathcal{P}^{0,M}$ connect $\{z:|z|_1 = 1\}$ and $\{z : |z|_1 = M\}$. 
    
    Next, we prove Claim~\ref{property-2-cal}. The fact that $Q \subset P + [0,8^{-j})^d$ follows from the requirement~\eqref{eq:path-condition-brw} in Lemma~\ref{lem:path-tube}. In fact, by the second condition of~\eqref{eq:path-condition-brw}, for all paths $P \in \mathcal{P}^{j,M}$ and all of its refinements $P' \in \mathcal{P}^{j+1, M}$, we have $P' \subset P + [0,8^{-j})^d$. Since $P'$ is a path on $8^{-j-1}\mathbb{Z}^d$, it follows that $P'+[0, 8^{-j-1})^d \subset P + [0,8^{-j})^d$. By induction, we then conclude that for any $k > j$ and any $Q \in \mathcal{P}^{k,M}$ which is a refinement of $P$, we have $Q + [0, 8^{-k})^d \subset P + [0,8^{-j})^d$ and thus $Q \subset P + [0,8^{-j})^d$.
    
    Now, we show that $Q \cap (x + [0,8^{-j})^d) \neq \emptyset$ for all $x \in P$. Let $P = (x_1, x_2, \ldots, x_{\mathcal{L}_j})$. From the first and third conditions in~\eqref{eq:path-condition-brw} and an induction argument, we can derive that $Q \cap (x_1 + [0,8^{-j})^d) \neq \emptyset$ and $Q \cap (x_{\mathcal{L}_j} + [0,8^{-j})^d) \neq \emptyset$, respectively. For $2 \leq i \leq \mathcal{L}_j-1$, the fact that $Q \cap (x_i + [0,8^{-j})^d) \neq \emptyset$ follows from the fact that $Q$ is a nearest-neighbor path which cannot move directly from $x_{i-1} + [0,8^{-j})^d$ to $x_{i+1} + [0,8^{-j})^d$.
\end{proof}

In the proof of Theorem~\ref{thm:BRW-thick}, we will take the subsequential limit of the paths in $\mathcal{P}^{n,n}$ with respect to the local Hausdorff distance, defined as follows:
\begin{equation}\label{eq:loc-hausdorff}
d_{H, {\rm loc}}(A,B) = \sum_{n=1}^\infty \frac{1}{2^n} {\rm min}\{1, d_H(A \cap B_{2^n}(0), B \cap B_{2^n} (0)) \} \quad \mbox{for two sets $A,B \subset \mathbb{R}^d$},
\end{equation}
where $d_H$ is the Euclidean Hausdorff distance. When defining the Hausdorff distance, we view the discrete paths as sets of points in $\mathbb{R}^d$. Next, we show that the subsequential limit $\widetilde P$ is always a continuous path in $\mathbb{R}^d$, i.e., there exists a continuous map $\phi:[0,\infty) \to \mathbb{R}^d$ such that $\widetilde P = \phi([0,\infty))$. Recall the definition of $\mathcal{B}_j$ before~\eqref{eq:def-brw}.

\begin{lemma}\label{lem:Hausdorff-limit-brw}
     Let $\{n_k\}_{k \geq 1}$ be an increasing sequence tending to infinity, and let $\{P_k\}_{k \geq 1}$ be a sequence of paths in $\mathcal{P}^{n_k,n_k}$ that converge to a set $\widetilde P$ with respect to the local Hausdorff distance. Then, $\widetilde P$ is a continuous path in $\mathbb{R}^d$ that connects $\{z:|z|_1 = 2\}$ to infinity. Moreover, for each $x \in \widetilde P$ and $j \geq 1$, there exists a sequence of points $x_k \in P_k$ such that $\mathcal{B}_j(x) =\mathcal{B}_j(x_k)$ for all sufficiently large $k$.
\end{lemma}

\begin{proof}
    The key observation is that for any fixed $j,M \geq 10$, 
    \begin{equation}\label{eq:lem2.11-assert}
        \mbox{the corresponding paths of $P_k$ in $\mathcal{P}^{j,M}$ are the same for all sufficiently large $k$.}
    \end{equation}
    This is due to the fact that for two different paths $P$ and $Q$ in $\mathcal{P}^{j,M}$, Claim~\ref{property-1-cal} from Lemma~\ref{lem:basic-property-mathcalP} ensures their refinements at any scale stay in $\{z : \mathfrak d_1(z, \overline{P}) \leq \frac{2}{7} 8^{-j} \}$ and $\{z : \mathfrak d_1(z, \overline{Q}) \leq \frac{2}{7} 8^{-j} \}$, respectively. Moreover, the $|\cdot|_1$-Hausdorff instance between $\overline{P}$ and $\overline{Q}$ is at least $8^{-j}$ (under the assumption that $P,Q$ are different), and thus, all their refinements at any scale are uniformly bounded away from each other in terms of the $|\cdot|_1$-Hausdorff distance, which implies~\eqref{eq:lem2.11-assert} (via proof of contradiction).

    We parametrize all the paths in $\mathcal{P}^{n,M}$ by multiplying their length by $10^{-n}$. For instance, the paths in $\mathcal{P}^{0,M}$ are parametrized by times $1,2,\ldots, M$. By this definition, when we refine a path $P = (x_1 ,\ldots, x_{\mathcal L_n})$ on $8^{-n-1} \mathbb{Z}^d$, the refinement of the edge $(x_i, x_{i+1})$ takes the times $t_i, t_i + 10^{-n-1}, t_i + 2 \cdot 10^{-n-1}, \ldots, t_i + 10 \cdot 10^{-n-1}$ where $t_i = 10^{-n} i$ is the time for $x_i$ in $P$. We claim that under this parametrization, $\{P_k\}_{k \geq 1}$ converges under the local uniform topology. This is because by~\eqref{eq:lem2.11-assert}, for any fixed $j,M \geq 10$, the paths in $\{P_k\}_{k \geq 1}$ take approximately the same places at the times $10^{-j} \mathbb{Z} \cap [1,M]$, and the neighboring subpaths are close to these points by Claim~\ref{property-1-cal} in Lemma~\ref{lem:basic-property-mathcalP}. Therefore, $\{P_k\}_{k \geq 1}$ converges under the local uniform topology, and thus, $\widetilde P$ is a continuous path. By Claim~\ref{property-0-cal} in Lemma~\ref{lem:basic-property-mathcalP}, we see that $\widetilde P$ connects $\{z:|z|_1 = 2\}$ to infinity.

    The last property in Lemma~\ref{lem:Hausdorff-limit-brw} follows from~\eqref{eq:lem2.11-assert} and Claims~\ref{property-1-cal} and~\ref{property-2-cal} in Lemma~\ref{lem:basic-property-mathcalP}, as we now elaborate. Fix $x \in \widetilde P$ and $j \geq 1$. Let $l = \lfloor j/3 \rfloor +1$ and let $\widetilde P_l$ be the path in~\eqref{eq:lem2.11-assert} with $l$ in place of $j$ and $M$ being a sufficiently large number satisfying $M \geq |x|_1 + 100$. Then by Claims~\ref{property-1-cal} and~\ref{property-2-cal} in Lemma~\ref{lem:basic-property-mathcalP}, for all sufficiently large $k$, the part of the path $P_k$ before reaching $\{z : |z|_1 = |x|_1 + 10\}$ is contained in ${\widetilde P}_l + [0, 8^{-l})^d$, and moreover, it is also contained $\{z : \mathfrak d_1(z, \overline{{\widetilde P}_l}) \leq \frac{2}{7} 8^{-l}\}$. Therefore, we have $x \in {\widetilde P}_l + [0, 8^{-l})^d $, and hence, there exists $y \in {\widetilde P}_l$ such that $x \in (y + [0,8^{-l})^d)$. By Claim~\ref{property-2-cal} in Lemma~\ref{lem:basic-property-mathcalP}, we see that $P_k \cap (y + [0,8^{-l})^d) \neq \emptyset$ for all sufficiently large $k$, which implies the last property, as for points in $y + [0,8^{-l})^d$, their corresponding $\mathcal{B}_j$-boxes are the same.\qedhere
\end{proof}

\subsection{First and second moment estimates}\label{subsec:BRW-estimate}

Fix $\alpha>0$. Recall from~\eqref{eq:brw-good} that for an integer $k \geq 0$, a path in $\mathcal{P}^{n, M}$ is called $k$-good if $a_{m, \mathcal{B}_m(x)} \geq \alpha $ for all $m \in [k,n] \cap \mathbb{Z}$ and vertices $x$ in the path. Consider the weighted sum
$$
\mathcal{N} = \mathcal{N}(\alpha, k, n, M) := \sum_{P \in \mathcal{P}^{n, M}} \frac{1}{\mathbb{P}[P \mbox{ is $k$-good}]}\mathbbm{1} \{P \mbox{ is $k$-good}\} .
$$
To prove Proposition~\ref{prop:BRW-thick-point}, we need an elementary result about sequences.

\begin{lemma}\label{lem:sequence}
    For any $a>1$, there exists a constant $C > 0$, which depends on $a$, such that the following hold for all $b \geq C$:
    \begin{enumerate}[(1)]
        \item \label{lem-sequence-1} Define the sequence $\{a_i\}_{i \geq 1}$ by $a_1 = a$ and $a_{i+1} = a( 1 + \frac{1}{b} a_i^{11})^2$ for all $i \geq 1$. Then, $\sup_{i \geq 1} a_i \leq 2a$.

        \item \label{lem-sequence-2} Define the sequence $\{b_i\}_{i \geq 1}$ by $b_1 = a$ and $b_{i+1} = (1 + \frac{1}{b} (b_i^{11} - 1))^2$ for all $i \geq 1$. Then, $b_i \leq 1 + 2^{-i}$ for all $i \geq 2$.
    \end{enumerate}

\end{lemma}
\begin{proof}
    Fix $a>1$. We first prove Claim~\eqref{lem-sequence-1}. For all sufficiently large $b$, we have $a - a(1 + \frac{1}{b} a^{11})^2 < 0 < 2a - a(1 + \frac{1}{b} (2a)^{11})^2$, and hence the equation $x = a (1 + \frac{1}{b} x^{11})^2$ has at least one solution in $(a, 2a)$. Let $x_*$ be such a solution. If $a_i \leq x_*$, then $a_{i+1} \leq a( 1 + \frac{1}{b} x_*^{11})^2 = x_*$. Since $a_1 = a < x_*$, by induction, we obtain that $a_i \leq x_* \leq 2a$ for all $i \geq 1$. This proves Claim~\eqref{lem-sequence-1}. 
    
    Next, we prove Claim~\eqref{lem-sequence-2}. For all sufficiently large $b$, we have $b_2 = (1 + \frac{1}{b}(a^{11}-1))^2 \leq 1 +2^{-2}$. If $b_i \leq 1+ 2^{-i}$ and $b \geq 10^{10}$, then we have $b_{i+1} \leq (1 + \frac{1}{b}((1+2^{-i})^{11}-1))^2 \leq (1 + \frac{10000}{b} 2^{-i})^2 \leq 1 + 2^{-i-1}$. By induction, we obtain Claim~\eqref{lem-sequence-2} for all sufficiently large $b$. \qedhere

\end{proof}

Next, we prove Proposition~\ref{prop:BRW-thick-point} by estimating the first and second moments of $\mathcal{N}$.

\begin{proof}[Proof of Proposition~\ref{prop:BRW-thick-point}]
    It suffices to show that there exists a constant $C>0$ depending on $\alpha$ such that for all $d \geq C$, integers $n \geq k \geq 1$, and $M \geq 10$:
    \begin{equation}\label{eq:prop2.1-1}
    \mathbb{E} \mathcal{N}^2 \leq \frac{1-c_1}{1 - c_1(1+2^{-\lfloor \frac{k}{4} \rfloor})} (\mathbb{E} \mathcal{N})^2 \quad \mbox{with the constant $c_1$ from Lemma~\ref{lem:high-d-intersect}.}
    \end{equation}
    Proposition~\ref{prop:BRW-thick-point} then follows from the Cauchy–Schwarz inequality: $\mathbb{P}[\mathcal{N}> 0] \geq \frac{(\mathbb{E} \mathcal{N})^2}{\mathbb{E} \mathcal{N}^2}$.
    
    Let $\mathfrak p = \mathbb{P}[N(0, \log 2) \geq \alpha]$, where $N(0, \log 2)$ denotes the Gaussian random variable with mean 0 and variance $\log 2$. Let $\mathbf E$ denote the expectation with respect to two paths $P$ and $Q$ uniformly chosen from $\mathcal{P}^{n,M}$. For $0 \leq i \leq n$, let $P_i$ and $Q_i$ be the corresponding paths of $P$ and $Q$ in $\mathcal{P}^{i, M}$, respectively. Since $\mathbb{E} \mathcal{N} = |\mathcal{P}^{n, M}|$ and
    $$
    \mathbb{E} \mathcal{N}^2 = \sum_{P, Q \in \mathcal{P}^{n, M}} \frac{\mathbb{P}[\mbox{$P,Q$ are $k$-good}]}{\mathbb{P}[\mbox{$P$ is $k$-good}] \cdot \mathbb{P}[\mbox{$Q$ is $k$-good}]},
    $$
    we have
    $$
    \frac{\mathbb{E} \mathcal{N}^2 }{(\mathbb{E}\mathcal{N})^2} = \mathbf E \Big[\frac{\mathbb{P}[\mbox{$P,Q$ are $k$-good}]}{\mathbb{P}[\mbox{$P$ is $k$-good}] \cdot \mathbb{P}[\mbox{$Q$ is $k$-good}]} \Big],
    $$
    where $P,Q$ are uniformly chosen from $\mathcal{P}^{n,M}$, and $\mathbb{P}$ is defined with respect to the branching random walk which is independent of $P,Q$. For all paths $P$ and $j \geq 0$, let $\mathcal{B}_j(P) = \{ x \in 2^{-j} \mathbb{Z}^d: x + [0,2^{-j})^d \cap P \neq \emptyset \}$.\footnote{Here we slightly abuse the notation, as in Section~\ref{subsec:BRW}, $\mathcal{B}_j(x)$ refers to the $2^{-j}$-box that contains $x$.} By the independence of $\{ a_{j, B} \}_{j \geq 0, B \in \mathcal{B}_j}$ (recall~\eqref{eq:def-brw}), we obtain
    \begin{equation}\label{eq:second-moment-brw}
    \frac{\mathbb{E} \mathcal{N}^2 }{(\mathbb{E}\mathcal{N})^2} = \mathbf E \Big[\frac{\mathfrak p^{\sum_{j=k}^n | \mathcal{B}_j(P) \cup \mathcal{B}_j(Q)|}}{\mathfrak p^{\sum_{j=k}^n | \mathcal{B}_j(P) |} \cdot \mathfrak p^{\sum_{j=k}^n | \mathcal{B}_j(Q) |}} \Big]= \mathbf E \Big[\mathfrak p^{-\sum_{j=k}^n | \mathcal{B}_j(P) \cap \mathcal{B}_j(Q)|} \Big].
    \end{equation}
    
    Next, we upper-bound the right-hand side of~\eqref{eq:second-moment-brw}. We first show that for all $P, Q \in \mathcal{P}^{n,M}$,
    \begin{equation}\label{eq:control-intersect-brw}
    \sum_{j=k}^n | \mathcal{B}_j(P) \cap \mathcal{B}_j(Q)| \leq 100 \sum_{j = \lfloor \frac{k}{4} \rfloor}^n |P_j \cap Q_j|.
    \end{equation}
    For $k \leq j \leq n$, let $l = \lfloor (j-1)/3 \rfloor + 1$ which satisfies $j \leq 3l$. By Claim~\ref{property-2-cal} in Lemma~\ref{lem:basic-property-mathcalP}, we have $\mathcal{B}_{3l}(P_l) = \mathcal{B}_{3l}(P)$, which implies that $\mathcal{B}_j(P_l) = \mathcal{B}_j(P)$. Therefore, we have $|\mathcal{B}_j(P) \cap \mathcal{B}_j(Q)| = |\mathcal{B}_j(P_l) \cap \mathcal{B}_j(Q_l)|$, yielding that
    \begin{equation}\label{eq:proof-brw-1}
    \sum_{j = k}^n | \mathcal{B}_j(P) \cap \mathcal{B}_j(Q)| \leq \sum_{l = \lfloor \frac{k-1}{3} \rfloor +1}^n \sum_{j = 3l-2}^{3l} |\mathcal{B}_j(P_l) \cap \mathcal{B}_j(Q_l)|.
    \end{equation}
    We claim that $\sum_{j = 3l-2}^{3l} |\mathcal{B}_j(P_l) \cap \mathcal{B}_j(Q_l)| \leq 100|P_{l-1} \cap Q_{l-1}|$. For each $2^{-j}$-box with $3l-2 \leq j \leq 3l$ that intersects both $P_l$ and $Q_l$, we consider the vertex $x$ in $P_l$ that lies in this box. We also consider the edge $e$ in $P_{l-1}$ that $x$ refines, i.e., $x$ belongs to the refinement of $e$ in $P_l$; see Figure~\ref{fig:intersect}. There can be at most two choices for $e$ and when $e$ is not unique we deterministically choose one of them arbitrarily. Now we show that $e \cap Q_{l-1} \neq \emptyset$. Suppose that $e \cap Q_{l-1} = \emptyset$, then $\mathfrak d_\infty(\overline{e}, \overline{Q_{l-1}}) \geq 8^{-l + 1}$. Moreover, by Claim~\ref{property-1-cal} in Lemma~\ref{lem:basic-property-mathcalP}, we know that 
    $$x \in \{z : \mathfrak d_\infty(z, \overline{e}) \leq \frac27 8^{-l + 1} \} \quad \mbox{and} \quad \overline{Q_l} \subset \{z : \mathfrak d_\infty(z, \overline{Q_{l-1}}) \leq \frac27 8^{-l + 1} \}.$$
    Therefore, it is impossible for a $2^{-j}$-box to intersect both $x$ and $Q_l$, which contradicts with the assumption. Thus, we have $e \cap Q_{l-1} \neq \emptyset$. 
    
    We map all the $2^{-j}$-boxes that intersect both $P_l$ and $Q_l$ to this pair of $(x,e)$. Notice that there are at most $2 |P_{l-1} \cap Q_{l-1}|$ edges in $P_{l-1}$ that intersect $Q_{l-1}$ (this is due to the fact that each of these edges has at least one endpoint in $P_{l-1} \cap Q_{l-1}$, and since $P_{l-1}$ is self-avoiding, each vertex in $P_{l-1} \cap Q_{l-1}$ can be the endpoint of at most two edges in $P_{l-1}$), and each edge is refined to 11 vertices in $P_l$. Therefore, the number of choices for $e$ is at most $2 |P_{l-1} \cap Q_{l-1}|$, and given $e$, the number of choices for $x$ and $j$ is at most $3 \times 11$. Combined with the above argument, this yields that
    \begin{equation}\label{eq:proof-brw-2}
    \sum_{j = 3l-2}^{3l} |\mathcal{B}_j(P_l) \cap \mathcal{B}_j(Q_l)| \leq 2 |P_{l-1} \cap Q_{l-1}| \times 3 \times 11 \leq 100 |P_{l-1} \cap Q_{l-1}|.
    \end{equation}
    This proves the claim. Combining~\eqref{eq:proof-brw-1} and~\eqref{eq:proof-brw-2} with the fact that $\lfloor \frac{k-1}{3} \rfloor \geq \lfloor \frac{k}{4} \rfloor$, we obtain~\eqref{eq:control-intersect-brw}. 

\begin{figure}[h]
\centering
\includegraphics[scale=0.6]{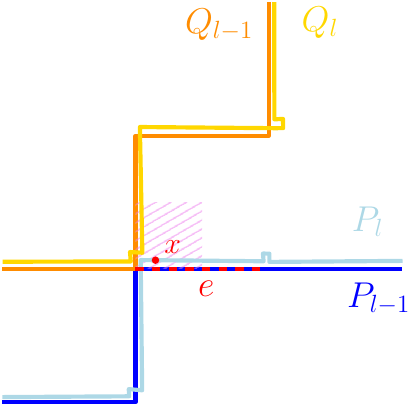}
\caption{The colored paths $P_{l-1}$ and $Q_{l-1}$ are defined on $8^{-l+1} \mathbb{Z}^d$, and the colored paths $P_l$ and $Q_l$ are defined on $8^{-l} \mathbb{Z}^d$. The shaded region corresponds to a $2^{-j}$-box with $j = 3l-2$ that intersects both $P_l$ and $Q_l$. We can map it to a pair $(x,e)$ such that $e$ is an edge in $P_{l-1}$ that intersects $Q_{l-1}$, and $x$ is a vertex in $P_l$ that belongs to the refinement of $e$. The red point corresponds to $x$, and the dashed red edge corresponds to $e$.}
\label{fig:intersect} 
\end{figure}
    
    Let $Y_j = |P_j \cap Q_j|$ for $0 \leq j \leq n$. Combining~\eqref{eq:second-moment-brw} and~\eqref{eq:control-intersect-brw}, we have
    \begin{equation}\label{eq:second-moment-brw-2}
         \frac{\mathbb{E} \mathcal{N}^2 }{(\mathbb{E}\mathcal{N})^2} \leq \mathbf E \Big[\mathfrak p^{-100 \sum_{j = \lfloor \frac{k}{4} \rfloor}^n Y_j} \Big].
    \end{equation}
    Next, we upper-bound the right-hand side of~\eqref{eq:second-moment-brw-2} using Lemmas~\ref{lem:intersect-BRW} and~\ref{lem:sequence}. Recall from Lemma~\ref{lem:intersect-BRW} that for all $1 \leq i \leq n$, given $(Y_0, Y_1, \ldots Y_{i-1})$, 
    $$
    Y_i\mbox{ is stochastically dominated by } 11(Z_1 + Z_2 + \ldots + Z_{2Y_{i-1}}),$$
    where $\{Z_l\}_{l \geq 1}$ are i.i.d.\ Bernoulli random variables with probability $\frac{50}{\mathfrak q - 10}$ of being 1. Therefore, for all $a>1$, given $(Y_0, Y_1, \ldots Y_{i-1})$, we have
    \begin{equation}\label{eq:proof-prop2.1-dominate}
    \mathbf E [a^{Y_i}|Y_0, Y_1, \ldots Y_{i-1}] \leq \mathbf E [a^{11(Z_1 + Z_2 + \ldots + Z_{2Y_{i-1}})}] = (1 + \frac{50}{\mathfrak q - 10}(a^{11}-1))^{2 Y_{i-1}}.
    \end{equation}
    Let $a_1 = \mathfrak p^{-100}$ and $a_{i+1} = \mathfrak p^{-100}(1 + \frac{50}{\mathfrak q - 10} a_i^{11})^2$ for all $i \geq 1$. By Claim~\eqref{lem-sequence-1} in Lemma~\ref{lem:sequence} with $a = \mathfrak{p}^{-100}$, for all sufficiently large $d$, we have $ a_i \leq 2 \mathfrak{p}^{-100}$ for all $i \geq 1$. Applying~\eqref{eq:proof-prop2.1-dominate} inductively to the right-hand side of~\eqref{eq:second-moment-brw-2}, we obtain that for all sufficiently large $d$,
    \begin{align*}
    \frac{\mathbb{E} \mathcal{N}^2}{(\mathbb{E} \mathcal{N})^2} &=  \mathbf E \Big[ \mathbf E \Big[\mathfrak p^{- 100\sum_{i=\lfloor \frac{k}{4} \rfloor }^{n-1} Y_i} \times a_1^{Y_n}\big{|} Y_0, Y_1, \ldots, Y_{n-1}\Big]\Big] \\
    &\leq \mathbf E \Big[  \mathfrak p^{- 100\sum_{i=\lfloor \frac{k}{4} \rfloor}^{n-2} Y_i} a_2^{Y_{n-1}}\Big]\leq \ldots \leq \mathbf E \Big[ a_{n-\lfloor \frac{k}{4} \rfloor-1}^{Y_{\lfloor \frac{k}{4} \rfloor}}\Big] \leq \mathbf E \big[ \mathfrak (2 \mathfrak p^{-100})^{Y_{\lfloor \frac{k}{4} \rfloor}}\big].
    \end{align*}
    Now let $b_1 = 2 \mathfrak p^{-100}$ and $b_{i+1} = (1 + \frac{50}{\mathfrak q - 10} (b_i^{11} - 1))^2$ for all $i \geq 1$. Applying Claim~\eqref{lem-sequence-2} in Lemma~\ref{lem:sequence} with $a = 2 \mathfrak p^{-100}$, we see that when $d$ is sufficiently large, we have $b_i \leq 1 + 2^{-i}$ for all $i \geq 2$. Applying~\eqref{eq:proof-prop2.1-dominate} inductively to the right-hand side of the above inequality, we obtain
    $$\frac{\mathbb{E} \mathcal{N}^2}{(\mathbb{E} \mathcal{N})^2} \leq \mathbf E \Big[\mathbf E \Big[ b_1^{Y_{\lfloor \frac{k}{4} \rfloor}}|Y_0, Y_1, \ldots, Y_{\lfloor \frac{k}{4} \rfloor-1}\Big]\Big]\leq \mathbf E \Big[ b_2^{Y_{\lfloor \frac{k}{4} \rfloor-1}} \Big]  \leq \ldots \leq \mathbf E \Big[b_{\lfloor \frac{k}{4} \rfloor+1}^{Y_0} \Big] \leq \mathbf E \Big[\mathfrak (1 + 2^{-\lfloor \frac{k}{4} \rfloor})^{Y_0}\Big].$$
    By Lemma~\ref{lem:high-d-intersect}, $Y_0$ is stochastically dominated by a geometric random variable with success probability $c_1$. By simple calculations, we obtain~\eqref{eq:prop2.1-1} and finish the proof of Proposition~\ref{prop:BRW-thick-point}.
\end{proof}

\subsection{Proof of Theorem~\ref{thm:BRW-thick}}\label{subsec:sec2-final-proof}

In this subsection, we use Proposition~\ref{prop:BRW-thick-point} and Lemma~\ref{lem:Hausdorff-limit-brw} to prove Theorem~\ref{thm:BRW-thick}. Fix $\alpha>0$ and consider sufficiently large $d$ that satisfies Proposition~\ref{prop:BRW-thick-point}. Fix an integer $k \geq 1$, which will eventually tend to infinity. For integers $n \geq k$ and $M \geq 10$, let $\mathcal{K}(n,M)$ be the event that there exists a $k$-good path in $\mathcal{P}^{n,M}$. Since $\mathcal{P}^{n,M}$ satisfies a restriction property with respect to $M$ (see Lemma~\ref{lem:restrict-M-brw}), the event $\mathcal{K}(n,M)$ is decreasing with respect $M$. Using Claim~\ref{property-2-cal} in Lemma~\ref{lem:basic-property-mathcalP}, we also have that $\mathcal{K}(n,M)$ is decreasing with respect $n$. Therefore, by Proposition~\ref{prop:BRW-thick-point}, we have
\begin{equation}\label{eq:thm1.2-prove-1}
\mathbb{P}\Big[\bigcap_{n \geq k, M \geq 10} \mathcal{K}(n,M)\Big] \geq 1 - \epsilon(k), \quad \mbox{with $\lim_{k \rightarrow \infty} \epsilon(k) = 0$.}
\end{equation}
We claim that on the event $\cap_{n \geq k} \cap_{M \geq 10} \mathcal{K}(n,M)$, the set $\TBRW_\alpha$ contains an unbounded path. Suppose that the event $\cap_{n \geq k} \cap_{M \geq 10} \mathcal{K}(n,M)$ occurs. Then, we can choose a $k$-good path $P_n \in \mathcal{P}^{n,n}$ for all $n \geq \max \{k, 10\}$. There exists a subsequence $\{n_m\}_{m \geq 1}$ such that $\{P_{n_m}\}_{m \geq 1}$ converges to a limiting set $\widetilde P$ with respect to the local Hausdorff distance. Indeed, this follows from compactness. Recall that for any $L \geq 1$, the space of closed subsets of $[-L,L]^d$ is compact w.r.t.\ the Hausdorff distance. Therefore, for any sequence, we can extract a subsequence along which $\{P_n \cap [-L,L]^2\}$ converges w.r.t.\ the Hausdorff distance. Using a diagonal argument, we can then select a subsequence that converges w.r.t.\ the local Hausdorff distance. 

By Lemma~\ref{lem:Hausdorff-limit-brw}, we see that $\widetilde P$ is a continuous path that connects $\{z : |z|_1 = 2 \}$ to infinity. Moreover, by the condition of $k$-good from~\eqref{eq:brw-good} and the last property in Lemma~\ref{lem:Hausdorff-limit-brw}, we have
$$
a_{j, \mathcal{B}_j(x)} \geq \alpha \quad \mbox{for all $j \geq k$ and $x \in \widetilde P$}.
$$
Therefore, by~\eqref{eq:def-brw}, $\liminf_{j \rightarrow \infty} \frac{1}{j} \mathcal{R}_j(x) \geq \alpha$ for all $x \in \widetilde P$, and hence, $\widetilde P \subset \TBRW_\alpha$. This proves the claim. Taking $k$ to infinity and using~\eqref{eq:thm1.2-prove-1} yields Theorem~\ref{thm:BRW-thick}.

\section{The white noise field}\label{sec:WN}

In this section, we study the white noise field $\{h_n\}_{n \geq 1}$ defined in~\eqref{eq:def-WN} and prove Theorems~\ref{thm:WN-thick} and \ref{thm:exponential-metric}. 

The proof strategy is similar to that of Theorem~\ref{thm:BRW-thick} using first and second moment estimates. However, unlike the branching random walk case, the increments $\{h_n - h_{n-1}\}_{n \geq 1}$ (where we set $h_0 = 0$) are not constant on a box. Instead, they have fluctuations that are costly to control in high dimensions. So, there will be differences in the choice of paths and the definition of good paths that we count. The first difference is in the choice of paths. Recall from~\eqref{eq:def-dimenion-constant} that $\mathfrak r = \lfloor \sqrt{d} \rfloor$. In Section~\ref{subsec:path-WN}, we will construct a set of paths $\mathscr{P}^{n,M}$ on $\frac{1}{\mathfrak r} 8^{-n} \mathbb{Z}^d$ that connect $\{z : |z|_1 = 1 + \frac{2}{7 \mathfrak r} \}$ to $\{z : |z|_1 = M - \frac{2}{7 \mathfrak r} \}$ for $n \geq 0$ and $M \geq 10$. Here, $\frac{1}{\mathfrak r}$ is the correlation length of the white noise field in high dimensions (Lemma~\ref{lem:est-correlation}), and we take 8 as a dyadic number large enough to allow the refinements of disjoint paths to be far away from each other. 

The second difference is in the definition of good paths. If we directly use the condition~\eqref{eq:prop-WN-1} below, estimating the second moment of the number of good paths would become challenging. To overcome this, we add an auxiliary condition so that the second moment estimate involving this condition is easier and this auxiliary condition ensures the occurrence of~\eqref{eq:prop-WN-1} with high probability. We refer to Section~\ref{subsec:good-WN} for the precise definition of good paths and the method to estimate the first and second moments of the (weighted) sum of good paths. The main result of this section is the following proposition. Theorems~\ref{thm:WN-thick} and~\ref{thm:exponential-metric} follow directly from this proposition; see Section~\ref{subsec:moment-WN}.
\begin{proposition}\label{prop:WN-thick}
    For any fixed $\alpha>0$, the following holds for all sufficiently large $d$. There exists a function $\epsilon : \mathbb{N} \rightarrow (0,1)$ depending on $\alpha$ and $d$ with $\lim_{k \rightarrow \infty} \epsilon(k) = 0$ such that for all integers $k \geq 1$, $n \geq 6k$ and all $M \geq 10$, with probability at least $1-\epsilon(k)$, there exists a path $P$ in $\mathscr{P}^{n,M}$ satisfying 
    \begin{equation}\label{eq:prop-WN-1}
    h_j(x) - h_{3k}(x) \geq \alpha (j-3k) \quad \mbox{for all $j \in [6k, n] \cap \mathbb{Z}$ and $x \in \cup_{j \leq i \leq n} \overline{P_i}$}.
    \end{equation}
    Here $P_j$ is the corresponding path of $P$ in $\mathscr{P}^{j,M}$ (see Definition~\ref{def:Pn-WN}, and in particular $P_n = P$) and
    $\overline P$ is the continuous path obtained by connecting neighboring vertices in $P$.
\end{proposition}

\begin{remark}\label{rmk:path-choose}
    We require~\eqref{eq:prop-WN-1} holds for all $x \in \cup_{j \leq i \leq n} \overline{P_i}$, not just for $x \in \overline{P_n} = \overline{P}$, to make the event decreasing with respect to $n$. In the proof of Theorem~\ref{thm:exponential-metric}, we will use a variant of Proposition~\ref{prop:WN-thick} (see Proposition~\ref{prop:WN-thick-variant}), where we can find a path with large $h_n$ values connecting two far-away boxes. Similarly, we can also find a path with large $h_n$ values crossing any box, which can be used to prove Theorem~\ref{thm:crossing}.
\end{remark}

The rest of this section is organized as follows. In Section~\ref{subsec:prelim}, we provide some preliminary results. In Section~\ref{subsec:path-WN}, we construct a collection of paths and collect their properties. In Section~\ref{subsec:good-WN}, we define good paths. In Section~\ref{subsec:conditional}, we show that the auxiliary conditions imply the condition~\eqref{eq:prop-WN-1} in a box with high probability. Finally, in Sections~\ref{subsec:moment-WN} and~\ref{subsec:proof-WN}, we perform the first and second moment estimates and prove Proposition~\ref{prop:WN-thick}, thereby proving Theorems~\ref{thm:WN-thick} and \ref{thm:exponential-metric}.

\subsection{Preliminary estimates}
\label{subsec:prelim}

We record three elementary results (Lemmas~\ref{lem:est-correlation}--\ref{lem:repulsion}) that will be used in the proof of Proposition~\ref{prop:WN-thick}.

\begin{lemma}\label{lem:est-correlation}

\begin{enumerate}
    \item \label{est-correlation-1} For any fixed $t>0$, there exists a constant $c_2 = c_2(t)$ such that for all $d \geq c_2^{-1}$ and $|x-y|_2 \leq \frac{t}{\sqrt{d}}$, we have $\Vol (B_1(x) \cap B_1(y)) \geq c_2 \Vol(B_1(0))$.
    
    \item \label{est-correlation-2} There exists a universal constant $c_3 > 0$ such that for all $d \geq 2$ and $x, y \in \mathbb{R}^d$, we have $\Vol (B_1(x) \cap B_1(y)) \leq \frac{1}{c_3} \exp(- c_3 d |x-y|_2^2) \Vol(B_1(0))$.
    
\end{enumerate}
\end{lemma}

\begin{proof}
    Without loss of generality, assume that $x = 0$ and $y = u \mathbf e_1$ for some $u \geq 0$. By definition, for all $0 \leq u \leq 2$
    \begin{equation}\label{eq:lem3.2-1}
    \begin{aligned}
        \Vol(B_1(0) \cap B_1(u\mathbf e_1)) &= \Vol ( \{ z: -1+u \leq z_1 \leq u/2, z_2^2  + \ldots + z_d^2 \leq 1-(u-z_1)^2 \})\\
        &\quad + \Vol ( \{ z: u/2 \leq z_1 \leq 1, z_2^2  + \ldots + z_d^2 \leq 1-z_1^2 \}) \\
        &= 2 \int_{u/2}^1 \widetilde{{\rm vol}} (B_1(0)) (1-r^2)^{\frac{d-1}{2}} dr,
    \end{aligned}
    \end{equation}
    where $\widetilde{{\rm vol}} (B_1(0))$ is the $(d-1)$-dimensional volume of a $(d-1)$-dimensional unit ball. We first prove Claim~\ref{est-correlation-1}. Fix $t>0$ and assume that $d > 4 t^2 $ (or equivalently, $\frac{t}{\sqrt{d}} < \frac{1}{2}$). Since $\Vol (B_1(0)) = \frac{\pi^{d/2}}{\Gamma(1 + d/2)}$ and $\widetilde{{\rm vol}} (B_1(0)) = \frac{\pi^{(d-1)/2}}{\Gamma(1/2+d/2)}$, by Stirling's formula, there exists a universal constant $c>0$ such that 
    \begin{equation}\label{eq:lem3.2-0}
        c \sqrt{d} \leq \frac{\widetilde{{\rm vol}} (B_1(0))}{\Vol(B_1(0))} \leq \frac{1}{c} \sqrt{d} \quad \mbox{for all $d \geq 2$}.
    \end{equation}Therefore, by~\eqref{eq:lem3.2-1}, for all $0 \leq u \leq \frac{t}{\sqrt{d}}$, we have
    \begin{align*}
        \Vol(B_1(0) \cap B_1(u\mathbf e_1)) & \geq 2 c \sqrt{d} \cdot \Vol(B_1(0)) \int_{\frac{t}{2 \sqrt{d}}}^{\frac{t}{\sqrt{d}}} (1-r^2)^{\frac{d-1}{2}} dr.
    \end{align*}
    We conclude Claim~\ref{est-correlation-1} by noting that $(1-r^2)^{\frac{d-1}{2}}$ is bounded from below by a positive constant that depends only on $t$ for all $d > 4t^2$ and $\frac{t}{2 \sqrt{d}} \leq r \leq \frac{t}{\sqrt{d}}$.

    Next, we prove Claim~\ref{est-correlation-2}. For all $a \in (0,1)$ and $b>0$, we have $(1-a)^b \leq e^{-ab}$. Combined with~\eqref{eq:lem3.2-1} and~\eqref{eq:lem3.2-0}, it yields that
    $$
    \Vol(B_1(0) \cap B_1(u\mathbf e_1)) \leq \frac{2}{c}  \sqrt{d} \cdot \Vol(B_1(0)) 
    \int_{u/2}^1 e^{-\frac{d-1}{2} r^2} dr. 
    $$
    Taking $r' = \sqrt{\frac{d-1}{2}} r$ and then using the inequality $\int_a^\infty e^{-z^2} dz \leq C e^{-a^2}$ for some universal constant $C>0$ and any $a>0$, we further have
    $$
    \Vol(B_1(0) \cap B_1(u\mathbf e_1)) \leq \frac{2}{c} \frac{\sqrt{d}}{\sqrt{\frac{d-1}{2}}} \cdot \Vol(B_1(0)) 
    \int_{\frac{u}{2} \sqrt{\frac{d-1}{2}}}^\infty e^{-r'^2} dr' \leq \frac{2}{c} \frac{\sqrt{d}}{\sqrt{\frac{d-1}{2}}} \cdot \Vol(B_1(0)) \cdot C e^{-\frac{u^2}{4} \frac{d-1}{2}}. 
    $$
    This concludes Claim~\ref{est-correlation-2} by choosing $c_3$ sufficiently small.
\end{proof}

Next, we prove a Gaussian correlation inequality. For $a \in \mathbb{R}$ and $b>0$, we use $N(a,b)$ to denote the normal distribution with mean $a$ and variance $b$. We will frequently use the following estimate for the asymptotic decay rate of the normal distribution from \cite[Theorem 1.2.6]{durrett}: There exists a universal constant $c_4 > 0$ such that
\begin{equation}
    \label{eq:gaussian-tail}
    c_4 x^{-1} e^{-x^2/2} \leq \mathbb{P}[ N(0,1) \geq x] \leq c_4^{-1} x^{-1} e^{-x^2/2} \quad \mbox{for all } x \geq \frac{1}{2}.
\end{equation}

\begin{lemma}\label{lem:gaussian-correlation}
    There exists a universal constant $C_1 > 0$ such that for all mean-zero joint normal variables $(X,Y)$ with $0 < \mathbb{E} [X^2], \mathbb{E} [Y^2] \leq 1$, $\mathbb{E}[XY] \geq 0$, and for all $t \geq 1$, we have
    \begin{equation}\label{eq:lem3.2}
    \frac{\mathbb{P}[ X \geq t, Y \geq t]}{\mathbb{P}[ X \geq t] \mathbb{P}[ Y \geq t]} \leq \exp\Big( C_1 t^2 \frac{\mathbb{E} [XY]}{\mathbb{E} [X^2] \mathbb{E} [Y^2] } \Big).
    \end{equation}
    Note that the equality holds when $\mathbb{E}[XY] = 0$.
\end{lemma}

\begin{proof}
    Without loss of generality, we assume that $\mathbb{E} [Y^2] \geq \mathbb{E} [X^2]$. Let $K > 10$ be a large universal constant to be chosen. The proof of~\eqref{eq:lem3.2} will be divided into two cases: (1). $\mathbb{E}[XY] \geq \frac{1}{K} \mathbb{E}[X^2]$; and (2). $\mathbb{E}[XY] < \frac{1}{K} \mathbb{E}[X^2]$. We first prove the case where $\mathbb{E}[XY] \geq \frac{1}{K} \mathbb{E}[X^2]$. Applying~\eqref{eq:gaussian-tail} with $x = \frac{t}{\sqrt{\mathbb{E}[Y^2]}}$, we obtain
    $$
    \frac{\mathbb{P}[ X \geq t, Y \geq t]}{\mathbb{P}[ X \geq t] \mathbb{P}[ Y \geq t]} \leq \frac{1}{\mathbb{P}[Y \geq t]} \leq c_4^{-1} \frac{t}{\sqrt{\mathbb{E}[Y^2]}} \exp \big( \frac{t^2}{2 \mathbb{E} [Y^2]} \big).
    $$
    Using $\frac{t^2}{\mathbb{E} [Y^2]} \geq 1$, we have
    \begin{align*}
    &\quad c_4^{-1} \frac{t}{\sqrt{\mathbb{E}[Y^2]}} \exp \big( \frac{t^2}{2 \mathbb{E} [Y^2]} \big) \leq c_4^{-1} \frac{t^2}{\mathbb{E}[Y^2]} \exp \big( \frac{t^2}{2 \mathbb{E} [Y^2]} \big) \\
    &\leq \exp\big( (c_4^{-1} + \frac12) \frac{t^2}{\mathbb{E} [Y^2]} \big) \leq \exp\big( K(c_4^{-1} + \frac12) t^2 \frac{\mathbb{E} [XY]}{\mathbb{E} [X^2] \mathbb{E} [Y^2] } \big),
    \end{align*}
    where the second inequality uses $z \leq e^z$ for $z \geq 0$, and the third inequality uses the assumption $\mathbb{E}[XY] \geq \frac{1}{K} \mathbb{E}[X^2]$. This implies~\eqref{eq:lem3.2} by choosing a sufficiently large $C_1$ which depends only on $K$.

    Next, we prove the case where $\mathbb{E}[XY] < \frac{1}{K} \mathbb{E}[X^2]$. Let $m = \lfloor \frac{\mathbb{E}[X^2]}{2 \mathbb{E}[XY]} \rfloor - 1$. Then, we have
    \begin{equation}\label{eq:lem3.2-2}
    \frac{\mathbb{P}[ X \geq t, Y \geq t]}{\mathbb{P}[ X \geq t] \mathbb{P}[ Y \geq t]} \leq \underbrace{\sum_{n=1}^m \frac{\mathbb{P}[ (n+1) t \geq X \geq n t, Y \geq t]}{\mathbb{P}[ X \geq t] \mathbb{P}[ Y \geq t]}}_{I_1} + \underbrace{\frac{\mathbb{P}[ X \geq (m+1) t ]}{\mathbb{P}[ X \geq t] \mathbb{P}[ Y \geq t]}}_{I_2}.
    \end{equation}
    Denote the two terms on the right-hand side of~\eqref{eq:lem3.2-2} by $I_1$ and $I_2$. Next we upper-bound them separately. Let $Z = Y - \frac{\mathbb{E}[XY]}{\mathbb{E}[X^2]} X$. Then $X$ and $Z$ are independent and $\mathbb{E}[Z^2] = \mathbb{E}[Y^2] - \frac{\mathbb{E}[XY]^2}{\mathbb{E}[X^2]} \leq \mathbb{E}[Y^2]$. For all $1 \leq n \leq m$, conditioned on the event $(n+1) t \geq X \geq nt$, the event $Y \geq t$ implies that $Z \geq t - \frac{\mathbb{E}[XY]}{\mathbb{E}[X^2]} (n+1) t$. Therefore,
    \begin{align*}
        I_1 &= \sum_{n=1}^m \frac{\mathbb{P}[Y \geq t |  (n+1) t \geq X \geq n t]}{\mathbb{P}[ Y \geq t]} \times \frac{\mathbb{P}[(n+1) t \geq X \geq n t]}{\mathbb{P}[ X \geq t] } \\
        &\leq \sum_{n=1}^m \frac{\mathbb{P}[Z \geq t - \frac{\mathbb{E}[XY]}{\mathbb{E}[X^2]} (n+1) t]}{\mathbb{P}[ Y \geq t]} \times \frac{\mathbb{P}[(n+1) t \geq X \geq n t]}{\mathbb{P}[ X \geq t] }.
    \end{align*}
    Using~\eqref{eq:gaussian-tail}, for all $\frac{1}{2} \leq x_1 \leq x_2$, we have
    \begin{align*}
    \frac{\mathbb{P}[ N(0,1) \geq x_1]}{\mathbb{P}[ N(0,1) \geq x_2]} &= 1 + \frac{\int_{x_1}^{x_2} e^{-y^2/2} dy}{\int_{x_2}^\infty e^{-y^2/2} dy} \leq 1 + \frac{(x_2 - x_1) e^{-x_1^2/2}}{ \sqrt{2 \pi} c_4 x_2^{-1} e^{-x_2^2/2}} \\
    &\leq 1 + c_4^{-1} x_2 (x_2-x_1) e^{x_2(x_2 - x_1)} \leq e^{ (c_4^{-1} + 1) x_2(x_2 - x_1)}.
    \end{align*}
    In the last inequality, we used the relation $e^{y+z} \geq e^y (1+z) \geq 1 + ze^y$ with $(y, z) = (x_2 (x_2-x_1),  c_4^{-1} x_2 (x_2-x_1))$.
    Using $\mathbb{P}[Z \geq t] \leq \mathbb{P}[Y \geq t]$ for all $t \geq 0$ and applying the preceding inequality with $x_1 = \frac{1}{\sqrt{\mathbb{E}[Y^2]}}(t - \frac{\mathbb{E}[XY]}{\mathbb{E}[X^2]} (n+1) t)$ and $x_2 = \frac{1}{\sqrt{\mathbb{E}[Y^2]}} t$, we further have
    $$
        I_1 \leq \sum_{n=1}^m e^{(c_4^{-1}+1) (n+1)  \frac{t^2 \mathbb{E}[XY]}{\mathbb{E}[X^2] \mathbb{E}[Y^2]} } \times \frac{\mathbb{P}[(n+1) t \geq X \geq n t]}{\mathbb{P}[ X \geq t] }.
    $$
    Using $\mathbb{P}[(n+1) t \geq X \geq n t] = \mathbb{P}[X \geq nt] - \mathbb{P}[X \geq (n+1)t]$ and rearranging the summation, we get
    \begin{align*}
        I_1 \leq e^{2 (c_4^{-1}+1) \frac{t^2 \mathbb{E}[XY]}{\mathbb{E}[X^2] \mathbb{E}[Y^2]}} + (e^{(c_4^{-1}+1) \frac{t^2 \mathbb{E}[XY]}{\mathbb{E}[X^2] \mathbb{E}[Y^2]}} - 1)  \underbrace{\sum_{n = 2}^m e^{ (c_4^{-1}+1) n \frac{t^2 \mathbb{E}[XY]}{\mathbb{E}[X^2] \mathbb{E}[Y^2]}} \frac{\mathbb{P}[X \geq n t]}{\mathbb{P}[ X \geq t] }}_{I_3}.
    \end{align*}
    Using~\eqref{eq:gaussian-tail}, we have $\frac{\mathbb{P}[X \geq n t]}{\mathbb{P}[ X \geq t] } \leq c_4^{-2} n^{-1} \exp( - \frac{(n^2-1)t^2}{2 \mathbb{E}[X^2]})$. Using $\mathbb{E}[XY] \leq \frac{1}{K} \mathbb{E}[Y^2]$ and $\frac{t^2}{\mathbb{E}[X^2]} \geq 1$, we can fix a sufficiently large $K$ such that the sum $I_3$ is upper-bounded by a universal constant. This implies that $I_1 \leq \exp(C\frac{t^2 \mathbb{E}[XY]}{\mathbb{E}[X^2] \mathbb{E}[Y^2]}) + C(\exp(C\frac{t^2 \mathbb{E}[XY]}{\mathbb{E}[X^2] \mathbb{E}[Y^2]})-1) \leq \exp(C'\frac{t^2 \mathbb{E}[XY]}{\mathbb{E}[X^2] \mathbb{E}[Y^2]})$ for some universal constants $C$ and $C'$. The term $I_2$ can be upper-bounded using~\eqref{eq:gaussian-tail} as follows:
    $$
    I_2 \leq c_4^{-3} \frac{t}{\lfloor \frac{\mathbb{E}[X^2]}{2 \mathbb{E}[XY]} \rfloor \sqrt{\mathbb{E}[Y^2]}} e^{\frac{t^2}{2 \mathbb{E}[X^2]} + \frac{t^2}{2 \mathbb{E}[Y^2]} - \frac{\lfloor \frac{\mathbb{E}[X^2]}{2 \mathbb{E}[XY]} \rfloor^2 t^2}{2 \mathbb{E}[X^2]}}.
    $$
    Using $\mathbb{E}[XY] \leq \frac{1}{10} \mathbb{E}[X^2]$ and $ \mathbb{E}[X^2] \leq \mathbb{E}[Y^2]$, we see that the term in the exponential is negative. Moreover, since $\frac{t^2}{\mathbb{E} [Y^2]} \geq 1$, we have $I_2 \leq C''\frac{t^2 \mathbb{E}[XY]}{\mathbb{E}[X^2] \mathbb{E}[Y^2]}$ for some universal constant $C''$. Combining~\eqref{eq:lem3.2-2} with the upper bounds for $I_1$ and $I_2$, and choosing $C_1$ sufficiently large, we obtain~\eqref{eq:lem3.2}. \qedhere

\end{proof}

Finally, we present an entropic repulsion type result about Gaussian random variables. In the special case where $X_1,X_2,\ldots,X_n$ are i.i.d.\ standard Gaussian random variables, the following lemma implies that conditioned on $\sum_{i=1}^n X_i \geq n \theta$ for fixed $\theta>0$, the law of $X_1$ asymptotically stochastically dominates $N(\theta,1)$ as $n \to \infty$ (in fact it converges to $N(\theta,1)$). For our purposes, we need a uniform control on the rate of this stochastic dominance, as incorporated in the following lemma.

\begin{lemma}\label{lem:repulsion}
    Fix $\theta, \delta > 0$. For any $\sigma^2, m > 0$, let $X$ and $Y$ be two independent mean-zero Gaussian random variables with $\mathbb{E} X^2 = \sigma^2$ and $\mathbb{E} Y^2 = m \sigma^2$. Then, we have
    $$
    \sup_{\delta \leq \sigma^2 \leq 1, t \leq \delta^{-1}} \frac{p[ X = t | X + Y \geq (m+1) \sigma^2 \theta]}{p [N(\sigma^2 \theta, \sigma^2) = t]} \leq 1 + o_m(1) \quad \mbox{as }m \rightarrow \infty,
    $$
    where $p(\cdot)$ denotes the probability density of a random variable, $N(\sigma^2 \theta, \sigma^2)$ is a normal random variable with mean $\sigma^2 \theta$ and variance $\sigma^2$, and the $o_m(1)$ term depends only on $\theta, \delta, m$.
\end{lemma}

\begin{proof}
    This lemma follows from the asymptotic expansion of $\mathbb{P}[N(0,1) \geq z]$ and simple calculations. By~\cite[Theorem 1.2.6]{durrett}, we have $\mathbb{P}[N(0,1) \geq z] = \frac{1 + o_z(1)}{\sqrt{2 \pi} z } e^{-z^2/2}$ as $z$ tend to infinity. Therefore, for all $\delta \leq \sigma^2 \leq 1$,
    $$
    \mathbb{P}[X+Y \geq (m+1) \sigma^2 \theta ] = \mathbb{P}[N(0,1) \geq \sqrt{m+1} \sigma \theta ] = \frac{1 + o_m(1)}{\sqrt{2 \pi m} \sigma \theta } e^{-\frac{(m+1) \sigma^2 \theta^2}{2}}.
    $$
    In addition, for all $\delta \leq \sigma^2 \leq 1$ and $t \leq \delta^{-1}$, we have
    \begin{align*}
    &\quad p[ X = t, X + Y \geq (m+1) \sigma^2 \theta] = \frac{1}{\sqrt{2\pi} \sigma }e^{-\frac{t^2}{2 \sigma^2}} \mathbb{P}[ Y \geq (m+1) \sigma^2 \theta - t] \\
    &= \frac{1 + o_m(1)}{2 \pi \sigma } e^{-\frac{t^2}{2 \sigma^2}} \frac{\sqrt{m} \sigma}{(m+1) \sigma^2 \theta - t}e^{-\frac{((m+1)  \sigma^2 \theta  - t)^2}{2m \sigma^2}}  \leq \frac{1 + o_m(1)}{2 \pi \sqrt{m} \sigma^2 \theta } e^{-\frac{(m+1) \sigma^2 \theta^2}{2}} \times e^{-\frac{t^2}{2\sigma^2} + \theta t - \frac{\sigma^2 \theta^2}{2} }.
    \end{align*}
    Note that all these $o_m(1)$ terms depend only on $\theta, \delta, m$ and are independent of $t$ and $\sigma$. Dividing the above two inequalities yields the result.
\end{proof}

\subsection{Construction of paths}\label{subsec:path-WN}

Recall that $\mathfrak r = \lfloor \sqrt{d} \rfloor$. In this subsection, for integers $n \geq 0$ and $M \geq 10$, we construct a set of self-avoiding paths $\mathscr{P}^{n,M}$ on $\frac{1}{\mathfrak r}2^{-3n} \mathbb{Z}^d$. Each path in $\mathscr{P}^{n,M}$ connects $\{z : |z|_1 = 1 + \frac{2}{7 \mathfrak r} \}$ to $\{z : |z|_1 = M - \frac{2}{7 \mathfrak r} \}$ (Lemma~\ref{lem:basic-property-mathscrP}), and two paths chosen uniformly from $\mathscr{P}^{n,M}$ will typically have few intersections (Lemma~\ref{lem:intersection-law-WN}). We also collect some properties of these paths that will be used in the proof of Proposition~\ref{prop:WN-thick}.

Recall from Definition~\ref{def:P0-brw} that $\mathcal{P}^{0,M}$ consists of oriented paths on $\mathbb{Z}^d$. The construction of $\mathscr{P}^{n,M}$ involves two steps: first we let $\mathscr{P}^{0,M}$ consist of oriented paths obtained by interpolating the paths in $\mathcal{P}^{0,M}$ (see the definition below), and then we inductively refine $\mathscr{P}^{0,M}$ on $\frac{1}{\mathfrak r} 8^{-n} \mathbb{Z}^d$ for each $n \geq 1$.

\begin{definition}[$\mathscr{P}^{0, M}$]\label{def:P0}
    For an integer $M \geq 10$, let $\mathscr{P}^{0,M}$ consist of the paths on $\frac{1}{\mathfrak r} \mathbb{Z}^d$ obtained by interpolating the paths in $\mathcal{P}^{0,M}$. For instance, for the edge $(\mathbf e_1,\mathbf e_1 + \mathbf e_2)$, we interpolate it as $(\mathbf e_1, \mathbf e_1 + \frac{1}{\mathfrak r} \mathbf e_2, \mathbf e_1 + \frac{2}{\mathfrak r} \mathbf e_2, \ldots, \mathbf e_1 + \mathbf e_2)$. 
\end{definition}

Then, $\mathscr{P}^{0, M}$ consists of nearest-neighbor oriented paths on $\frac{1}{\mathfrak r} \mathbb{Z}^d$ from $\{z : |z|_1 = 1\}$ to $\{z : |z|_1 = M\}$ with length $\mathfrak r (M-1) + 1$, and the cardinality $|\mathscr{P}^{0, M}| = |\mathcal{P}^{0,M}| = d^M$. 

Next, we inductively refine the paths in $\mathscr{P}^{0,M}$ on $\frac{1}{\mathfrak r} 8^{-n} \mathbb{Z}^d$ for each $n \geq 1$. Recall the definition of tubes and refinements from Section~\ref{subsec:BRW-paths} which can be directly extended to the rescaled lattice $\frac{1}{\mathfrak r} 8^{-n} \mathbb{Z}^d$. As mentioned in Remark~\ref{rmk:refine}, the choices of the refinements are the same except that we start with $\mathscr{P}^{0,M}$ which is a set of paths on $\frac{1}{\mathfrak r} \mathbb{Z}^d$. All the properties proved in Section~\ref{subsec:BRW-paths} can be extended to this case, which we will recall later.

As in Definition~\ref{def:Pn-brw}, for each $j \geq 0$, the set $\mathscr{P}^{j + 1, M}$ will consist of self-avoiding paths on $\frac{1}{\mathfrak r}8^{-j - 1} \mathbb{Z}^d$ that are some refinements of the paths in $\mathscr{P}^{j, M}$ as defined below. By Lemma~\ref{lem:property-refine}, all of the paths in $\mathscr{P}^{j, M}$ have the same length $\mathfrak L_j$ defined as follows:
\begin{equation}\label{eq:def-mathfrak-l}
\mathfrak L_0 = \mathfrak r (M - 1) + 1 \quad \mbox{and} \quad \mathfrak L_{i+1} = 10(\mathfrak L_i - 1) + 1 \quad \mbox{for all } i \geq 0.
\end{equation}

\begin{definition}[$\mathscr{P}^{n, M}$]\label{def:Pn-WN}
    Let an integer $M \geq 10$. We assume that $d$ is sufficiently large such that $\mathfrak q \geq 100$. Recall $\mathscr{P}^{0, M}$ from Definition~\ref{def:P0}. Suppose that $\mathscr{P}^{j, M}$ has been defined for some integer $j \geq 0$. For each path $P = (x_1, \ldots, x_{\mathfrak L_j}) \in \mathscr{P}^{j, M}$ on $\frac{1}{\mathfrak r}8^{-j} \mathbb{Z}^d$, the set $\mathscr{P}^{j + 1,M}$ contains the following $(\mathfrak q - 10)^{\mathfrak L_j - 1}$ refinements of $P$:
    \begin{enumerate}
    \item Recall $\mathcal{T}_{x_1,x_2}$ from Lemma~\ref{lem:path-tube}. We deterministically choose a subset of cardinality $(\mathfrak q - 10)$ from $\mathcal{T}_{x_1,x_2}$ and then choose a refinement of $(x_1,x_2)$ from this subset;

    \item Suppose the refinements of $(x_1,x_2),(x_2,x_3),\ldots, (x_{i-1}, x_i)$ have been defined for $i < \mathfrak L_j $. Let $z$ be the last point of the refinement of $(x_{i-1}, x_i)$. Since at most 10 paths in $\mathcal{T}_{x_i, x_{i+1}}^z$ intersect the refinement of $(x_{i-1}, x_i)$ at a point other than $z$, we can deterministically choose a subset of $\mathcal{T}_{x_i, x_{i+1}}^z$ with cardinality $(\mathfrak q - 10)$ such that these paths do not intersect the refinement of $(x_{i-1}, x_i)$ except at $z$, and then choose a refinement of $(x_i,x_{i + 1})$ from this subset;
    
    \item By concatenating the refinements of $(x_i, x_{i+1})$ for $1 \leq i \leq \mathfrak L_j - 1$, we obtain $(\mathfrak q - 10)^{\mathfrak L_j - 1}$ refinements of $P$ on $\frac{1}{\mathfrak r} 8^{-j - 1} \mathbb{Z}^d$, each with length $\mathfrak L_{j+1}$.

\end{enumerate}
By induction, this defines $\mathscr{P}^{n,M}$ for all $n \geq 0$. In addition, for different $M_1 > M_2 \geq 10$, if two paths $P \in  \mathscr{P}^{j,M_1}$ and $Q \in \mathscr{P}^{j,M_2}$ are the same up to length $|Q| = \mathfrak L_j(M_2)$ (the number in~\eqref{eq:def-mathfrak-l} starting from $\mathfrak L_0 = \mathfrak r (M_2-1) + 1$), we require that their refinements in $\mathscr{P}^{j+1,M_1}$ and $\mathscr{P}^{j+1,M_2}$ are also the same up to length $\mathfrak L_{j+1}(M_2)$.
\end{definition}

By Definition~\ref{def:Pn-WN}, we have $|\mathscr{P}^{j + 1,M}| = (\mathfrak q - 10)^{\mathfrak L_j - 1} \times |\mathscr{P}^{j,M}|$. For each path $P \in \mathscr{P}^{n,M}$, by Claim~\ref{lem3.8-claim3} in Lemma~\ref{lem:property-refine}, there exists a unique sequence of paths $P_j \in \mathscr{P}^{j,M}$ for $0 \leq j \leq n$ such that $P_n = P$ and $P_{j+1}$ is a refinement of $P_j$ for all $0 \leq j \leq n-1$. We will often use $P_j$ to denote this sequence of paths, and call $P_j$ the corresponding path of $P$ in $\mathscr{P}^{j,M}$. 

Similar to Lemma~\ref{lem:restrict-M-brw}, the set of paths $\mathscr{P}^{n,M}$ also satisfies a restriction property with respect to $M$.

\begin{lemma}\label{lem:restrict-M-WN}
    For all $M_1 \geq M_2 \geq 10$ and $n \geq 0$, if we restrict the paths in $\mathscr{P}^{n,M_1}$ before length $\mathfrak L_n(M_2)$ (the number in~\eqref{eq:def-mathfrak-l} starting from $\mathfrak L_0 = \mathfrak r(M-1) + 1$), then we obtain $\mathscr{P}^{n,M_2}$.
\end{lemma}

\begin{proof}
    The case $n = 0$ follows from Definition~\ref{def:P0}. The case $n \geq 1$ follows from the case $n = 0$ and Definition~\ref{def:Pn-WN} with an induction argument. The proof is similar to that of Lemma~\ref{lem:restrict-M-brw}, and thus we omit further details.
\end{proof}

The following lemma can be proved in the same way as Lemma~\ref{lem:intersect-BRW}, and thus we omit the proof.

\begin{lemma}\label{lem:intersection-law-WN}
    The following holds for all $n \geq 0$ and $M \geq 10$. For two paths $P, Q$ chosen uniformly from $\mathscr{P}^{n, M}$ and $0 \leq j \leq n$, let $P_j$ and $Q_j$ be their corresponding paths in $\mathscr{P}^{j,M}$ and let $Y_j = |P_j \cap Q_j|$. Then for all $1 \leq i \leq n$, given $(Y_0, Y_1, \ldots Y_{i-1})$, we have
    \begin{equation*}
    Y_i\mbox{ is stochastically dominated by } 11(Z_1 + Z_2 + \ldots + Z_{2 Y_{i-1}}),\end{equation*}
    where $(Z_l)_{l \geq 1}$ are i.i.d.\ Bernoulli random variables with probability $\frac{50}{\mathfrak q - 10}$ of being 1. 
\end{lemma}

The following lemma is proved in the same way as Lemma~\ref{lem:basic-property-mathcalP}. We refer to the refinement of a refinement as a refinement at scale $2$, etc.
\begin{lemma}\label{lem:basic-property-mathscrP}
    The following hold for all $j \geq 0$ and $d \geq 2$:
    \begin{enumerate}

        \item \label{property-1-scr} For an edge $(x,y)$ in $ \frac{1}{\mathfrak r} 8^{-j} \mathbb{Z}^d$, let $P$ be a refinement of $(x,y)$ at any scale. Then we have $$\overline P \subset \{ z : \mathfrak d_1(z, \overline{(x,y)} )  \leq \frac{2}{7 \mathfrak r} 8^{-j} \} \quad \mbox{and} \quad \overline{(x,y)} \subset \{ z : \mathfrak d_1(z, \overline{P}) \leq \frac{2}{7 \mathfrak r} 8^{-j} \}.$$

        \item \label{property-0-src} All of the paths in $\mathscr{P}^{j, M}$ connect $\{z : |z|_1 = 1 + \frac{2}{7 \mathfrak r} \}$ and $\{z : |z|_1 = M - \frac{2}{7 \mathfrak r} \} $.
 
    \end{enumerate}
    
\end{lemma}

\begin{proof}
    Claim~\ref{property-1-scr} directly follows from Claim~\ref{lem3.8-claim4} in Lemma~\ref{lem:property-refine} and the fact that $\sum_{i=j}^\infty \frac{2}{\mathfrak r} 8^{-i-1} = \frac{2}{7 \mathfrak r} 8^{-j}$. Claim~\ref{property-0-src} follows from Claim~\ref{property-1-scr} and the fact that paths in $\mathscr{P}^{0,M}$ connect $\{z:|z|_1 = 1\}$ and $\{z : |z|_1 = M\}$.
\end{proof}

The following lemma can be proved in the same way as Lemma~\ref{lem:Hausdorff-limit-brw}, and thus we omit the proof. Recall that we say $\widetilde P$ is a continuous path if there exists a continuous map $\phi: [0,\infty) \to \mathbb{R}^d$ such that $\widetilde P = \phi([0,\infty))$.

\begin{lemma}\label{lem:Hausdorff-limit-WN}
     Let $\{n_k\}_{k \geq 1}$ be an increasing sequence tending to infinity, and let $\{P_k\}_{k \geq 1}$ be a sequence of paths in $\mathscr{P}^{n_k,n_k}$ that converge to a set $\widetilde P$ with respect to the local Hausdorff distance. Then, $\widetilde P$ is a continuous path in $\mathbb{R}^d$ that connects $\{z:|z|_1 = 2\}$ to infinity.
\end{lemma}

We need the following lemma to estimate Gaussian correlations in the proof of Proposition~\ref{prop:WN-thick}. In Section~\ref{subsec:good-WN}, we will consider the average of white noise over the time interval $[2^{-3j}, 2^{-3j+1}]$ for paths in $\mathscr{P}^{j,M}$ (see~\eqref{eq:def-white-noise-average} there), as we aim to control the field $h_{3j}$ along these paths. This requires us to estimate the left-hand side of~\eqref{eq:lem3.5-1} below since for $x,y \in \mathbb{R}^d$
\begin{align*}
& {\rm Cov}\bigg(\int_{2^{-3j}}^{2^{-3j+1}} \int_{ B_t(x) } \Vol(B_t(0))^{-\frac{1}{2}} t^{-\frac{1}{2}} W(dy,dt), \int_{2^{-3j}}^{2^{-3j+1}} \int_{ B_t(y) } \Vol(B_t(0))^{-\frac{1}{2}} t^{-\frac{1}{2}} W(dy,dt) \bigg) \\
&\qquad = \int_{2^{-3j}}^{2^{-3j + 1}} \frac{\Vol(B_t(x) \cap B_t(y))}{\Vol(B_t(0))} t^{-1} dt \,.
\end{align*}
In fact, we could also consider the average of white noise over the time interval $[2^{-3j}, 2^{-3j+3}]$, and~\eqref{eq:lem3.5-1} would still hold for $[2^{-3j}, 2^{-3j+3}]$, albeit with a larger constant $C_2$. However, focusing on the time interval $[2^{-3j}, 2^{-3j+1}]$ is sufficient for our purposes (specifically, to ensure that Lemma~\ref{lem:3.17} below holds).

\begin{lemma}\label{lem:gaussian-to-intersection}
    There exists a universal constant $C_2 > 0$ such that for all $d \geq 2$, $M \geq 10$, $ n \geq k \geq C_2 \log d$ and for all paths $P, Q \in \mathscr{P}^{n, M}$, we have
        \begin{equation}
        \label{eq:lem3.5-1}
        \sum_{j=k}^{n} \int_{2^{-3j}}^{2^{-3j + 1}} \frac{\Vol(B_t(P_j) \cap B_t(Q_j))}{\Vol(B_t(0))} t^{-1} dt \leq C_2 \sum_{j= \lfloor k / 2 \rfloor }^{n} |P_j \cap Q_j|,
        \end{equation}
    where $P_j$ and $Q_j$ are the corresponding paths of $P$ and $Q$ in $\mathscr{P}^{j, M}$, respectively.
\end{lemma}

\begin{proof}
    Denote the left-hand side of~\eqref{eq:lem3.5-1} by $J$. The constants $C$ in this proof may change from line to line but are universal. The proof consists of four steps.
    
    \textbf{Step 1: Upper-bound $J$ using Claim~\ref{est-correlation-2} in Lemma~\ref{lem:est-correlation}.} For $k \leq j \leq n$, let 
    \begin{equation}\label{eq:def-mathcal-A}
    \mathcal{A}_j := \{ \mbox{vertex }x \in P_j: \mbox{there exists an edge } e \in P_j \mbox{ such that } x \in e \mbox{ and }e \cap Q_j = \emptyset \}.
    \end{equation}
    Then we have 
    \begin{equation}\label{eq:lem3.12-edge}
        |P_j \setminus \mathcal{A}_j| \leq 2 | \{ e : e \in P_j, e \cap Q_j \neq \emptyset\}| \leq 4 |P_j \cap Q_j|.
    \end{equation}(The first inequality follows from the fact that for each $x \in P_j \setminus \mathcal{A}_j$, there exists an edge $e \in P_j$ such that $x$ is an endpoint of $e$ and $e \cap Q_j \neq \emptyset$, and moreover, each edge has two endpoints. The second inequality follows from the fact that each edge in $\{ e : e \in P_j, e \cap Q_j \neq \emptyset\}$ has at least one endpoint in $P_j \cap Q_j$, and since $P_j$ is self-avoiding, each vertex in $P_j \cap Q_j$ can be the endpoint of at most two edges in $P_j$.) Combining~\eqref{eq:lem3.12-edge} with the inequality
    $$
    \Vol (B_t(P_j) \cap B_t(Q_j)) \leq \sum_{x \in \mathcal{A}_j, y \in Q_j} \Vol (B_t(x) \cap B_t(y)) + |P_j \setminus \mathcal{A}_j| \cdot \Vol(B_t(0)),
    $$
    we obtain
    \begin{align*}
        J \leq  \sum_{j=k}^{n} \sum_{x \in \mathcal{A}_j, y \in Q_j} \int_{2^{-3j}}^{2^{-3j + 1}} \frac{\Vol(B_t(x) \cap B_t(y))}{\Vol(B_t(0))} t^{-1} dt + 4 \log 2 \cdot \sum_{j=k}^{n}   |P_j \cap Q_j|.
    \end{align*}
    Applying Claim~\ref{est-correlation-2} in Lemma~\ref{lem:est-correlation} and using $\Vol(B_t(x) \cap B_t(y)) = 0$ if $|x-y|_2 \geq 2t$, we further have
    \begin{equation}
    \label{eq:lem3.5-2}
    \begin{aligned}
        J &\leq \sum_{j=k}^{n} \sum_{x \in \mathcal{A}_j, y \in Q_j} \int_{2^{-3j}}^{2^{-3j + 1}} 1_{\{ |x-y|_2 < 2t \}} \frac{1}{c_3} e^{-c_3 d |x-y|_2^2/ t^2}  t^{-1} dt + 4 \log 2 \cdot \sum_{j=k}^{n}   |P_j \cap Q_j| \\
        &\leq \frac{\log 2}{c_3} \cdot \sum_{j=k}^{n} \sum_{x \in \mathcal{A}_j, y \in Q_j} 1_{\{ |x-y|_2 < 4 \cdot 8^{-j} \}} e^{-\frac{c_3}{4} (\sqrt{d} 8^j |x-y|_2)^2} + 4 \log 2 \cdot \sum_{j=k}^{n}   |P_j \cap Q_j|.
    \end{aligned}
    \end{equation}
    In the second inequality, we used $1_{\{ |x-y|_2 < 2t \}} e^{-c_3 d |x-y|_2^2/ t^2} \leq 1_{\{ |x-y|_2 < 4 \cdot 8^{-j} \}} e^{-\frac{c_3}{4} (\sqrt{d} 8^j |x-y|_2)^2}$ for all $2^{-3j} \leq t \leq 2^{-3j + 1}$.

    \textbf{Step 2: Some properties about $\mathscr{P}^{j, M}$.} In this step, we prove the following:
    \begin{enumerate}[(i)]
    \item \label{lem3.5-step2-p1} 
    For all $j \geq 0$ and any edge $e = (x,y)$ in $\frac{1}{\mathfrak r} 8^{-j} \mathbb{Z}^d$, if $P \in \mathscr{P}^{j, M}$ and $x, y \not \in P$, then all the refinements of $(x,y)$ at any scale have $|\cdot|_2$-distance at least $\frac{3}{7 \mathfrak r} 8^{-j}$ from all the refinements of $P$ at any scale.
    
    \item \label{lem3.5-step2-p2} There exists a universal constant $C>0$ such that for all $j \geq 10 \log d$, $0 \leq r \leq 8 \mathfrak r$, $x \in \mathbb{R}^d$ and for all paths $P \in \mathscr{P}^{j, M}$, we have
    $$
    | P \cap \overline{B_{\frac{r}{\mathfrak r} 8^{-j}}(x)}| \leq C + r^C,
    $$
    where $\overline{B_{\frac{r}{\mathfrak r} 8^{-j}}(x)}$ is the closure of the ball.
    \end{enumerate}
    By Claim~\ref{property-1-scr} in Lemma~\ref{lem:basic-property-mathscrP} and the fact that $\mathfrak d_2(\overline e, \overline P) \geq \frac{1}{\mathfrak r} 8^{-j}$, the $|\cdot|_2$-distance between any two refinements of $e$ and $P$ is at least $\frac{1}{\mathfrak r} 8^{-j} - 2 \times \frac{2}{7 \mathfrak r} 8^{-j} = \frac{3}{7 \mathfrak r} 8^{-j}$, which yields Claim~\eqref{lem3.5-step2-p1}.
    
    Next, we prove Claim~\eqref{lem3.5-step2-p2}. See Figure~\ref{fig:edge-number} for an illustration. For $P \in \mathscr{P}^{j, M}$, let $P_1, P_2, \ldots, P_j$ be its corresponding paths in $\mathscr{P}^{1, M}, \mathscr{P}^{2, M}, \ldots, \mathscr{P}^{j, M}$. Let $s \leq j - 1$ be an integer to be chosen. For $z \in P \cap \overline{B_{\frac{r}{\mathfrak r} 8^{-j}}(x)}$, let $e_z$ be the edge in $P_{j-s}$ such that $z$ belongs to a refinement of $e_z$ at scale $s$. There are at most two choices for $e_z$ and when $e_z$ is not unique we deterministically choose one of them arbitrarily. By Claim~\ref{property-1-scr} in Lemma~\ref{lem:basic-property-mathscrP}, we have $\mathfrak d_2(z, \overline{e_z}) \leq \mathfrak d_1(z, \overline{e_z}) \leq \frac{2}{7 \mathfrak r} 8^{-j+s}$. Then, $\mathfrak d_2(x, \overline{e_z}) \leq |x-z|_2 + \mathfrak d_2(z, \overline{e_z}) \leq \frac{r}{\mathfrak r} 8^{-j} + \frac{2}{7 \mathfrak r} 8^{-j+s}$. We choose $s$ to be the smallest non-negative integer such that $r < \frac{3}{14} 8^s$. Since $j \geq 10 \log d$ and $r \leq 8 \mathfrak r$, we have $s \leq j-1$. In this case, $\mathfrak d_\infty(x, \overline{e_z}) \leq \mathfrak d_2(x, \overline{e_z}) < \frac{3}{14 \mathfrak r} 8^{-j + s} + \frac{2}{7 \mathfrak r} 8^{-j+s} = \frac{1}{2 \mathfrak r}8^{-j+s}$. Indeed, there are at most two edges in $P_{j-s}$ that satisfy the condition $\mathfrak d_\infty(x, \overline{e_z}) < \frac{1}{2 \mathfrak r}8^{-j+s}$, as this condition is equivalent to that there exists a unique $u \in \frac{1}{\mathfrak r}8^{-j+s} \mathbb{Z}^d$ such that $\mathfrak d_\infty(x,u) < \frac{1}{2 \mathfrak r}8^{-j+s}$ and $e_z$ is an edge with one endpoint being $u$. Moreover, since $P_{j-s}$ is a self-avoiding path, at most two edges in $P_{j-s}$ can have $u$ as an endpoint. Therefore, for all $z \in P \cap \overline{B_{\frac{r}{\mathfrak r} 8^{-j}}(x)}$, $e_z$ has at most two choices. Since any refinement of a fixed $e_z$ contains at most $11^s$ vertices at scale $s$, we conclude that $| P \cap \overline{B_{\frac{r}{\mathfrak r} 8^{-j}}(x)}| \leq 2 \times 11^s$. This proves Claim~\eqref{lem3.5-step2-p2}.

    \begin{figure}[h]
\centering
\includegraphics[scale=0.8]{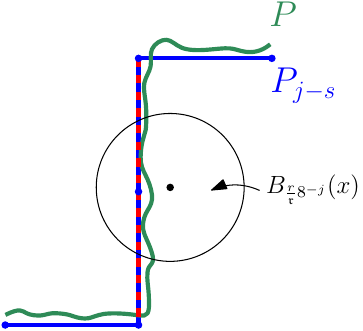}
\caption{The green path $P$ is defined on $\frac{1}{\mathfrak r} 8^{-j} \mathbb{Z}^d$, and the blue path $P_{j-s}$ is defined on $\frac{1}{\mathfrak r} 8^{-j+s} \mathbb{Z}^d$. If we choose the integer $s$ such that $r < \frac{3}{14} 8^s$, then all the vertices in $P \cap \overline{B_{\frac{r}{\mathfrak r} 8^{-j}}(x)}$ belong to the refinements of at most two edges in $P_{j-s}$ at scale $s$, as shown by the dashed red edges in the figure. This gives an upper bound on $|P \cap \overline{B_{\frac{r}{\mathfrak r} 8^{-j}}(x)}|$.}
\label{fig:edge-number}
\end{figure}

    \textbf{Step 3. Estimate the contribution of one edge.} In this step, we show that there exists a universal constant $C>0$ such that the following holds for all $j \geq 10 \log d$, all edges $e_j$ in $\frac{1}{\mathfrak r} 8^{-j} \mathbb{Z}^d$, and for all paths $Q_j \in \mathscr{P}^{j, M}$ with $e_j \cap Q_j = \emptyset$. For $i \geq 1$, let the path $Q_{j+i}$ be an arbitrary refinement of $Q_j$ at scale $i$, and let the path $e_{j+i}$ be an arbitrary refinement of $e_j$ at scale $i$. Then, we claim that
    \begin{equation}
        \label{eq:lem3.5-one-edge}
        \sum_{s=j}^\infty \sum_{x \in e_s, y \in Q_s} 1_{\{ |x-y|_2 < 4 \cdot 8^{-s} \}} e^{-\frac{c_3}{4} (\sqrt{d} 8^s |x-y|_2)^2} \leq C,
    \end{equation}
    where $x$ and $y$ range over all vertices in the paths $e_s$ and $Q_s$, respectively. The claim as in~\eqref{eq:lem3.5-one-edge} follows from Claims~\eqref{lem3.5-step2-p1} and \eqref{lem3.5-step2-p2}, as we now elaborate. For all $x \in \mathbb{R}^d$, we can cover $B_{4 \cdot 8^{-s}}(x)$ by $\overline{B_{\frac{1}{\mathfrak r} 8^{-s}}(x)}$ and $ \overline{B_{\frac{2^p}{\mathfrak r} 8^{-s}}(x)} \setminus B_{\frac{2^{p-1}}{\mathfrak r} 8^{-s}}(x)$ for $1 \leq p \leq \lfloor \log_2 \mathfrak r \rfloor + 3$. Therefore, for all $s \geq j$ and $x \in e_s$, we have
    $$
    \sum_{y \in Q_s} 1_{\{ |x-y|_2 < 4 \cdot 8^{-s} \}} e^{-\frac{c_3}{4} (\sqrt{d} 8^s |x-y|_2)^2} \leq |Q_s \cap \overline{B_{\frac{1}{\mathfrak r} 8^{-s}}(x)}| + \sum_{p=1}^{\lfloor \log_2 \mathfrak r \rfloor + 3} |Q_s \cap \overline{B_{\frac{2^p}{\mathfrak r} 8^{-s}}(x)}| \times e^{-\frac{c_3}{4} (\sqrt{d} 8^s \cdot \frac{2^{p-1}}{\mathfrak r} 8^{-s})^2}.
    $$
    By Claim~\eqref{lem3.5-step2-p1}, we have $\mathfrak d_2 (\overline{e_s}, \overline{Q_s}) \geq \frac{3}{7 \mathfrak r} 8^{-j}$, and hence $|Q_s \cap \overline{B_{\frac{2^p}{\mathfrak r} 8^{-s}}(x)}| = 0$ for all $p \leq 3(s-j)-2$. Combined with Claim~\eqref{lem3.5-step2-p2} and $\mathfrak r = \lfloor \sqrt{d} \rfloor$, this yields that
    $$
    \sum_{y \in Q_s} 1_{\{ |x-y|_2 < 4 \cdot 8^{-s} \}} e^{-\frac{c_3}{4} (\sqrt{d} 8^s |x-y|_2)^2} \leq \sum_{p=3(s-j)-1}^{\lfloor \log_2 \mathfrak r \rfloor + 3} (C + 2^{Cp}) \times e^{-2^{2p}/C}.
    $$
    Since the length of $e_s$ is at most $2 \times 11^{s - j}$, we further have (write $s' = s-j$ below)
    \begin{align*}
        &\quad \sum_{s=j}^\infty \sum_{x \in e_s, y \in Q_s} 1_{\{ |x-y|_2 < 4 \cdot 8^{-s} \}} e^{-\frac{c_3}{4} (\sqrt{d} 8^s |x-y|_2)^2} \leq \sum_{s = j}^\infty 2 \times 11^{s - j} \sum_{p=3(s-j)-1}^{\infty} (C + 2^{Cp}) \times e^{-2^{2p}/C}\\
        &= \sum_{p=-1}^\infty (C + 2^{Cp}) \times e^{-2^{2p}/C} \sum_{s'=0}^{\lfloor (p+1)/3 \rfloor } 2 \times 11^{s'} \leq \sum_{p=-1}^\infty (C + 2^{Cp}) \times e^{-2^{2p}/C} \times (C + e^{Cp}) \leq C.
    \end{align*}

    \textbf{Step 4: Proof of~\eqref{eq:lem3.5-1}.} Finally, we prove~\eqref{eq:lem3.5-1} using~\eqref{eq:lem3.5-2} and \eqref{eq:lem3.5-one-edge}. Assume that $k \geq 100 \log d$ and $j \geq k$. Consider the summand term $\sum_{y \in Q_j} 1_{\{ |x-y|_2 < 4 \cdot 8^{-j} \}} e^{-\frac{c_3}{4} (\sqrt{d} 8^j |x-y|_2)^2}$ on the right-hand side of~\eqref{eq:lem3.5-2}. For each vertex $x \in \mathcal{A}_j$ with $\mathfrak d_2 (x, \overline{Q_j}) \geq 4 \cdot 8^{-j}$, the summand term is 0. Next we consider $x \in \mathcal{A}_j$ such that $\mathfrak d_2 (x, \overline{Q_j}) < 4 \cdot 8^{-j}$. For all $0 \leq i \leq j-1$, let $\widetilde e_i$ be the edge in $P_i$ whose refinement at scale $(j-i)$ contains $x$, and let $\widetilde e_j$ be the edge in $P_j$ that contains $x$. If there are multiple choices (at most two), we deterministically choose one of them. We can also require that $\widetilde e_j \cap Q_j = \emptyset$ by~\eqref{eq:def-mathcal-A} and that $\widetilde e_{i+1}$ belongs to a refinement of $\widetilde e_i$. Let $s \leq j-1$ be an integer to be chosen. By Claim~\eqref{lem3.5-step2-p1}, if $\widetilde e_{j-s} \cap Q_{j-s} = \emptyset$, then $\mathfrak d_2 (\overline{\widetilde e_j}, \overline{Q_j}) \geq \frac{3}{7 \mathfrak r} 8^{-j + s}$. However, $\mathfrak d_2 (\overline{\widetilde e_j}, \overline{Q_j}) \leq \mathfrak d_2 (x, \overline{Q_j}) < 4  \cdot 8^{-j}$. Therefore, when $\frac{3}{7 \mathfrak r} 8^{-j + s} > 4  \cdot 8^{-j}$ (or equivalently, $8^s > \frac{28}{3} \mathfrak r$), we have $\widetilde e_{j-s} \cap Q_{j-s} \neq \emptyset$. Let $s_x$ be the largest integer $s$ such that $\widetilde e_{j-s} \cap Q_{j-s} = \emptyset$. 
    Then we have $0 \leq s_x \leq 10 \log d$. Let $\mathcal{E}$ denote the set of all different choices of $\widetilde e_{j-s_x}$ for all $j \in [k, n] \cap \mathbb{Z}$ and $x \in \mathcal{A}_j$ with $\mathfrak d_2 (x, \overline{Q_j}) < 4 \cdot 8^{-j}$. By~\eqref{eq:lem3.5-one-edge}, we have
    \begin{equation}\label{eq:lem3.5-3}
    \sum_{j=k}^{n} \sum_{x \in \mathcal{A}_j, y \in Q_j} 1_{\{ |x-y|_2 < 4 \cdot 8^{-j} \}} e^{-\frac{c_3}{4} (\sqrt{d} 8^j |x-y|_2)^2} \leq C | \mathcal{E}| \quad \mbox{with the constant }C \mbox{ from }\eqref{eq:lem3.5-one-edge}.
    \end{equation}
    Since all the edges in $\mathcal{E}$ are contained in a refinement of an edge in $P_l$ that intersects $Q_l$ for some $l \in [k - \lfloor 10 \log d \rfloor - 1, n] \cap \mathbb{Z}$, and each edge is refined into 10 edges, we have $$|\mathcal{E}| \leq 10 \sum_{l = k - \lfloor 10 \log d \rfloor - 1}^n |\{e: e\in P_l, e\cap Q_l \neq \emptyset\}| \leq 20 \sum_{l = k - \lfloor 10 \log d \rfloor - 1}^n |P_l \cap Q_l|,$$ (see~\eqref{eq:lem3.12-edge} for the proof of the second inequality). This, together with~\eqref{eq:lem3.5-2} and~\eqref{eq:lem3.5-3}, yields~\eqref{eq:lem3.5-1}. \qedhere
    
\end{proof}

\subsection{Good paths}\label{subsec:good-WN}

Now we give the definition of good paths in $\mathscr{P}^{n,M}$ which depends on several parameters $\alpha, \beta, d, n, k, M$. As mentioned at the beginning of Section~\ref{sec:WN}, we need to consider an auxiliary condition.

\begin{definition}\label{def:good-WN}
    Fix $\alpha>0$. Let $\beta>0$ be a large constant to be chosen which depends only on $\alpha$ and is independent of $d$. Let integers $k \geq 10$, $n \geq 6k$ and $M \geq 10$. For a path $P$ in $\mathscr{P}^{n, M}$ and $0 \leq j \leq n$, let $P_j$ be its corresponding path in $\mathscr{P}^{j, M}$. We also let $\widetilde P_0$ be its corresponding path in $\mathcal{P}^{0,M}$. We call the path $P$ \textbf{good} if the following two conditions hold for all vertices $x \in \widetilde P_0$:
    \begin{enumerate}[(a)]
        \item \label{condition-good-1} For $j \in [k, n] \cap \mathbb{Z}$, let the average of white noise along $P_j$ be
        \begin{equation}
        \label{eq:def-white-noise-average}
        \mathscr{W}_j(x, P_j) := \int_{2^{-3j}}^{2^{-3j+1}} \int_{ B_t(P_j) \cap (x + [-\frac{1}{2}, \frac{1}{2}]^d) } \Vol(B_t(0))^{-\frac{1}{2}} t^{-\frac{1}{2}} W(dy,dt).
        \end{equation}
        Then we have $\mathscr{W}_j (x, P_j) \geq \beta \mathbb{E}[\mathscr{W}_j (x, P_j)^2]$ for all $j \in [k, n] \cap \mathbb{Z}$. Denote this event by $\mathcal{G}_1(x,P)$.

        \item \label{condition-good-2} For all $j \in [6k, n] \cap \mathbb{Z}$ and $y \in B_{\frac{2}{\mathfrak r} 8^{-j}}(P_j) \cap (x + [-\frac{1}{2}, \frac{1}{2}]^d)$, we have 
        \begin{equation}\label{eq:def-good-2}
        h_j(y) - h_{3k}(y) \geq \alpha (j-3k).
        \end{equation}
        Denote this event by $\mathcal{G}_2(x,P)$.
    \end{enumerate}
\end{definition}

\begin{remark}\label{rmk:3.16}
    \begin{enumerate}
        \item We now show that Condition~\eqref{condition-good-2} implies ~\eqref{eq:prop-WN-1} in Proposition~\ref{prop:WN-thick} for the same $n, k, M$. For all $j \in [6k, n] \cap \mathbb{Z}$ and $j \leq i \leq n$, by Claim~\ref{property-1-scr} in Lemma~\ref{lem:basic-property-mathscrP}, we have $\overline{P_i} \subset B_{\frac{2}{7 \mathfrak r} 8^{-j}}(\overline{P_j})$. By definition, $\overline{P_j} \subset B_{\frac{1}{\mathfrak r} 8^{-j}}(P_j)$. Therefore, we have $\cup_{j \leq i \leq n} \overline{P_i} \subset B_{\frac{2}{\mathfrak r} 8^{-j}}(P_j)$. Notice that $\overline{P_j} \subset B_{\frac{2}{7 \mathfrak r}}(\overline{P_0}) = B_{\frac{2}{7 \mathfrak r}}(\overline{\widetilde{P}_0})$, which implies that $\cup_{x \in \widetilde P_0} (x + [-\frac12,\frac12]^d)$ covers $B_{\frac{2}{\mathfrak r} 8^{-j}}(P_j)$. Hence, Condition~\eqref{condition-good-2} implies~\eqref{eq:prop-WN-1}.\label{rmk3.16-claim-1}

        \item By definition, the event $\mathcal{G}_1(x,P)$ only depends on the white noise in $(x + [-\frac12, \frac12]^d)$. Let $x_-, x_+$ be the neighboring vertices of $x$ in $\widetilde P_0$. Then, the event $\mathcal{G}_2(x, P)$ only depends on the white noise in $\cup_{y \in \{x, x_-, x_+ \}} (y + [-\frac12, \frac12]^d)$. These properties will be crucial in the second moment estimates. \label{rmk3.16-claim-2}
    \end{enumerate}
\end{remark}
In order to prove Proposition~\ref{prop:WN-thick}, we study the weighted sum
\begin{equation}\label{eq:def-N}
\mathscr{N} := \sum_{P \in \mathscr{P}^{n, M}} \frac{1}{\mathbb{P}[ P \mbox{ is }\mbox{good}]} \mathbbm{1}\{P \mbox{ is }\mbox{good}\}.
\end{equation}
By Cauchy–Schwarz inequality $\mathbb{P}[\mathscr{N} > 0] \geq \frac{(\mathbb{E}\mathscr{N})^2}{\mathbb{E} \mathscr{N}^2}$, it suffices to show that the ratio is close to 1. By simple calculations, we have 
$$
\frac{\mathbb{E} \mathscr{N}^2 }{(\mathbb{E} \mathscr{N})^2 } = \widetilde{\mathbf E} \Big[\frac{\mathbb{P}[\mbox{$P,Q$ are good}]}{\mathbb{P}[\mbox{$P$ is good}] \cdot \mathbb{P}[\mbox{$Q$ is good}]}\Big],
$$
where $\widetilde{\mathbf E}$ denotes the expectation with respect to two paths $P$ and $Q$ uniformly chosen from $\mathscr{P}^{n, M}$, and $\mathbb{P}$ is defined with respect to the space-time white noise, which is independent of $P,Q$. If $P,Q$ are far away from each other, the right-hand side is close to 1 as the two events are approximately independent. The main difficulty is to estimate the correlation in the case where they are close. In particular, the probability of Condition~\eqref{condition-good-2} is hard to evaluate explicitly. Our strategy is to first show that by choosing $\beta$ sufficiently large (depending only on $\alpha$), Condition~\eqref{condition-good-2} can be implied by Condition~\eqref{condition-good-1} with high probability. Then, it suffices to estimate the correlation of Condition~\eqref{condition-good-1}. Using Lemmas~\ref{lem:gaussian-correlation} and~\ref{lem:gaussian-to-intersection}, we can upper-bound the Gaussian correlation using intersection numbers of $P$ and $Q$ at different scales, which can be controlled using Lemma~\ref{lem:intersection-law-WN}. 

\subsection{Condition~\eqref{condition-good-1} implies Condition~\eqref{condition-good-2} with high probability}\label{subsec:conditional}

In this subsection, we prove the following lemma.

\begin{lemma}\label{lem:condition-WN}
    For fixed $\alpha>0$, the following holds for all sufficiently large $\beta$. For all $d \geq \beta$, there exists a constant $C$ depending on $\alpha, \beta, d$ such that for all $k \geq C$, $n \geq 6k$, $M \geq 10$, and for all paths $P\in \mathscr{P}^{n,M}$ and $x \in \widetilde P_0$, we have
    \begin{equation}\label{eq:lem3.14-1}
    \mathbb{P}[\mathcal{G}_2(x,P) | \cap_{w \in \widetilde P_0} \mathcal{G}_1(w, P)] \geq 1 - e^{-k}.
    \end{equation}
\end{lemma}

We assume that $k \geq 100 d$. It suffices to show that for all $j \in [6k, n] \cap \mathbb{Z}$ and $y \in P_j$, we have
    \begin{equation}\label{eq:lem3.14-2}
    \mathbb{P}\Big[\inf_{z \in B_{\frac{2}{\mathfrak r} 8^{-j}}(y)} h_j(z) - h_{3k}(z) \geq \alpha (j - 3k)| \cap_{w \in \widetilde P_0} \mathcal{G}_1(w, P) \Big] \geq 1 - e^{-100j}.
    \end{equation}
    Recall Condition~\eqref{condition-good-2}. For each $x \in \widetilde P_0$, we claim that there are at most $3 \mathfrak r \times 10^j$ number of vertices $y$ in $P_j$ such that $B_{\frac{2}{\mathfrak r} 8^{-j}}(y)$ intersects $x + [-\frac12, \frac12]^d$. To prove the claim, let $x_-$ and $x_+$ be the neighboring vertices of $x$ in $\widetilde P_0$. By Claim~\ref{property-1-scr} in Lemma~\ref{lem:basic-property-mathscrP}, such $y$ must belong to the refinement at the $j$-th scale of an edge in $P_0$ that interpolates $(x_-, x)$ or $(x, x_+)$. In addition, by Definition~\ref{def:P0} and Claim~\ref{lem3.8-claim2} in Lemma~\ref{lem:property-refine}, there are $2 \mathfrak r$ such edges in $P_0$ and each of them is refined to $10^j + 1$ vertices in $P_j$. Finally, note that $2 \mathfrak r \times (10^j + 1) \leq 3 \mathfrak r \times 10^j$. This proves the claim. Then, \eqref{eq:lem3.14-1}~follows from~\eqref{eq:lem3.14-2} by taking a union bound over the choices of $y$ and $j$.

    In order to prove~\eqref{eq:lem3.14-2}, we separate the white noise field $(h_j(z) - h_{3k}(z))_{z \in B_{\frac{2}{\mathfrak r} 8^{-j}}(y)}$ into two parts: $(Y(z))_{z \in B_{\frac{2}{\mathfrak r} 8^{-j}}(y)}$ and a common part $Z$ as defined below. Fix $y \in P_j$. For $t>0$, let
    \begin{equation}\label{eq:def-Z}
    A_t = \bigcap_{z \in B_{\frac{2}{\mathfrak r} 8^{-j}}(y)} B_t(z) \quad \mbox{and} \quad Z = \sum_{i=k + 1}^{\lfloor j/3 \rfloor} \int_{2^{-3i}}^{2^{-3i+1}} \int_{A_t}\Vol(B_t(0))^{-\frac{1}{2}} t^{-\frac{1}{2}} W(dy,dt).
    \end{equation}
    In particular, $A_t = B_{t - \frac{2}{\mathfrak r} 8^{-j}}(y)$ if $t \geq \frac{2}{\mathfrak r} 8^{-j}$, and $A_t = \emptyset$ if $t < \frac{2}{\mathfrak r} 8^{-j}$. Let $Y(z) = h_j(z) - h_{3k}(z) - Z$ for $z \in B_{\frac{2}{\mathfrak r} 8^{-j}}(y)$. By~\eqref{eq:def-WN} and the identity $h_j(z) - h_{3k}(z) = \sum_{i=3k+1}^j h_i(z) - h_{i-1}(z)$, we see that $Y(z)$ equals
    \begin{align*}
    \sum_{i=k + 1}^{\lfloor j/3 \rfloor} \int_{2^{-3i}}^{2^{-3i+1}} \int_{B_t(z) \setminus A_t}\Vol(B_t(0))^{-\frac{1}{2}} t^{-\frac{1}{2}} W(dy,dt) + \sum_{\substack{3k + 1 \leq i \leq j\\ 3 \nmid i}} \int_{2^{-i}}^{2^{-i+1}} \int_{B_t(z)}\Vol(B_t(0))^{-\frac{1}{2}} t^{-\frac{1}{2}} W(dy,dt).
    \end{align*}
    Here for $i \in \mathbb{Z}$, $3 \nmid i$ means that $i$ is not divisible by 3. From this expression, we see that $(Y(z))$ is a positively correlated Gaussian field. We first use standard Gaussian process estimates to lower-bound $(Y(z))$. Note that the event $\cap_{w \in \widetilde P_0} \mathcal{G}_1(w, P)$ is an increasing event with respect to the white noise, so by the FKG inequality (see e.g.~\cite{pitt-positively-correlated}) the inequality~\eqref{eq:lem3.14-3} below also holds when conditioned on the event $\cap_{w \in \widetilde P_0} \mathcal{G}_1(w, P)$.

    \begin{lemma}\label{lem:3.16}
        There exists a universal constant $K>0$ such that the following holds. For all $d \geq 2$, there exists a constant $C = C(d)$ such that for all $k \geq C$ and $j \geq 6k$, we have
        \begin{equation}\label{eq:lem3.14-3}
        \mathbb{P}\Big[\inf_{z \in B_{\frac{2}{\mathfrak r} 8^{-j}}(y)} Y(z) \geq - K (j - 3k) \Big] \geq 1 - e^{-200j} \quad \mbox{for all $y \in \mathbb{R}^d$}.
        \end{equation}
    \end{lemma}
    \begin{proof}
        By~\eqref{eq:def-WN} and the calculations in~\cite[Lemma 2.3]{dgz-exponential-metric}, we have
    $$
    \mathbb{E}[Y(y)^2]  \leq \mathbb{E}[(h_j(y) - h_{3k}(y))^2]  =  (j- 3k ) \log 2.
    $$
    Therefore, we can choose $K$ sufficiently large such that for all $k \geq 1$ and $j \geq 6k$,
    \begin{equation}\label{eq:lem3.16-1}
         \mathbb{P}\Big[ Y(y) \geq - \frac{K}{2} (j - 3k) \Big] \geq 1 - \frac{1}{2} e^{-200j}.
    \end{equation}
    Next, we upper-bound the field $(Y(y) - Y(z))_{z \in B_{\frac{2}{\mathfrak r} 8^{-j}}(y)}$. By definition, we have $\mathbb{E}[(Y(z_1) - Y(z_2))^2] = \mathbb{E}[((h_j(z_1) - h_{3k}(z_1)) - (h_j(z_2) - h_{3k}(z_2)))^2 ]$. For $z_1,z_2 \in B_{\frac{2}{\mathfrak r} 8^{-j}}(y)$, we have
    \begin{align*}
        &\quad \mathbb{E}\left[ (h_j(z_1) - h_{3k}(z_1)) (h_j(z_2) - h_{3k}(z_2)) \right] \\
        &= \mathbb{E} \Big[ \int_{2^{-j}}^{2^{-3k}} \int_{B_t(z_1)} {\rm vol}(B_1(0))^{-\frac{1}{2}} t^{-\frac{d+1}{2}} W(dw,dt) \cdot \int_{2^{-j}}^{2^{-3k}} \int_{B_t(z_2)} {\rm vol}(B_1(0))^{-\frac{1}{2}} t^{-\frac{d+1}{2}} W(dw,dt) \Big] \\
        &= \int_{2^{-j}}^{2^{-3k}} t^{-1} \frac{{\rm vol}(B_t(z_1) \cap B_t(z_2))}{{\rm vol}(B_t(0))} dt.
    \end{align*}
    There exists a constant $C'$ depending only on $d$ such that $\frac{{\rm vol}(B_t(z_1) \cap B_t(z_2))}{{\rm vol}(B_t(0))} \geq 1 - C' \frac{|z_1 - z_2|_2}{t}$. This, combined with the above equation, yields that 
    \begin{equation}\label{eq:lem3.16-2}
    \begin{aligned}
    \mathbb{E}[(Y(z_1) - Y(z_2))^2]& = \mathbb{E}[((h_j(z_1) - h_{3k}(z_1)) - (h_j(z_2) - h_{3k}(z_2)))^2 ]\\
    &= 2 (j-3k)\log 2 - 2 \mathbb{E}\left[ (h_j(z_1) - h_{3k}(z_1)) (h_j(z_2) - h_{3k}(z_2)) \right] \\
    &\leq 2C' \int_{2^{-j}}^{2^{-3k}}  \frac{|z_1 - z_2|_2}{t^2} dt \leq C' 2^{j+1}|z_1 - z_2|_2.
    \end{aligned}
    \end{equation}
    By Dudley's inequality (see e.g.~Equation (2.38) of~\cite{Talagrand-book14}), we have 
    \begin{equation}\label{eq:lem3.16-3}
    \mathbb{E}\Big[\sup_{z \in B_{\frac{2}{\mathfrak r} 8^{-j}}(y)} Y(y) - Y(z)\Big] \leq C'' \quad \mbox{for some constant $C''$ depending only on $d$.}
    \end{equation}
    This is because by Dudley's inequality, $\mathbb{E}[\sup_{z \in B_{\frac{2}{\mathfrak r} 8^{-j}}(y)} Y(y) - Y(z)] \leq A \int_0^\infty \sqrt{\log N(\epsilon)} d\epsilon$ for some universal constant $A>0$, where $N(\epsilon)$ is the minimum number of $\epsilon$-balls that are needed to cover the space $B_{\frac{2}{\mathfrak r} 8^{-j}}(y)$ with the metric $D(z_1,z_2) = \sqrt{C' 2^{j+1}|z_1 - z_2|_2}$. It is easy to see that $N(\epsilon) \leq A' (2^j \epsilon)^{-2d}$ for some constant $A'$ depending only on $d$ if $\epsilon \leq 2^{-j} \sqrt{8C'/\mathfrak r}$, and that $N(\epsilon) = 1$ if $\epsilon > 2^{-j} \sqrt{8C'/\mathfrak r}$. This yields~\eqref{eq:lem3.16-3}. 
    
    Combining~\eqref{eq:lem3.16-3} with the Borel-TIS inequality (see e.g.~\cite[Theorem 2.1.1]{adler-taylor-fields}) and the fact that $\mathbb{E}[(Y(y) - Y(z))^2] \leq 1$ for all sufficiently large $j$, we can enlarge $K$ such that for all sufficiently large $k$ (which depends on $d$) and $j \geq 6k$,
    \begin{equation}\label{eq:lem3.16-4}
        \mathbb{P}\Big[\sup_{z \in B_{\frac{2}{\mathfrak r} 8^{-j}}(y)} Y(y) - Y(z) \leq \frac{K}{2}(j- 3k)\Big] \geq 1 - \frac{1}{2} e^{-200j}.
    \end{equation}
    Combining~\eqref{eq:lem3.16-1} and~\eqref{eq:lem3.16-4} yields~\eqref{eq:lem3.14-3}.
    \end{proof}

    Next, we show that conditioned on the event $\cap_{w \in \widetilde P_0} \mathcal{G}_1(w, P)$, the random variable $Z$ (recall~\eqref{eq:def-Z}) is large with very high probability. The proof essentially follows from the entropic repulsion result in Lemma~\ref{lem:repulsion}.

    \begin{lemma}\label{lem:3.17}
        For any fixed $\lambda>0$, the following holds for all sufficiently large $\beta$. For all $d \geq \beta$, there exists a constant $C>0$, that depends on $\lambda, \beta, d$, such that for all $k \geq C$, $n \geq 6k$, $M\geq 10$, and $j \geq 6k$, and for all paths $P \in \mathscr{P}^{n,M}$, we have
        \begin{equation}\label{eq:lem3.14-4}
            \mathbb{P}\Big[Z \geq \lambda (j - 3k)| \cap_{w \in \widetilde P_0} \mathcal{G}_1(w, P) \Big] \geq 1 - e^{-200j} \quad \mbox{for all $y \in P_j$}.
        \end{equation}
    \end{lemma}

    \begin{proof}
        Fix $y \in P_j$. Recall from~\eqref{eq:def-Z} the definition of $A_t$ and $Z$, and also recall from Condition~\eqref{condition-good-1} the white noise average $\mathscr{W}_j(x, P_j)$. For $k + 1 \leq i \leq \lfloor j/3 \rfloor$, let
        \begin{equation}\label{eq:decompose-Zi}
            Z_i:=\int_{2^{-3i}}^{2^{-3i+1}} \int_{A_t}\Vol(B_t(0))^{-\frac{1}{2}} t^{-\frac{1}{2}} W(dy,dt).
        \end{equation}
        We will show that the correlation between $Z_i$ and $\mathscr{W}_i(x, P_i)$ is uniformly bounded from below for some $x \in \widetilde P_0$ (see~\eqref{eq:lem3.17-1} below). Then, by Lemma~\ref{lem:repulsion}, conditioned on $\mathscr{W}_i(x, P_i) \geq \beta \mathbb{E}[\mathscr{W}_i(x, P_i)^2]$, the random variable $Z_i$ is pushed to take on a large value, which grows to infinity with $\beta$. By choosing a sufficiently large $\beta$, we can then prove~\eqref{eq:lem3.14-4}.

        Fix $i \in [k + 1, \lfloor j/3 \rfloor - 10 d] \cap \mathbb{Z}$. By Claim~\ref{property-1-scr} in Lemma~\ref{lem:basic-property-mathscrP}, we have $y \in B_{\frac{2}{7\mathfrak r} 8^{-i}}(\overline{P_i})$. Therefore, there exists a vertex $u \in P_i$ such that $|u-y|_2 \leq \frac{2}{7\mathfrak r} 8^{-i} + \frac{1}{\mathfrak r} 8^{-i} < \frac{2}{\mathfrak r} 8^{-i}$. We choose a vertex $x_u \in \widetilde P_0$ such that $u \in (x_u + [-\frac{1}{2}, \frac{1}{2}]^d)$. For simplicity, we write $\widehat B_t(u) = B_t(u) \cap (x_u + [-\frac12, \frac12]^d)$. By definition, we have
        \begin{equation}\label{eq:lem3.17-correlation}
        \mathbb{E}[Z_i \mathscr{W}_i(x_u, P_i)] \geq \int_{2^{-3i}}^{2^{-3i+1}} t^{-1} \frac{{\rm vol}(A_t \cap \widehat B_t(u))}{{\rm vol}(B_t(0))} dt.
        \end{equation}

        Next, we lower-bound the right-hand side of~\eqref{eq:lem3.17-correlation}. We first show that for all sufficiently large $d$,
        \begin{equation}\label{eq:lem3.17-claim-1}
        \frac{{\rm vol}(B_t(y) \cap \widehat B_t(u))}{{\rm vol}(B_t(0))}\mbox{ is bounded from below by a universal constant for all $2^{-3i} \leq t \leq 2^{-3i+1}$.}
        \end{equation}
        When $B_t(u), B_t(y) \subset x_u + [-\frac12, \frac12]^d$, Claim~\eqref{eq:lem3.17-claim-1} follows from Claim~\ref{est-correlation-1} in Lemma~\ref{lem:est-correlation}. Now we explain how to address the case where $B_t(u)$ or $B_t(y)$ is not contained in $x_u + [-\frac12, \frac12]^d$. We know from Claim~\ref{property-1-scr} in Lemma~\ref{lem:basic-property-mathscrP} that $u \in B_{\frac{2}{7 \mathfrak r}} ( \overline{\widetilde P_0})$. In addition, $\overline{\widetilde P_0}$ has $|\cdot|_\infty$-distance $\frac{1}{2}$ from the edges of the hypercube $x_u + [-\frac12, \frac12]^d$, which implies that $u$ is bounded away from the edges of $x_u + [-\frac12, \frac12]^d$. Therefore, we can find a half-space $\mathbb{H}$ with $u$ on its boundary such that ${\rm vol}(B_t(y) \cap \widehat B_t(u)) \geq {\rm vol}(B_t(u) \cap B_t(y) \cap \mathbb{H})$; see Figure~\ref{fig:area}. Next, we lower-bound ${\rm vol}(B_t(u) \cap B_t(y) \cap \mathbb{H})$ for all $|u-y|_2 < \frac{2}{\mathfrak r} 8^{-i}$ and $2^{-3i} \leq t \leq 2^{-3i+1}$. Without loss of generality, assume that $u = (0,0, \ldots, 0)$ and $\mathbb{H} = \{z : z_2 \geq 0\}$. Fix $2^{-3i} \leq t \leq 2^{-3i+1}$. By rotational invariance, we can also assume that $y = (a,b,0,\ldots,0)$ for some $a,b$. Since $|u-y|_2 < \frac{2}{\mathfrak r} 8^{-i} \leq \frac{2}{\mathfrak r}t$, we have $|a|/t , |b|/t < \frac{2}{\mathfrak r}$. We consider the following region $$\{z : 0 \leq z_1 \leq \frac{t}{\sqrt d}, 0 \leq z_2 \leq \frac{t}{\sqrt{d}}, z_3^2 + z_4^2 + \ldots z_d^2 \leq t^2(1 - (\frac{1}{\sqrt{d}} + |a|/t)^2 - (\frac{1}{\sqrt{d}} + |b|/t)^2) \}.$$
        It is easy to see that this region is contained in $B_t(0) \cap B_t(y) \cap \mathbb{H}$. Moreover, for all sufficiently large $d$, its volume is bounded from below by some universal constant times ${\rm vol}(B_t(0))$, as the $(d - 2)$-dimensional volume of a $(d - 2)$-dimensional unit ball is up-to-constant equivalent to $d \cdot {\rm vol}(B_1(0))$ (see~\eqref{eq:lem3.2-0}). This proves Claim~\eqref{eq:lem3.17-claim-1}.

\begin{figure}[h]
\centering
\includegraphics[scale=0.8]{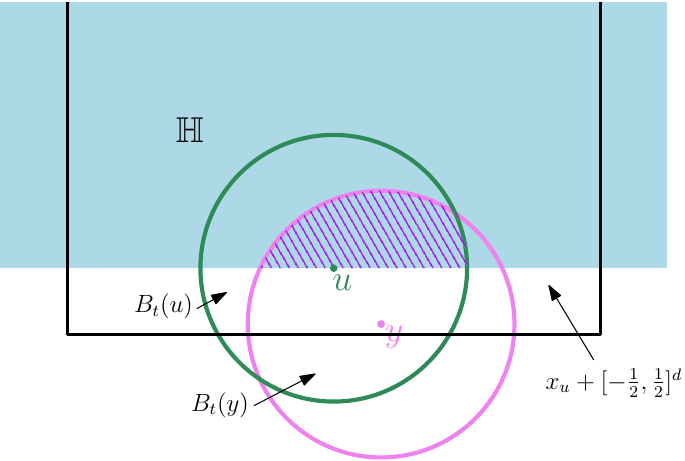}
\caption{There exists a half-space $\mathbb{H}$ such that ${\rm vol}(B_t(y) \cap B_t(u) \cap (x_u + [-\frac12, \frac12]^d)) \geq {\rm vol}(B_t(u) \cap B_t(y) \cap \mathbb{H})$. The shaded domain corresponds to $B_t(u) \cap B_t(y) \cap \mathbb{H}$, whose volume is lower-bounded by some positive universal constant times ${\rm vol}(B_t(0))$.}
\label{fig:area}
\end{figure}

        In addition, for all $2^{-3i} \leq t \leq 2^{-3i+1}$, we have $\frac{{\rm vol}(B_t(y) \setminus A_t)}{{\rm vol}(B_t(0))} = 1 - (1 - \frac{2}{t \mathfrak r} 8^{-j})^d$ (recall~\eqref{eq:def-Z}), which uniformly tends to 0 as $d$ tends to infinity (where we used $i \leq \lfloor j/3 \rfloor - 10 d$). Combining this with~\eqref{eq:lem3.17-correlation} and~\eqref{eq:lem3.17-claim-1}, we obtain that there exists a universal constant $c>0$ such that for all sufficiently large $d$,
        \begin{equation}\label{eq:lem3.17-1}
        \mathbb{E}[Z_i \mathscr{W}_i(x_u, P_i)] \geq \int_{2^{-3i}}^{2^{-3i+1}} t^{-1} c dt = c \log 2 \quad \mbox{for all $i \in [k + 1, \lfloor j/3 \rfloor - 10 d] \cap \mathbb{Z}$}.
        \end{equation}
        
        By the independence of space-time white noise, we can further separate $Z_i$ into two parts $Z_i^{(1)}$ and $Z_i^{(2)}$, where $Z_i^{(1)}$ (resp., $Z_i^{(2)}$) corresponds to integrating over $A_t \setminus (B_t(P_i) \cap (x_u + [-\frac12, \frac12]^d))$ (resp., $A_t \cap B_t(P_i) \cap (x_u + [-\frac12, \frac12]^d)$) in~\eqref{eq:decompose-Zi} rather than over $A_t$. Then, we see that $Z_i^{(1)}, Z_i^{(2)},$ and $\mathscr{W}_i(x_u, P_i) - Z_i^{(2)}$ are independent of each other, and \begin{equation}\label{eq:lem3.17-corre-2}
            \mathbb{E}[(Z_i^{(2)})^2] = \mathbb{E}[(Z_i^{(1)} + Z_i^{(2)})( \mathscr{W}_i(x_u, P_i) - Z_i^{(2)} + Z_i^{(2)})] = \mathbb{E}[Z_i \mathscr{W}_i(x_u, P_i)].
        \end{equation}Now, using~\eqref{eq:lem3.17-1}, we can apply Lemma~\ref{lem:repulsion} with $(X,Y) = (Z_i^{(2)}, \mathscr{W}_i(x_u, P_i) - Z_i^{(2)})$ to show that for all sufficiently large $d$,
        \begin{equation}\label{eq:lem3.17-2}
            \lim_{\beta \rightarrow \infty} \limsup_{k \rightarrow \infty} \sup_{j \geq 6k, k + 1 \leq i \leq \lfloor j/3 \rfloor - 10 d} \mathbb{E} \big[\exp(-Z_i^{(2)}) \big{|} \mathscr{W}_i(x_u, P_i) \geq \beta \mathbb{E}[\mathscr{W}_i(x_u, P_i)^2] \big] = 0.
        \end{equation}
        To be precise, let $\delta>0$ be a small constant that will eventually tend to 0. We apply Lemma~\ref{lem:repulsion} with $\theta = \beta$, this $\delta$, and $(X,Y) = (Z_i^{(2)}, \mathscr{W}_i(x_u, P_i) - Z_i^{(2)})$. Then, $\sigma^2 = \mathbb{E}[(Z_i^{(2)})^2]$, $m \sigma^2 = \mathbb{E}[(\mathscr{W}_i(x_u, P_i) - Z_i^{(2)})^2] =  \mathbb{E}[\mathscr{W}_i(x_u, P_i)^2] - \mathbb{E}[(Z_i^{(2)})^2] $, and the condition $\mathscr{W}_i(x_u, P_i) \geq \beta \mathbb{E}[\mathscr{W}_i(x_u, P_i)^2]$ is equivalent to the condition that $X + Y \geq (m+1) \sigma^2 \theta$. Furthermore, $m$ tends to infinity as $k$ tends to infinity, and by~\eqref{eq:lem3.17-1} and~\eqref{eq:lem3.17-corre-2}, $\sigma^2$ is lower-bounded by a positive universal constant for sufficiently large $d$. Let $\widehat p$ denote the conditional density of $Z_i^{(2)}$ conditioned on the event $\mathscr{W}_i(x_u, P_i) \geq \beta \mathbb{E}[\mathscr{W}_i(x_u, P_i)^2]$. Then, by Lemma~\ref{lem:repulsion}, we have $\widehat p(t) \leq (1+o_k(1)) p[N( \beta \sigma^2, \sigma^2) = t]$ for all $t \leq \delta^{-1}$. Therefore, for all sufficiently large $k$, the conditional expectation of $\exp(-Z_i^{(2)})$ is upper-bounded by $2 \mathbb{E} [e^{-\min \{X, \delta^{-1}\}}]$ where $X$ follows the law $N( \beta \sigma^2, \sigma^2)$. Since $\sigma^2$ is lower-bounded by a universal constant, we can first take $\delta$ to 0 and then send $\beta$ to infinity such that $2 \mathbb{E} [e^{-\min \{X, \delta^{-1} \}}]$ tends to 0. This proves Claim~\eqref{eq:lem3.17-2}.
        
        The inequality~\eqref{eq:lem3.14-4} then simply follows from~\eqref{eq:lem3.17-2}, as we elaborate. By~\eqref{eq:def-Z} and~\eqref{eq:decompose-Zi}, we have
        \begin{equation}\label{eq:decompose-Z}
            Z = \sum_{i= k + 1}^{\lfloor j/3 \rfloor - 10 d} Z_i^{(2)} + \sum_{i = k + 1}^{\lfloor j/3 \rfloor - 10 d} Z_i^{(1)} + \sum_{i=\lfloor j/3 \rfloor - 10 d + 1}^{\lfloor j/3 \rfloor} Z_i.
        \end{equation}Since the last two sums are mean-zero Gaussian random variables with variance at most $(\lfloor j/3 \rfloor - k) \log 2$, we can choose a universal constant $C>0$ such that the last two sums are bounded from below by $-C(j - 3k)$ with probability at least $1 - \frac{1}{2} e^{-200j}$. This also holds when conditioned on the event $\cap_{w \in \widetilde P_0} \mathcal{G}_1(w, P)$, as the later event is increasing. Then we can choose $\beta$ sufficiently large to show that conditioned on the event $\cap_{w \in \widetilde P_0} \mathcal{G}_1(w, P)$, the first sum is at least $(\lambda + C)(j-3k)$ with probability at least $1 - \frac{1}{2} e^{-200j}$. This can be achieved by applying~\eqref{eq:lem3.17-2} and the Markov's inequality $\mathbb{P}[X \leq t] \leq e^t \mathbb{E}[\exp(-X)]$ with $X = \sum_{i= k + 1}^{\lfloor j/3 \rfloor - 10 d} Z_i^{(2)}$, $t = (\lambda + C)(j-3k)$, and $\mathbb{P}$ being the conditional law on the event $\cap_{w \in \widetilde P_0} \mathcal{G}_1(w, P)$. (Note that since $Z_i^{(2)}$ depends only on the white noise in the space $x_u + [-\frac12, \frac12]^d$ and the time $[2^{-3i}, 2^{-3i+1}]$, the law of $Z_i^{(2)}$ conditioned on $\cap_{w \in \widetilde P_0} \mathcal{G}_1(w, P)$ is equivalent to that conditioned on the event $\mathscr{W}_i(x_u, P_i) \geq \beta \mathbb{E}[\mathscr{W}_i(x_u, P_i)^2]$).
    \end{proof}

Now we conclude the proof of~\eqref{eq:lem3.14-2} and Lemma~\ref{lem:condition-WN} by combining Lemmas~\ref{lem:3.16} and \ref{lem:3.17}.

\begin{proof}[Proof of Lemma~\ref{lem:condition-WN}]
    Let $K$ be the constant from Lemma~\ref{lem:3.16}, and apply Lemma~\ref{lem:3.17} with $\lambda = \alpha + K$. Combining Lemmas~\ref{lem:3.16} and \ref{lem:3.17}, we obtain~\eqref{eq:lem3.14-2}, which implies Lemma~\ref{lem:condition-WN}.
\end{proof}

\subsection{First and second moment estimates}\label{subsec:moment-WN}

In this subsection, we study the weighted sum $\mathscr{N}$ defined in~\eqref{eq:def-N} and prove Proposition~\ref{prop:WN-thick}.

We need the following lemma tailored from Proposition 1.2 of~\cite{Liggett-domination}.

\begin{lemma}\label{lem:domination}
    Let $G = (V,E)$ be a graph with maximum degree $\Delta \geq 1$. Let $p \in (0,1)$. Suppose that $\mu$ is a probability measure on $\eta \in \{0,1\}^V$ such that (1). $\mu[\eta(s) = 1] \geq p$ for all $s \in V$; (2). For all disjoint sets $S_1, S_2 \subset V$ that are not neighboring in $G$, i.e., there is no edge in $G$ joining a vertex in $S_1$ to a vertex in $S_2$, the values of $\eta$ on $S_1$ and $S_2$ are independent under the law $\mu$. Let $\theta \in (0,1)$ such that $(1 - \theta) \theta^\Delta \geq 1-p$. Then, we have for all $x \in V$ and $S \subset V$ with $x \not \in S$,
    $$
    \mu[\eta(x) = 1 | \eta(s) = 1 \mbox{ for all $s \in S$}] \geq \theta.
    $$
\end{lemma}
\begin{proof}
    This follows from (1.4) in Proposition 1.2 of~\cite{Liggett-domination} by taking $\Delta + 1$ in place of $\Delta$, together with $r = 1-\theta$ and $\epsilon_1 = \epsilon_2 = \ldots =\epsilon_j =1$. 
\end{proof}

\begin{proof}[Proof of Proposition~\ref{prop:WN-thick}]
    Fix $\alpha>0$ and choose the constant $\beta$ as in Lemma~\ref{lem:condition-WN}, which depends only on $\alpha$. Consider $d \geq \beta$, and let $\widetilde C$ be the constant in Lemma~\ref{lem:condition-WN}. We assume that $k\geq \max \{\widetilde C, C_2 \log d\}$, $n \geq 6k$, and $M \geq 10$, such that these parameters satisfy both the conditions in Lemmas~\ref{lem:gaussian-to-intersection} and~\ref{lem:condition-WN}. Since $\mathbb{E} \mathscr{N} = |\mathscr{P}^{n, M}|$ and
    $$
    \mathbb{E} \mathscr{N}^2 = \sum_{P, Q \in \mathscr{P}^{n, M}} \frac{\mathbb{P}[\mbox{$P,Q$ are good}]}{\mathbb{P}[\mbox{$P$ is good}] \cdot \mathbb{P}[\mbox{$Q$ is good}]},
    $$
    we have
    \begin{equation}\label{eq:proof-prop3.1-0}
    \frac{\mathbb{E} \mathscr{N}^2 }{(\mathbb{E}\mathscr{N})^2} = \widetilde{\mathbf E} \Big[\frac{\mathbb{P}[\mbox{$P,Q$ are good}]}{\mathbb{P}[\mbox{$P$ is good}] \cdot \mathbb{P}[\mbox{$Q$ is good}]}\Big],
    \end{equation}
    where $\widetilde{\mathbf E}$ is the expectation with respect to two paths $P$ and $Q$ uniformly chosen from $\mathscr{P}^{n, M}$, and $\mathbb{P}$ is defined with respect to the space-time white noise which is independent of $P,Q$. Next, we show that the right-hand side of~\eqref{eq:proof-prop3.1-0} is close to 1 in three steps.
    
    \textbf{Step 1.} For $P, Q \in \mathscr{P}^{n,M}$, let $\widetilde P_0, \widetilde Q_0$ be their corresponding paths in $\mathcal{P}^{0,M}$, and let $\mathcal{Y} = \widetilde P_0 \cap \widetilde Q_0$, which is a set of vertices in $\mathbb{Z}^d$. We first use Lemma~\ref{lem:condition-WN} to show that for all paths $P,Q \in \mathscr{P}^{n,M}$,
    \begin{equation}\label{eq:proof-prop3.1-1}
    \frac{\mathbb{P}[\mbox{$P,Q$ are good}]}{\mathbb{P}[\mbox{$P$ is good}] \cdot \mathbb{P}[\mbox{$Q$ is good}]} \leq (1 -  2e^{-k})^{-6|\mathcal{Y}|} \frac{\mathbb{P}[ \cap_{x \in \mathcal Y}( \mathcal{G}_1(x, P) \cap \mathcal{G}_1(x, Q))]}{\mathbb{P}[ \cap_{x \in \mathcal Y} \mathcal{G}_1(x, P)] \cdot \mathbb{P}[ \cap_{x \in \mathcal Y} \mathcal{G}_1(x, Q)]}.
    \end{equation}
    Let $\mathcal{Y}_P = \{x \in \widetilde P_0: \mathfrak d_1(x, \mathcal{Y}) \leq 1\}$ and $\mathcal{Y}_Q = \{x \in \widetilde Q_0: \mathfrak d_1(x, \mathcal{Y}) \leq 1\}$. By Claim~\ref{rmk3.16-claim-2} in Remark~\ref{rmk:3.16}, the following three events are independent:
    \begin{align*}
    &\bigcap_{x \in \widetilde P_0 \setminus \mathcal Y} \mathcal{G}_1(x,P) \cap \bigcap_{x \in \widetilde P_0 \setminus \mathcal{Y}_P} \mathcal{G}_2(x,P), \quad \bigcap_{x \in \widetilde Q_0 \setminus \mathcal Y} \mathcal{G}_1(x,Q) \cap \bigcap_{x \in \widetilde Q_0 \setminus \mathcal{Y}_Q} \mathcal{G}_2(x,Q), \\
    &\mbox{and} \quad \bigcap_{x \in \mathcal Y} (\mathcal{G}_1(x,P) \cap \mathcal{G}_1(x,Q)).
    \end{align*}
    Therefore, we have
    \begin{equation}\label{eq:proof-prop3.1-2}
    \begin{aligned}
        \mathbb{P}[P, Q \mbox{ are good}] &\leq \mathbb{P}\Big[\bigcap_{x \in \widetilde P_0 \setminus \mathcal Y} \mathcal{G}_1(x,P) \cap \bigcap_{x \in \widetilde P_0 \setminus \mathcal{Y}_P} \mathcal{G}_2(x,P) \Big] \times \mathbb{P}\Big[\bigcap_{x \in \widetilde Q_0 \setminus \mathcal Y} \mathcal{G}_1(x,Q) \cap \bigcap_{x \in \widetilde Q_0 \setminus \mathcal{Y}_Q} \mathcal{G}_2(x,Q) \Big]\\
        &\quad \times \mathbb{P} \Big[ \bigcap_{x \in \mathcal Y} (\mathcal{G}_1(x,P) \cap \mathcal{G}_1(x,Q)) \Big].
    \end{aligned}
    \end{equation}
    Next, we lower-bound $\mathbb{P}[\mbox{$P$ is good}]$. By definition, we have
    \begin{equation}\label{eq:proof-prop3.1-3}
    \begin{aligned}
    \mathbb{P}[\mbox{$P$ is good}] &= \mathbb{P}\Big[\bigcap_{x \in \mathcal{Y}_P} \mathcal{G}_2(x,P) \Big{|} \bigcap_{w \in \widetilde P_0 \setminus \mathcal{Y}_P} \mathcal{G}_2(w,P) \cap \bigcap_{w \in \widetilde P_0} \mathcal{G}_1(w,P) \Big] \\
    &\quad \times \mathbb{P} \Big[\bigcap_{x \in \widetilde P_0 \setminus \mathcal{Y}_P} \mathcal{G}_2(x,P) \cap \bigcap_{x \in \widetilde P_0} \mathcal{G}_1(x,P) \Big].
    \end{aligned}
    \end{equation}
    We lower-bound the first term on the right-hand side of~\eqref{eq:proof-prop3.1-3} using Lemma~\ref{lem:condition-WN}. Let $\widehat{\mathbb{P}} = \mathbb{P}[ \cdot| \cap_{w \in \widetilde P_0} \mathcal{G}_1(w,P)]$. By Claim~\ref{rmk3.16-claim-2} in Remark~\ref{rmk:3.16}, under the law $\widehat{\mathbb{P}}$, the events in $\{\mathcal{G}_2(x,P)\}_{x \in \widetilde P_0}$ are independent if their pairwise $|\cdot|_1$-distances are at least 3. Moreover, by Lemma~\ref{lem:condition-WN}, we know that each event happens with probability at least $1 - e^{-k}$ under the law $\widehat{\mathbb{P}}$. Then, we can apply Lemma~\ref{lem:domination} to the indicator functions for $\mathcal{G}_2(x,P)$ to conclude that for all $x \in \mathcal{Y}_P$ and $S \subset \mathcal{Y}_P$,
    \begin{equation}\label{eq:prop3.1-step1-0}
    \widehat{\mathbb{P}}\Big[\mathcal{G}_2(x,P) \Big{|} \bigcap_{w \in S} \mathcal{G}_2(w,P) \Big] \geq 1 - 2 e^{-k}.
    \end{equation}
    To be precise, we take the dependence graph $G$ there to be $\widetilde P_0$, where two vertices on $\widetilde P_0$ are considered connected if their $|\cdot|_1$-distance is at most $2$. By Definition~\ref{def:P0-brw}, the maximum degree $\Delta \leq 5$. The indicator functions for $\mathcal{G}_2(x,P)$ under the law $\widehat{\mathbb{P}}$ satisfy the conditions in Lemma~\ref{lem:domination} with $p = 1 - e^{-k}$. Thus, we can take $\theta = 1 - 2 e^{-k}$ in Lemma~\ref{lem:domination}, which yields~\eqref{eq:prop3.1-step1-0}. Let $\mathcal{Y}_P = \{x_1, \ldots, x_{|\mathcal{Y}_P|}\}$. By~\eqref{eq:prop3.1-step1-0}, we have
    \begin{align*}
        &\quad \mathbb{P}\Big[\bigcap_{x \in \mathcal{Y}_P} \mathcal{G}_2(x,P) \Big{|} \bigcap_{w \in \widetilde P_0 \setminus \mathcal{Y}_P} \mathcal{G}_2(w,P) \cap \bigcap_{w \in \widetilde P_0} \mathcal{G}_1(w,P) \Big] = \widehat{\mathbb{P}}\Big[\bigcap_{x \in \mathcal{Y}_P} \mathcal{G}_2(x,P) \Big{|} \bigcap_{w \in \widetilde P_0 \setminus \mathcal{Y}_P} \mathcal{G}_2(w,P) \Big]\\
        &= \prod_{i=1}^{|\mathcal{Y}_P|} \widehat{\mathbb{P}}\Big[ \mathcal{G}_2(x_i,P) \Big{|} \bigcap_{w \in \widetilde P_0 \setminus \mathcal{Y}_P} \mathcal{G}_2(w,P) \cap \bigcap_{1 \leq p \leq i-1} \mathcal{G}_2(x_p,P) \Big] \geq (1 - 2 e^{-k})^{|\mathcal{Y}_P|}.
    \end{align*}
    Combining this with~\eqref{eq:proof-prop3.1-3} and the fact that $\cap_{x \in \widetilde P_0 \setminus \mathcal Y} \mathcal{G}_1(x,P) \cap \cap_{x \in \widetilde P_0 \setminus \mathcal{Y}_P} \mathcal{G}_2(x,P)$ and $\cap_{x \in \mathcal Y} \mathcal{G}_1(x,P)$ are independent, we further have
    \begin{equation}\label{eq:proof-prop3.1-4}
        \mathbb{P}[\mbox{$P$ is good}] \geq (1 - 2 e^{-k} )^{|\mathcal{Y}_P|} \times \mathbb{P} \Big[\bigcap_{x \in \widetilde P_0 \setminus \mathcal Y} \mathcal{G}_1(x,P) \cap \bigcap_{x \in \widetilde P_0 \setminus \mathcal{Y}_P} \mathcal{G}_2(x,P) \Big] \times \mathbb{P} \Big[\bigcap_{x \in \mathcal Y} \mathcal{G}_1(x,P) \Big].
    \end{equation}
    Similarly, we have
    \begin{equation}\label{eq:proof-prop3.1-5}
        \mathbb{P}[\mbox{$Q$ is good}] \geq (1 - 2 e^{-k})^{|\mathcal{Y}_Q|} \times \mathbb{P} \Big[\bigcap_{x \in \widetilde Q_0 \setminus \mathcal Y} \mathcal{G}_1(x,Q) \cap \bigcap_{x \in \widetilde Q_0 \setminus \mathcal{Y}_Q} \mathcal{G}_2(x,Q) \Big] \times \mathbb{P} \Big[\bigcap_{x \in \mathcal Y} \mathcal{G}_1(x,Q) \Big].
    \end{equation}
    Dividing~\eqref{eq:proof-prop3.1-2} by~\eqref{eq:proof-prop3.1-4} and~\eqref{eq:proof-prop3.1-5}, and using the fact that $|\mathcal{Y}_P| \leq 3|\mathcal{Y}|$ and $|\mathcal{Y}_Q| \leq 3 |\mathcal{Y}|$, we conclude~\eqref{eq:proof-prop3.1-1}.

    \textbf{Step 2.} For $P, Q \in \mathscr{P}^{n,M}$, let $P_j, Q_j$ be their corresponding paths in $\mathscr{P}^{j,M}$ for $0 \leq j \leq n$. Based on~\eqref{eq:proof-prop3.1-1}, we further use Lemmas~\ref{lem:gaussian-correlation} and~\ref{lem:gaussian-to-intersection} to show that 
    \begin{equation}\label{eq:proof-prop3.1-6}
        \frac{\mathbb{P}[\mbox{$P,Q$ are good}]}{\mathbb{P}[\mbox{$P$ is good}] \cdot \mathbb{P}[\mbox{$Q$ is good}]} \leq (1 -  2 e^{-k})^{-6|\mathcal{Y}|} \exp \Big(C_1 C_2 \beta^2 \sum_{j = \lfloor k/2 \rfloor}^n |P_j \cap Q_j| \Big).
    \end{equation}
    By~\eqref{eq:proof-prop3.1-1}, Condition~\eqref{condition-good-1}, and the independence of space-time white noise (see Claim~\ref{rmk3.16-claim-2} in Remark~\ref{rmk:3.16}), we have
    \begin{align*}
    &\quad \frac{\mathbb{P}[\mbox{$P,Q$ are good}]}{\mathbb{P}[\mbox{$P$ is good}] \cdot \mathbb{P}[\mbox{$Q$ is good}]} \\
    &\leq (1 -  2e^{-k})^{-6|\mathcal{Y}|} \prod_{x \in \mathcal Y} \prod_{j = k}^n  \frac{\mathbb{P}\big[\mathscr{W}_j (x, P_j) \geq \beta \mathbb{E}[\mathscr{W}_j (x, P_j)^2], \mathscr{W}_j (x, Q_j) \geq \beta \mathbb{E}[\mathscr{W}_j (x, Q_j)^2] \big]}{\mathbb{P}\big[\mathscr{W}_j (x, P_j) \geq \beta \mathbb{E}[\mathscr{W}_j (x, P_j)^2]\big] \mathbb{P}\big[\mathscr{W}_j (x, Q_j) \geq \beta \mathbb{E}[\mathscr{W}_j (x, Q_j)^2] \big]}.
    \end{align*}
    Applying Lemma~\ref{lem:gaussian-correlation} with $X = \frac{\mathscr{W}_j (x, P_j)}{\mathbb{E}[\mathscr{W}_j (x, P_j)^2]}$ and $Y = \frac{\mathscr{W}_j (x, Q_j)}{\mathbb{E}[\mathscr{W}_j (x, Q_j)^2]}$, we further have
    \begin{equation}\label{eq:prop3.1-step2-1}
    \frac{\mathbb{P}[\mbox{$P,Q$ are good}]}{\mathbb{P}[\mbox{$P$ is good}] \cdot \mathbb{P}[\mbox{$Q$ is good}]} \leq (1 -  2e^{-k})^{-6|\mathcal{Y}|} \prod_{x \in \mathcal Y} \prod_{j = k}^n  \exp \Big( C_1 \beta^2 \mathbb{E}[\mathscr{W}_j (x, P_j) \mathscr{W}_j (x, Q_j)] \Big).
    \end{equation}
    By simple calculations, we obtain that for all $j \in [k, n] \cap \mathbb{Z}$,
    \begin{equation}\label{eq:prop3.1-step2-2}
    \begin{aligned}
    \sum_{x \in \mathcal Y}\mathbb{E}[\mathscr{W}_j (x, P_j) \mathscr{W}_j (x, Q_j)] &= \sum_{x \in \mathcal Y} \int_{2^{-3j}}^{2^{-3j + 1}} \frac{\Vol(B_t(P_j) \cap B_t(Q_j) \cap (x + [-\frac12, \frac12]^d))}{\Vol(B_t(0))} t^{-1} dt \\
    &\leq \int_{2^{-3j}}^{2^{-3j + 1}} \frac{\Vol(B_t(P_j) \cap B_t(Q_j))}{\Vol(B_t(0))} t^{-1} dt.
    \end{aligned}
    \end{equation}
    Combining~\eqref{eq:prop3.1-step2-1} and~\eqref{eq:prop3.1-step2-2} and applying Lemma~\ref{lem:gaussian-to-intersection}, we obtain~\eqref{eq:proof-prop3.1-6}.
    
    \textbf{Step 3.} Let $Y_j = |P_j \cap Q_j|$. By~\eqref{eq:proof-prop3.1-0} and \eqref{eq:proof-prop3.1-6}, we have
    \begin{equation}\label{eq:prop3.1-step3-1}
        \frac{\mathbb{E} \mathscr{N}^2 }{(\mathbb{E}\mathscr{N})^2} \leq  \widetilde{\mathbf E} \Big[\frac{\exp \big(C_1 C_2 \beta^2 \sum_{j = \lfloor k/2 \rfloor}^n Y_j \big)}{(1 -  2e^{-k})^{6|\mathcal{Y}|}} \Big].
    \end{equation}
    The final step is to use Lemmas~\ref{lem:intersection-law-WN} and~\ref{lem:high-d-intersect} to show that, when $d$ is sufficiently large, the right-hand side of~\eqref{eq:prop3.1-step3-1} is upper-bounded by $1 + 2^{-k/10}$ for all sufficiently large $k$ (which depends only on $d$). The proof is similar to that in Section~\ref{subsec:BRW-estimate}.
    
    Recall from Lemma~\ref{lem:intersection-law-WN} that for all $1 \leq i \leq n$, given $(Y_0, Y_1, \ldots, Y_{i-1}, |\mathcal{Y}|)$,\footnote{It is easy to see from the proof of Lemma~\ref{lem:intersect-BRW} that the stochastic dominance statement holds even when conditioned on $P_{i-1}$ and $Q_{i-1}$ which determine the random variables $Y_0, Y_1,\ldots, Y_{i-1}$, and $|\mathcal{Y}|$.}
    \begin{equation*}
    Y_i\mbox{ is stochastically dominated by } 11(Z_1 + Z_2 + \ldots + Z_{2 Y_{i-1}})
    \end{equation*}
    where $\{Z_l\}_{l \geq 1}$ are i.i.d.\ Bernoulli random variables with probability $\frac{50}{\mathfrak q - 10}$ to be equal to 1. Therefore, for all $a >1$, given $(Y_0, Y_1, \ldots, Y_{i-1}, |\mathcal{Y}|)$, we have
    \begin{equation}\label{eq:prop3.1-step3-2}
    \widetilde{\mathbf E}[a^{Y_{i}}| Y_0,Y_1,\ldots, Y_{i-1}, |\mathcal{Y}|] \leq \widetilde{\mathbf E}[a^{11(Z_1 + Z_2 + \ldots + Z_{2Y_{i-1}})}] = \big( 1 + \frac{50}{\mathfrak q - 10}(a^{11} - 1) \big)^{2Y_{i-1}}.
    \end{equation}
    Consider the sequence defined by $a_1 = e^{C_1 C_2 \beta^2}$ and $a_{l+1} = e^{C_1 C_2 \beta^2} ( 1 + \frac{50}{\mathfrak q - 10}(a_l^{11} - 1) )^2$ for all $l \geq 1$. Applying Claim~\eqref{lem-sequence-1} in Lemma~\ref{lem:sequence}, we have that for all sufficiently large $d$, $\sup_{l \geq 1} a_l \leq 2  e^{C_1 C_2 \beta^2}$. Then, applying~\eqref{eq:prop3.1-step3-2} to~\eqref{eq:prop3.1-step3-1} inductively, we obtain that for all sufficiently large $d$,
    \begin{equation}\label{eq:prop3.1-step3-3}
    \begin{aligned}
        \frac{\mathbb{E} \mathscr{N}^2 }{(\mathbb{E}\mathscr{N})^2} &=  \widetilde{\mathbf E}\Big[\widetilde{\mathbf E} \Big[\frac{\exp \big(C_1 C_2 \beta^2 \sum_{j = \lfloor k/2 \rfloor}^{n-1} Y_j \big)}{(1 -  2e^{-k})^{6|\mathcal{Y}|}} \times a_1^{Y_n} \Big{|} Y_0,Y_1,\ldots, Y_{n-1}, |\mathcal{Y}| \Big] \Big]\\
        &\leq \widetilde{\mathbf E}\Big[\frac{\exp \big(C_1 C_2 \beta^2 \sum_{j = \lfloor k/2 \rfloor}^{n-2} Y_j \big)}{(1 -  2e^{-k})^{6|\mathcal{Y}|}} \times a_2^{Y_{n-1}} \Big] \leq \ldots \leq \widetilde{\mathbf E}\Big[\frac{1}{(1 -  2e^{-k})^{6|\mathcal{Y}|}} \times a_{n+1 - \lfloor k/2 \rfloor}^{Y_{\lfloor k/2 \rfloor}} \Big]\\
        &\leq \widetilde{\mathbf E}\Big[\frac{1}{(1 -  2e^{-k})^{6|\mathcal{Y}|}} \times \big(2 e^{C_1 C_2 \beta^2}\big)^{Y_{\lfloor k/2 \rfloor}} \Big].
    \end{aligned}
    \end{equation}
    Next, we consider the sequence $b_1 = 2 e^{C_1 C_2 \beta^2}$ and $b_{l+1} = ( 1 + \frac{50}{\mathfrak q - 10}(b_l^{11} - 1) )^2$ for all $l \geq 1$. Applying Claim~\eqref{lem-sequence-2} in Lemma~\ref{lem:sequence}, we have that for all sufficiently large $d$, $b_l \leq 1 + 2^{-l}$ for all $l \geq 2$. Applying~\eqref{eq:prop3.1-step3-2} to~\eqref{eq:prop3.1-step3-3} inductively, we obtain that for all sufficiently large $d$,
    \begin{equation}\label{eq:prop3.1-step3-4}
    \begin{aligned}
        \frac{\mathbb{E} \mathscr{N}^2}{(\mathbb{E}\mathscr{N})^2} &\leq \widetilde{\mathbf E}\Big[\widetilde{\mathbf E}\Big[\frac{1}{(1 -  2e^{-k})^{6|\mathcal{Y}|}} \times b_1^{Y_{\lfloor k/2 \rfloor}} |Y_0,Y_1,\ldots, Y_{\lfloor k/2 \rfloor - 1}, |\mathcal{Y}|\Big]\\
        &\leq \widetilde{\mathbf E}\Big[\frac{1}{(1 -  2e^{-k})^{6|\mathcal{Y}|}} \times b_2^{Y_{\lfloor k/2 \rfloor} - 1} \Big] \leq \ldots \leq \widetilde{\mathbf E}\Big[\frac{1}{(1 -  2e^{-k})^{6|\mathcal{Y}|}} \times b_{\lfloor k/2 \rfloor + 1}^{Y_0} \Big]\\
        &\leq \widetilde{\mathbf E}\Big[\frac{1}{(1 -  2e^{-k})^{6|\mathcal{Y}|}} \times (1 + 2^{-k/2})^{Y_0} \Big].
    \end{aligned}
    \end{equation}
    By Lemma~\ref{lem:high-d-intersect}, for any $a>1$, we have $\widetilde{\mathbf E}[a^{|\mathcal{Y}|}] \leq \widetilde{\mathbf E}[a^X] = \frac{1 - c_1}{1 - c_1 a}$, where $X$ is a geometric random variable with success probability $c_1$. Combining this with the fact that $Y_0 \leq 2 \mathfrak r |\mathcal{Y}|$, we have for all sufficiently large $k$ (which may depend on $d$),
    $$
    \frac{\mathbb{E} \mathscr{N}^2}{(\mathbb{E}\mathscr{N})^2} \leq \widetilde{\mathbf E}\Big[\big((1 -  2e^{-k})^{-6}(1 + 2^{-k/2})^{2 \mathfrak r} \big)^{|\mathcal{Y}|} \Big] \leq \widetilde{\mathbf E}\Big[\big( 1 + 2^{-k/5} \big)^{|\mathcal{Y}|} \Big] \leq \frac{1 - c_1}{1- c_1(1+2^{-k/5})} \leq 1 + 2^{-k/10}.
    $$
    Combining this with the Cauchy–Schwarz inequality $\mathbb{P}[\mathscr{N} > 0] \geq \frac{(\mathbb{E}\mathscr{N})^2}{\mathbb{E} \mathscr{N}^2}$, we complete the proof of Proposition~\ref{prop:WN-thick} by recalling from Claim~\ref{rmk3.16-claim-1} of Remark~\ref{rmk:3.16} that a good path in $\mathscr{P}^{n,M}$ satisfies~\eqref{eq:prop-WN-1} in Proposition~\ref{prop:WN-thick} with the same $n,k,M$. \qedhere
\end{proof}

\subsection{Proof of Theorems~\ref{thm:WN-thick} and~\ref{thm:exponential-metric}}\label{subsec:proof-WN}

In this subsection, we prove Theorems~\ref{thm:WN-thick} and~\ref{thm:exponential-metric} using Proposition~\ref{prop:WN-thick}.

\begin{proof}[Proof of Theorem~\ref{thm:WN-thick}]
   Fix $\alpha>0$ and consider sufficiently large $d$ that satisfies Proposition~\ref{prop:WN-thick}. Fix an integer $k \geq 1$, which will eventually tend to infinity. For integers $n \geq 6k$ and $M \geq 10$, let $\mathscr{K}(n,M)$ be the event that there exists a path in $\mathscr{P}^{n,M}$ that satisfies~\eqref{eq:prop-WN-1} with the same $n,k,M$. Since the set $\mathscr{P}^{n,M}$ satisfies a restriction property with respect to $M$ (see Lemma~\ref{lem:restrict-M-WN}) and refines the paths in $\mathscr{P}^{n-1,M}$, the event $\mathscr{K}(n,M)$ is decreasing with respect to $n,M$. By Proposition~\ref{prop:WN-thick}, we have
   \begin{equation}\label{eq:proof-thm1.3-1}
       \mathbb{P}\Big[\bigcap_{n \geq 6k, M \geq 10} \mathscr{K}(n,M)\Big] \geq 1 - \epsilon(k), \quad \mbox{with $\lim_{k \rightarrow \infty} \epsilon(k) = 0$}.
   \end{equation}
   We claim that on the event $\cap_{n \geq 6k, M \geq 10} \mathscr{K}(n,M)$, the set $\TWN_\alpha$ contains an unbounded path. Suppose that the event $\cap_{n \geq 6k, M \geq 10} \mathscr{K}(n,M)$ occurs. Then, we can choose a path $P_n \in \mathscr{P}^{n,n}$ satisfying~\eqref{eq:prop-WN-1} with the same $n,k$ for all $n \geq \max \{6 k, 10\}$. By compactness (see Section~\ref{subsec:sec2-final-proof}), there exists a subsequence $\{n_m\}_{m \geq 1}$ such that $\{P_{n_m}\}_{m \geq 1}$ converges to a limiting set $\widetilde P$ with respect to the local Hausdorff distance. By Lemma~\ref{lem:Hausdorff-limit-WN}, $\widetilde P$ is a continuous path connecting $\{z : |z|_1 = 2\}$ to infinity. Moreover, by~\eqref{eq:prop-WN-1}, we have
   \begin{equation}\label{eq:proof-thm1.3-2}
   h_j(x) - h_{3k}(x) \geq \alpha (j-3k) \quad \mbox{for all $j \geq 6k$ and $x \in \widetilde P$}.
   \end{equation}
   This is because for all $j \geq 6k$ and $x \in \widetilde P$, we can choose a sequence of points $x_m \in P_{n_m}$ such that $\lim_{m \rightarrow \infty} x_m = x$. By~\eqref{eq:prop-WN-1}, $h_j(x_m) - h_{3k}(x_m) \geq \alpha ( j - 3k)$ for all sufficiently large $m$. Taking $m$ to infinity yields~\eqref{eq:proof-thm1.3-2}. Therefore, all the points on $\widetilde P$ satisfy that $\liminf_{j \rightarrow \infty} \frac{1}{j} h_j(x) \geq \alpha \geq \sqrt{\log 2} \alpha$, and hence belong to $\TWN_{\alpha}$. This proves the claim. Taking $k$ to infinity and using ~\eqref{eq:proof-thm1.3-1} yields Theorem~\ref{thm:WN-thick}.
\end{proof}

Next, we prove Theorem~\ref{thm:exponential-metric}. To this end, we need the following variant of Proposition~\ref{prop:WN-thick}. We will focus on the pair of sets $\{z : |z|_1 \leq 2 \}$ and $\{z : |z - 2M \mathbf e_1 |_1 \leq 2\}$. However, the analog holds for any two disjoint compact sets $K_1, K_2$ with non-empty interiors by rotating or rescaling the discrete paths that we consider. Furthermore, for fixed $\alpha>0$, the condition on the dimension $d$ in the following Proposition~\ref{prop:WN-thick-variant} is in fact the same as the condition in Proposition~\ref{prop:WN-thick}.

\begin{proposition}\label{prop:WN-thick-variant}
    For any fixed $\alpha>0$, the following holds for all sufficiently large $d$. There exists a function $\epsilon : \mathbb{N} \rightarrow (0,1)$ depending on $\alpha$ and $d$ with $\lim_{k \rightarrow \infty} \epsilon(k) = 0$ such that for all integers $k \geq 1$, $n \geq 6k$ and all $M \geq 10$, with probability at least $1-\epsilon(k)$, there exists a path $P$ on $\frac{1}{\mathfrak r} 2^{-3n} \mathbb{Z}^d$ with length at most $4 M \mathfrak r \times 11^n$ that connects $\{z : |z|_1 \leq 2 \}$ and $\{z : |z - 2M \mathbf e_1 |_1 \leq 2\}$ in $\{z : |z|_1 \leq 10M \}$,\footnote{The set $\{z : |z|_1 \leq 10M \}$ is chosen to be an arbitrary box that contains both $\{z : |z|_1 \leq 2 \}$ and $\{z : |z - 2M \mathbf e_1 |_1 \leq 2\}$.} and satisfies
    \begin{equation}\label{eq:prop-WN-variant}
    h_j(x) - h_{3k}(x) \geq \alpha (j-3k) \quad \mbox{for all $j \in [6k, n] \cap \mathbb{Z}$ and $x \in \overline{P}$}.
    \end{equation}
\end{proposition}

\begin{proof}
    Recall that in the proof of Proposition~\ref{prop:WN-thick}, we start with a set of paths $\mathcal{P}^{0,M}$ that connect $\{z : |z|_1 = 1\}$ to $\{z : |z|_1 = M\}$. To prove this variant, we just need to start with a set of paths $\widetilde{\mathcal{P}^{0,M}}$ that connect $\{z : |z|_1 = 1\}$ to $\{z : |z - 2M\mathbf e_1|_1 = 1\}$ such that the intersecting number of two paths chosen uniformly from $\widetilde{\mathcal{P}^{0,M}}$ can be bounded similar to Lemma~\ref{lem:high-d-intersect} in high dimensions, and the rest of the proof is the same as Proposition~\ref{prop:WN-thick}. 
    
    We can construct $\widetilde{\mathcal{P}^{0,M}}$ as follows. Take $(i_1, i_2,\ldots, i_M)$ from $\{2,3,\ldots,d\}^{M}$ and consider a nearest-neighbor path that starts from $\mathbf e_{i_1}$ and moves in the following directions chronologically
    $$\mathbf e_1, \mathbf e_{i_2}, \mathbf e_1, \ldots, \mathbf e_{i_M}, \mathbf e_1, \mathbf e_1, -\mathbf e_{i_M}, \mathbf e_1, -\mathbf e_{i_{M-1}}, \mathbf e_1,\ldots, -\mathbf e_{i_2}, \mathbf e_1.$$ 
    In particular, this path ends at $2M\mathbf e_1 + \mathbf e_{i_1}$ and has length $4M-1$. The set $\widetilde{\mathcal{P}^{0,M}}$ will consist of all these paths. For a path $P \in \widetilde{\mathcal{P}^{0,M}}$ generated in the above way, we let $\widehat P$ denote the oriented path defined as in Definition~\ref{def:P0-brw} using the same sequence $(i_1, i_2, \ldots, i_M)$. Namely, $\widehat P = ( \mathbf e_{i_1},  \mathbf e_{i_1} +  \mathbf e_{i_2}, \ldots, \sum_{k=1}^M  \mathbf e_{i_k})$. For two paths $P,Q$ chosen uniformly from $\widetilde{\mathcal{P}^{0,M}}$, it is easy to see that $|P \cap Q| \leq 4 |\widehat P \cap \widehat Q|$. Furthermore, by Lemma~\ref{lem:high-d-intersect}, we have that $|\widehat P \cap \widehat Q|$ can be stochastically dominated by a geometric random variable $G$ with success probability $c_1 = c_1(d-1)$ for $d \geq 4$. Here, we have $d-1$ instead of $d$ because the number of choices for each step is $|\{2,3,...,d\}| = d-1$. Therefore, $|P \cap Q|$ can be stochastically dominated by $4G$.
    
    Then, we can follow the proof of Proposition~\ref{prop:WN-thick} in Sections~\ref{subsec:path-WN}--\ref{subsec:moment-WN} to inductively refine the paths in $\widetilde{\mathcal{P}^{0,M}}$ on $\frac{1}{\mathfrak r} 2^{-3n} \mathbb{Z}^d$ and estimate the number of good paths in them. This proves the proposition. By~\eqref{eq:def-mathfrak-l} and Claim~\ref{property-1-scr} in Lemma~\ref{lem:basic-property-mathscrP}, the refinement on $\frac{1}{\mathfrak r} 2^{-3n} \mathbb{Z}^d$ has length at most $(4M-1)\mathfrak r \times 11^n \leq 4M \mathfrak r \times 11^n$ and connects $\{z : |z|_1 \leq 2 \}$ and $\{z : |z - 2M \mathbf e_1 |_1 \leq 2\}$ in $\{z : |z|_1 \leq 10M \}$ (in fact, it connects $\{z : |z|_1 \leq 1 + \frac{2}{7 \mathfrak r} \}$ and $\{z : |z - 2M \mathbf e_1 |_1 \leq 1 + \frac{2}{7 \mathfrak r}\}$). \qedhere
\end{proof}

\begin{proof}[Proof of Theorem~\ref{thm:exponential-metric}]
    Fix $A>0$. Let $H > 2 \log 2$ be any fixed constant. We apply Proposition~\ref{prop:WN-thick-variant} with $\alpha = AH$ and consider sufficiently large $d$ that satisfies Proposition~\ref{prop:WN-thick-variant}. Without loss of generality, assume that $K_1 \supset \{z:|z|_1 \leq 2\}$ and $K_2 \supset \{z:|z - 2M \mathbf e_1|_1 \leq 2\}$ for some integer $M \geq 10$. Other cases can be derived from the same argument using the corresponding analog of Proposition~\ref{prop:WN-thick-variant}. 
    
    Let $n,k \geq 1$ be two integers. We will let $n$ tend to infinity and $k$ tend to infinity slowly with $n$. Applying Proposition~\ref{prop:WN-thick-variant} with this $\alpha,n,k,M$ and using the symmetry of $h_n$, we have that with probability at least $1 - \epsilon(k)$, there exists a path $P$ on $\frac{1}{\mathfrak r} 2^{-3n} \mathbb{Z}^d$ with length at most $4 M \mathfrak r \times 11^n$ that connects $K_1$ and $K_2$ in $\{z : |z|_1 \leq 10 M\}$, and satisfies
    $$
    h_n(x) - h_{3k}(x) \leq - \alpha ( n - 3k) \quad \mbox{for all $x \in \overline{P}$}.
    $$
    This implies that with probability at least $1 - \epsilon(k) - o_n(1)$, $h_n(x) \leq -\frac{\alpha}{2}n$ for all $x \in \overline{P}$. Since $P$ is a path on $\frac{1}{\mathfrak r} 2^{-3n} \mathbb{Z}^d$ with length at most $4 M \mathfrak r \times 11^n$, the Euclidean length of $\overline P$ is at most $\frac{1}{\mathfrak r} 2^{-3n} \times |P| \leq 4M \times (\frac{11}{8})^n$. Then, by Definition~\eqref{eq:def-exponential-metric}, we have
    $$
    D_n(K_1,K_2) \leq \int_0^1 e^{\xi h_n(\overline{P}(t))} | \overline{P}'(t)| dt \leq e^{-\xi \frac{\alpha}{2} n} \times 4M (\frac{11}{8})^n.
    $$
    The right-hand side is upper-bounded by $2^{-A \xi n}$ for all $\xi > \frac{\log (11/8)}{\alpha/2 - A \log 2} = \frac{\log (11/8)}{A(H/2 - \log 2)}$ and sufficiently large $n$ (which may depend on $M, \alpha, \xi$). Taking $k$ to infinity slowly with $n$ and using $\lim_{k \rightarrow \infty} \epsilon(k) = 0$ yields Theorem~\ref{thm:exponential-metric} with $L = \frac{\log (11/8)}{A(H/2 - \log 2)} + 1$. Note that by Proposition~\ref{prop:WN-thick-variant}, $H$ can be sent to infinity slowly with $d$, and thus we can require that $\lim_{x \to \infty} L(x) = 0$.
\end{proof}

\section{Thick points via different approximations}\label{sec:approx}

In this section, we show that the thick point sets defined via different approximations are in fact the same. This follows from similar arguments to~\cite{CH-thick-point}, and we include a proof for completeness. In this section, we assume that
\begin{equation}\label{eq:def-rho}
    \mbox{$\rho$ is a non-negative, compactly supported and smooth function on $\mathbb{R}^d$ with $\int_{\mathbb{R}^d} \rho(x) dx = 1$.}
\end{equation}For integers $n \geq 1$, let $\rho_n(x) = 2^{nd} \rho(2^n x)$. Recall from~\eqref{eq:def-WN} the definition of $h$ and $h_n$. In this section, we only consider $|\cdot|_2$-distance, so we omit $2$ in the subscript.

\begin{lemma}\label{lem:sec4-1}
    For any $\rho$ satisfying~\eqref{eq:def-rho}, there exists a constant $C>0$, which depends on $\rho$ and $d$, such that with probability at least $1 - C e^{-n^{4/3}/C}$, we have
    \begin{equation}\label{eq:lem4.1}
        \sup_{x \in [0,1]^d} \big{|} h_n(x) - h * \rho_n(x) \big{|} \leq n^{2/3}.
    \end{equation}
\end{lemma}

\begin{proof}
    The constant $C$ in this proof may change from line to line but only depends on $\rho$ and $d$. Suppose that $\rho$ is supported in $B(0,M)$ for some $M \geq 1$. For simplicity, we write $\tilde h_n (x) = h * \rho_n (x)$ and $Z_n(x) = h_n(x) - \tilde h_n(x)$ for $x \in \mathbb{R}^d$. We claim that there exists a constant $A>0$, depending only on $\rho$ and $d$, such that
    \begin{enumerate}[(1)]
        \item $\mathbb{E} Z_n(x)^2 \leq A$ for all $x \in \mathbb{R}^d$.\label{lem4.1-claim1}
        \item $\mathbb{E}(Z_n(x) - Z_n(y))^2 \leq A n 2^n |x-y|$ for all $x,y \in \mathbb{R}^d$ with $|x-y| \leq 2 \cdot 2^{-n}$.\label{lem4.1-claim2}
    \end{enumerate}
    We first prove~\eqref{eq:lem4.1} assuming these two inequalities. By Dudley's inequality (see e.g.~Equation (2.38) of~\cite{Talagrand-book14}), for some constant $C>0$, we have
    \begin{equation}\label{eq:lem4.1-1}
    \mathbb{E} \Big[\sup_{y \in B(x, \frac{1}{n} 2^{-n})} |Z_n(y) - Z_n(x)| \Big] \leq C \quad \mbox{for all $x \in \mathbb{R}^d$}.
    \end{equation}
    The proof is similar to that of~\eqref{eq:lem3.16-3}, and thus we omit the details. Therefore, by Claim~\eqref{lem4.1-claim2}, \eqref{eq:lem4.1-1}, and Borell-TIS inequality (see e.g.~\cite[Theorem 2.1.1]{adler-taylor-fields}), we further have
    $$
    \mathbb{P}\big[ \sup_{y \in B(x, \frac{1}{n} 2^{-n})} |Z_n(y) - Z_n(x)| \geq \frac{1}{2} n^{2/3}\big] \leq Ce^{-n^{4/3}/C} \quad \mbox{for some constant $C>0$.}
    $$
    We can cover $[0,1]^d$ with $C n^{d} 2^{nd}$ balls of Euclidean radius $\frac{1}{n} 2^{-n}$. Let $J$ be the centers of these balls. Then, by the above inequality and Claim~\eqref{lem4.1-claim1}, we have
    \begin{align*}
    &\quad \mathbb{P}\big[\sup_{x \in [0,1]^d} |Z_n(x)| \geq n^{2/3}\big] \\
    &\leq \sum_{x \in J} \mathbb{P}\big[|Z_n(x)| \geq \frac{1}{2} n^{2/3}\big] + \sum_{x \in J} \mathbb{P}\big[ \sup_{y \in B(x, \frac{1}{n} 2^{-n})} |Z_n(y) - Z_n(x)| \geq \frac{1}{2} n^{2/3}\big] \leq Ce^{-n^{4/3}/C}.
    \end{align*}
    This yields~\eqref{eq:lem4.1}.

    In the rest of the proof, we prove Claims~\eqref{lem4.1-claim1} and~\eqref{lem4.1-claim2}. To prove Claim~\eqref{lem4.1-claim1}, it suffices to show that for all $x \in \mathbb{R}^d$,
    \begin{equation}\label{eq:lem4.1-2}
    \mathbb{E}[h_n(x)^2] = n \log 2; \quad \mathbb{E} [\tilde h_n(x)^2] = n \log 2 + O(1); \quad \mathbb{E} [h_n(x) \tilde h_n(x)] = n \log 2 + O(1).
    \end{equation}
    The first equation in~\eqref{eq:lem4.1-2} follows from simple calculations (see \cite[Lemma 2.3]{dgz-exponential-metric}). The second equation in~\eqref{eq:lem4.1-2} follows from the fact that for all $|x - y| \leq 10 M$,
    \begin{equation}\label{eq:lem4.1-3}
    \mathbb{E} [h(x) h(y)] = \int_0^1 t^{-1} \KK * \KK (\frac{x-y}{t}) dt = \int_{|x-y|/2}^1 t^{-1} (1 - C\frac{|x-y|}{t}) dt + O(1) = \log \frac{1}{|x-y|} + O(1).
    \end{equation}
    In the second equality, we used the relation $\KK * \KK (x) = 1 - C |x| + O(|x|^2)$ as $|x|$ tends to 0, and $\KK * \KK$ is supported in $B_2(0)$. Here, the $O(1)$ terms depend only on $\rho, M, d$. Therefore, we have
    \begin{align*}
        \mathbb{E} [\tilde h_n(x)^2] &= \iint_{y_1, y_2 \in B(x, M \cdot 2^{-n})} \mathbb{E} [h(y_1) h (y_2)] \times \rho_n(x - y_1) \rho_n(x - y_2) dy_1 dy_2 \\
        &= \iint_{y_1, y_2 \in B(x, M \cdot 2^{-n})} \big(\log \frac{1}{|y_1 - y_2|} + O(1) \big) \times \rho_n(x - y_1) \rho_n(x - y_2) dy_1 dy_2 = n \log 2 + O(1).
    \end{align*}
    Similarly, for all $|x-y| \leq M \cdot 2^{-n}$, we have
    $$
    \mathbb{E} [h_n(x) h_n(y)] = \int_{2^{-n}}^1 t^{-1} \KK * \KK (\frac{x-y}{t}) dt = \int_{2^{-n}}^1 t^{-1} (1 - C\frac{|x-y|}{t} + O(\frac{|x-y|^2}{t^2})) dt = n \log 2 + O(1),
    $$
    which implies the third equation in~\eqref{eq:lem4.1-2} as follows:
    \begin{align*}
        \mathbb{E} [h_n(x) \tilde h_n(x)] &= \int_{y \in B(x, M \cdot 2^{-n})} \mathbb{E} [h_n(x) h_n(y)] \times \rho_n(x - y) dy \\
        &= \int_{y \in B(x, M \cdot 2^{-n})} \big(n \log 2 + O(1) \big) \times \rho(x - y) dy = n \log 2 + O(1).
    \end{align*}
    Combining the equalities in~\eqref{eq:lem4.1-2} yields Claim~\eqref{lem4.1-claim1}.
    
    Next, we prove Claim~\eqref{lem4.1-claim2}. It suffices to show that for some constant $C>0$, and for all $x,y \in \mathbb{R}^d$ with $|x-y| \leq 2 \cdot 2^{-n}$,
    \begin{equation}\label{eq:lem4.1-4}
    \mathbb{E} [(h_n(x) - h_n(y))^2] \leq C 2^n |x-y| \quad \mbox{and} \quad \mathbb{E} [(\tilde h_n(x) - \tilde h_n(y))^2] \leq C n 2^n |x-y|.
    \end{equation}
    The first inequality can be derived from $\mathbb{E} [h_n(x)^2] = \mathbb{E} [h_n(y)^2] = n \log 2$ and the following estimate:
    \begin{align*}
        \mathbb{E} [h_n(x) h_n(y)] = \int_{2^{-n}}^1 t^{-1} \KK * \KK (\frac{x-y}{t}) dt \geq \int_{2^{-n}}^1 t^{-1} ( 1 - C\frac{|x-y|}{t}) dt \geq n \log 2 - C 2^n |x-y|,
    \end{align*}
    where we used the fact that $\KK * \KK (x) \geq 1 - C |x| $ for some constant $C>0$. The second inequality in~\eqref{eq:lem4.1-4} can be derived as follows:
    \begin{align*}
        &\quad \mathbb{E} [(\tilde h_n(x) - \tilde h_n(y))^2] = 2 \mathbb{E} [h * \rho_n(x) (h * \rho_n(x) - h * \rho_n(y))]\\
        &=2 \iint_{y_1, y_2 \in B(x, (M+2) \cdot 2^{-n})} \mathbb{E}[h(y_1) h (y_2)] \times \rho_n(x - y_1) (\rho_n(x - y_2) - \rho_n(y - y_2)) dy_1 dy_2 \\
        &\leq C \iint_{y_1, y_2 \in B(x, (M+2) \cdot 2^{-n})} \big(\log \frac{1}{|y_1 - y_2|} + O(1) \big) \times \rho_n(x - y_1) \times 2^{nd} 2^n|x-y|  dy_1 dy_2 \leq Cn 2^n|x-y|.
    \end{align*}
    In the third line, we used~\eqref{eq:lem4.1-3} and the fact that the derivative of $\rho$ is bounded.
\end{proof}

Recall the whole-plane log-correlated field $\phi$ defined in~\eqref{eq:cov-lgf}, where it is normalized such that the average over the unit sphere is zero. As explained in Section 4.1.1 of~\cite{lgf-survey}, $\phi$ differs from $h$ by a random continuous function.

\begin{lemma}\label{lem:lgf-approx}
    There exists a coupling of the space-time white noise $W(dy,dt)$ and $\phi$ such that $h - \phi$ is a random continuous function.
\end{lemma}

\begin{proof}
    This result follows from the argument in Appendix B of~\cite{afs-metric-ball}. Their argument was explained in two dimensions but the same proof works for $d \geq 3$. Their argument was also stated only for a smooth function, but the only place they used this assumption was to show that $\phi_{-\infty,0}$, defined there, is a random continuous function defined modulo a global additive constant. (They argued this by showing that $\nabla \phi_{-\infty,0}$ is a random continuous function.) For the function $\KK$, which is not smooth, we can instead use Gaussian process estimates (similar to those in Lemma~\ref{lem:sec4-1}) to show that $\phi_{-\infty,0}(x) - \phi_{-\infty, 0}(0)$ is a random continuous function for $x \in \mathbb{R}^d$. Thus, their Lemma B.2 applies to the function $\KK$.
\end{proof}

Next, we show that under the coupling in Lemma~\ref{lem:lgf-approx}, the thick point sets $\T_\alpha$ and $\TWN_\alpha$, defined in~\eqref{eq:lgf-thick} and~\eqref{eq:WN-thick}, are almost surely the same. As a consequence, we see that the thick point set of $\phi$ does not depend on the choice of $\rho$. 

\begin{lemma}\label{lem:approx-thick}
    For any $\rho$ satisfying~\eqref{eq:def-rho} and $\alpha>0$, under the coupling in Lemma~\ref{lem:lgf-approx}, we have $\T_\alpha = \TWN_\alpha$ almost surely.
\end{lemma}
\begin{proof}
    Fix $\rho$ and $\alpha>0$. By Lemma~\ref{lem:sec4-1} and the Borel–Cantelli lemma, we have
    $$
    \limsup_{n \rightarrow \infty} \sup_{x \in [0,1]^d} \frac{|h * \rho_n(x) - h_n(x)|}{n^{2/3}} \leq 1 \quad \mbox{a.s.}
    $$
    Therefore, we have
    $$
    \{x \in [0,1]^d: \liminf_{n \rightarrow \infty} \frac{h *\rho_n(x)}{n} \geq \alpha \sqrt{\log 2} \} = \{ x \in [0,1]^d: \liminf_{n \rightarrow \infty} \frac{h_n(x)}{n} \geq \alpha \sqrt{\log 2} \} \quad \mbox{a.s.}
    $$
    Together with the translation invariance of $h$ and $h_n$, we obtain
    $$
    \{x \in \mathbb{R}^d: \liminf_{n \rightarrow \infty} \frac{h * \rho_n(x)}{n} \geq \alpha \sqrt{\log 2} \} = \{ x \in \mathbb{R}^d: \liminf_{n \rightarrow \infty} \frac{h_n(x)}{n} \geq \alpha \sqrt{\log 2} \} \quad \mbox{a.s.}
    $$
    Under the coupling in Lemma~\ref{lem:lgf-approx}, the left-hand side equals $\T_\alpha$ while the right-hand side equals $\TWN_\alpha$. This yields the lemma.
\end{proof}

Finally we conclude the proof of Theorems~\ref{thm:lgf-thick} and~\ref{thm:crossing}.

\begin{proof}[Proof of Theorem~\ref{thm:lgf-thick}]
    Combine Theorem~\ref{thm:WN-thick} and Lemma~\ref{lem:approx-thick}.
\end{proof}

\begin{proof}[Proof of Theorem~\ref{thm:crossing}]
    It suffices to show that there exists a constant $C = C(\alpha)$ such that for all $d \geq C$, any fixed box $B$ and any two opposite faces of $B$, we have
    \begin{equation}\label{eq:thm-crossing-fixed}
        \mbox{the set $\TWN_\alpha \cap B$ contains a path connecting these two faces almost surely.}
    \end{equation}This then implies that Claim~\eqref{eq:thm-crossing-fixed} holds for all rational boxes (i.e., boxes with rational coordinates) almost surely, which can be extended to all boxes, since for any box $B$ and any two opposite faces of $B$, there exist a rational box $B'$ and a pair of opposite faces of $B'$ such that the crossing of $B'$ is also a crossing of $B$. To prove Claim~\eqref{eq:thm-crossing-fixed}, we can first prove its analog for $\TWN_\alpha$ using the method in Section~\ref{sec:WN}; see Remark~\ref{rmk:path-choose} and Proposition~\ref{prop:WN-thick-variant}. This, combined with Lemma~\ref{lem:approx-thick}, yields Claim~\eqref{eq:thm-crossing-fixed}, thereby proving Theorem~\ref{thm:crossing}.
\end{proof}

\section{Open problems}\label{sec:open}

In this section, we list some open problems about the log-correlated Gaussian field (recall~\eqref{eq:cov-lgf}).

\subsubsection*{Minimal dimension, threshold of $\alpha$, and phase transition}

In Theorem~\ref{thm:lgf-thick}, we proved that the thick point set of the log-correlated Gaussian field $\T_\alpha$ may contain an unbounded connected component in sufficiently high dimensions. In two dimensions, it was known from~\cite{APP-disconnect} that the set $\T_\alpha$ is totally disconnected for all $\alpha>0$. A natural question is what happens in the intermediate dimensions. Notice that the log-correlated Gaussian field satisfies a restriction property: restricting the $d$-dimensional log-correlated Gaussian field to a $(d-1)$-dimensional hyperplane yields the $(d-1)$-dimensional log-correlated Gaussian field (see \cite[Proposition 1.4]{lgf-survey} and Section 7 of~\cite{fgf-survey}). Therefore, there should exist a minimal dimension $d_c \geq 3$ such that for all $d \geq d_c$, there exists $\alpha = \alpha(d) > 0$ such that $\T_\alpha$ almost surely has a non-singleton connected component.

\begin{prob}\label{prop:dimension-threshold}
    Find the minimal dimension $d_c$.
\end{prob}

The cases for $d = 3$ or $4$ may be of particular interest. In particular, for $d = 4$, the log-correlated Gaussian field corresponds to the bi-Laplacian Gaussian field (also known as the membrane model) and is related to some lattice models~\cite{LawlerSunWu-LERW}. 

It is also natural to consider the dependence of $\T_\alpha$ on $\alpha$ for a fixed dimension. We define $\alpha_c$ to be the largest $\alpha$ such that $\T_\alpha$ almost surely contains a non-singleton connected component. Since $\T_\alpha$ is a decreasing set in $\alpha$, this property holds for all $0 < \alpha < \alpha_c$. We can also introduce another threshold $\alpha_{\infty, c}$ corresponding to the largest $\alpha$ such that $\T_\alpha$ contains an unbounded connected component. By definition, $\alpha_{\infty, c} \leq \alpha_c$. It would be interesting to study the phase transition of $\T_\alpha$ with respect to $\alpha$.

\begin{prob}\label{prob:threshold-alpha}
    Determine the value of $\alpha_c$ or $\alpha_{\infty, c}$ for any $d \geq 3$.
\end{prob}

\begin{prob}
    Does $\T_{\alpha_c}$ contain a non-singleton connected component? Does $\T_{\alpha_{\infty, c}}$ contain an unbounded connected component?
\end{prob}

An upper bound for $\alpha_c$ can be given as follows. By \cite[Proposition 3.5]{falconer-fractal-geometry}, a sufficient condition for a set to be totally disconnected is that its Hausdorff dimension is strictly less than 1. Since the Hausdorff dimension of $\T_\alpha$ is almost surely $d - \frac{\alpha^2}{2}$~\cite{hmp-thick-pts, CH-thick-point}, we conclude that $\alpha_c \leq \sqrt{2 (d-1)}$. 

It is also interesting to study whether all the non-singleton connected components in $\T_\alpha$ are bounded or unbounded. There could be three different phases.

\begin{prob}
    Which of the following three phases will occur or will not occur?
    \begin{enumerate}
        \item $\T_\alpha$ contains only bounded non-singleton connected components.

        \item $\T_\alpha$ contains only unbounded non-singleton connected components.

        \item $\T_\alpha$ contains both bounded and unbounded non-singleton connected components.
    \end{enumerate}
\end{prob}

\subsubsection*{Geometric properties of non-singleton connected components}

In the case that $\T_\alpha$ contains non-singleton connected components, let $\mathscr{C}_\alpha$ be the set of these components. It would also be interesting to study the geometric properties of $\mathscr{C}_\alpha$. A basic question is the Hausdorff dimension of $\mathscr{C}_\alpha$, which is upper-bounded by that of $\T_\alpha$.

\begin{prob}
    What is the Hausdorff dimension of $\mathscr{C}_\alpha$?
\end{prob}

One can also ask other geometric properties of $\mathscr{C}_\alpha$ for instance whether the set $\mathscr{C}_\alpha$ contains some specific geometric objects. 

\begin{prob}
     Is there a loop or hyper-surface in $\mathscr{C}_\alpha$?
\end{prob}

In fact, our argument can likely be used to show that for any fixed $\alpha>0$ and $n \geq 1$, the set $\T_\alpha$ contains an unbounded $n$-dimensional hyper-surface for all sufficiently large $d$. The case $n=1$ corresponds to an unbounded path. This is because we should be able to construct a family of discrete $n$-dimensional surfaces that do not intersect much in high dimensions. Then, we can refine this family of surfaces on sub-lattices and estimate the number of good surfaces it contains similar to Section~\ref{sec:WN}.

\subsubsection*{Super-supercritical exponential metric}

As explained in Remark~\ref{rmk:super-super}, Theorem~\ref{thm:exponential-metric} suggests that there might be a new phase for the exponential metric where the set-to-set distance exponent (if it exists) is negative. It would be interesting to construct the exponential metric in this case.

\begin{prob}
    Is it possible to construct an exponential metric associated with the log-correlated Gaussian field in the super-supercritical phase?
\end{prob}

\bibliographystyle{alpha}
\bibliography{theta,cibib}

@book{pw-gff-notes,
    AUTHOR = {Werner, Wendelin and Powell, Ellen},
     TITLE = {Lecture notes on the {G}aussian free field},
    SERIES = {Cours Sp\'{e}cialis\'{e}s [Specialized Courses]},
    VOLUME = {28},
 PUBLISHER = {Soci\'{e}t\'{e} Math\'{e}matique de France, Paris},
      YEAR = {2021},
     PAGES = {vi+171},
      ISBN = {978-2-85629-952-4},
   MRCLASS = {60-02 (60G15 60G60 60J67 82B20 82B21)},
  MRNUMBER = {4466634},
       eprint = {\arxiv{2004.04720}},
}

@article{ss-contour,
       AUTHOR = {Schramm, Oded and Sheffield, Scott},
     TITLE = {A contour line of the continuum {G}aussian free field},
   JOURNAL = {Probab. Theory Related Fields},
  FJOURNAL = {Probability Theory and Related Fields},
    VOLUME = {157},
      YEAR = {2013},
    NUMBER = {1-2},
     PAGES = {47--80},
      ISSN = {0178-8051},
     CODEN = {PTRFEU},
   MRCLASS = {Preliminary Data},
  MRNUMBER = {3101840},
       DOI = {10.1007/s00440-012-0449-9},
       URL = {http://dx.doi.org/10.1007/s00440-012-0449-9},
       eprint={\arxiv{1008.2447}}, 
       }

@article{hmp-thick-pts,
AUTHOR = {Hu, Xiaoyu and Miller, Jason and Peres, Yuval},
     TITLE = {Thick points of the {G}aussian free field},
   JOURNAL = {Ann. Probab.},
  FJOURNAL = {The Annals of Probability},
    VOLUME = {38},
      YEAR = {2010},
    NUMBER = {2},
     PAGES = {896--926},
      ISSN = {0091-1798},
     CODEN = {APBYAE},
   MRCLASS = {60G15 (28A80 60G18 60G60 81T40)},
  MRNUMBER = {2642894 (2011c:60117)},
       DOI = {10.1214/09-AOP498},
       URL = {http://dx.doi.org/10.1214/09-AOP498},
eprint={\arxiv{0902.3842}}
}

@ARTICLE{dgz-exponential-metric,
       author = {{Ding}, Jian and {Gwynne}, Ewain and {Zhuang}, Zijie},
        title = "{Tightness of exponential metrics for log-correlated Gaussian fields in arbitrary dimension}",
      journal = {ArXiv e-prints},
         year = 2023,
        month = oct, 
          doi = {10.48550/arXiv.2310.03996},
archivePrefix = {arXiv},
       eprint = {\arxiv{2310.03996}},
 primaryClass = {math.PR},
       adsurl = {https://ui.adsabs.harvard.edu/abs/2023arXiv231003996D}
}

@ARTICLE{cg-support-thm,
       author = {{Contreras Hip}, Andres A. and {Gwynne}, Ewain},
        title = "{A support theorem for Liouville quantum gravity}",
      journal = {ArXiv e-prints},
         year = 2023,
        month = may, 
          doi = {10.48550/arXiv.2305.15588},
archivePrefix = {arXiv},
       eprint = {\arxiv{2305.15588}},
 primaryClass = {math.PR},
       adsurl = {https://ui.adsabs.harvard.edu/abs/2023arXiv230515588C}
}

@article{dg-uniqueness,
    AUTHOR = {Ding, Jian and Gwynne, Ewain},
     TITLE = {Uniqueness of the critical and supercritical {L}iouville
              quantum gravity metrics},
   JOURNAL = {Proc. Lond. Math. Soc. (3)},
  FJOURNAL = {Proceedings of the London Mathematical Society. Third Series},
    VOLUME = {126},
      YEAR = {2023},
    NUMBER = {1},
     PAGES = {216--333},
      ISSN = {0024-6115},
   MRCLASS = {60G60 (60D05 83C45)},
  MRNUMBER = {4535021},
       eprint = {\arxiv{2110.00177}}, 
}

@ARTICLE{ddg-metric-survey,
       author = {{Ding}, Jian and {Dubedat}, Julien and {Gwynne}, Ewain},
        title = "{Introduction to the Liouville quantum gravity metric}",
      journal = {ArXiv e-prints},
         year = 2021,
        month = sep, 
archivePrefix = {arXiv},
       eprint = {\arxiv{2109.01252}},
 primaryClass = {math.PR},
       adsurl = {https://ui.adsabs.harvard.edu/abs/2021arXiv210901252D}
}

@ARTICLE{dg-critical-lqg,
       author = {{Ding}, Jian and {Gwynne}, Ewain},
        title = "{The critical Liouville quantum gravity metric induces the Euclidean topology}",
      journal = {Frontiers of Mathematics}, 
         year = 2021,  
       eprint = {\arxiv{2108.12067}},
       volume={to appear}
}

@article{pfeffer-supercritical-lqg, 
       author = {{Pfeffer}, Joshua},
        title = "{Weak Liouville quantum gravity metrics with matter central charge $\mathbf{c} \in (-\infty, 25)$}",
      journal = {ArXiv e-prints},
         year = 2021,
        month = apr, 
archivePrefix = {arXiv},
       eprint = {\arxiv{2104.04020}},
 primaryClass = {math.PR},
       adsurl = {https://ui.adsabs.harvard.edu/abs/2021arXiv210404020P}
}

@article{lfpp-pos,
    AUTHOR = {Ding, Jian and Gwynne, Ewain and Sep\'{u}lveda, Avelio},
     TITLE = {The distance exponent for {L}iouville first passage
              percolation is positive},
   JOURNAL = {Probab. Theory Related Fields},
  FJOURNAL = {Probability Theory and Related Fields},
    VOLUME = {181},
      YEAR = {2021},
    NUMBER = {4},
     PAGES = {1035--1051},
      ISSN = {0178-8051},
   MRCLASS = {60D05 (60G60)},
  MRNUMBER = {4344137},
       DOI = {10.1007/s00440-021-01093-x},
       URL = {https://doi.org/10.1007/s00440-021-01093-x},
       eprint = {\arxiv{2005.13570}}
}

@article{dg-supercritical-lfpp,
    AUTHOR = {Ding, Jian and Gwynne, Ewain},
     TITLE = {Tightness of supercritical {L}iouville first passage
              percolation},
   JOURNAL = {J. Eur. Math. Soc. (JEMS)},
  FJOURNAL = {Journal of the European Mathematical Society (JEMS)},
    VOLUME = {25},
      YEAR = {2023},
    NUMBER = {10},
     PAGES = {3833--3911},
      ISSN = {1435-9855},
   MRCLASS = {Prelim},
  MRNUMBER = {4634685},
       DOI = {10.4171/jems/1273},
       URL = {https://doi.org/10.4171/jems/1273},
       eprint = {\arxiv{2005.13576}},
}

@article{afs-metric-ball, 
    AUTHOR = {Ang, Morris and Falconet, Hugo and Sun, Xin},
     TITLE = {Volume of metric balls in {L}iouville quantum gravity},
   JOURNAL = {Electron. J. Probab.},
  FJOURNAL = {Electronic Journal of Probability},
    VOLUME = {25},
      YEAR = {2020},
     PAGES = {Paper No. 160, 50},
   MRCLASS = {60D05 (60G15 60G60)},
  MRNUMBER = {4193901},
       DOI = {10.1214/20-ejp564},
       URL = {https://doi.org/10.1214/20-ejp564},
       eprint = {\arxiv{2001.11467}},
}

@article{gm-uniqueness, 
    AUTHOR = {Gwynne, Ewain and Miller, Jason},
     TITLE = {Existence and uniqueness of the {L}iouville quantum gravity
              metric for {$\gamma\in(0,2)$}},
   JOURNAL = {Invent. Math.},
  FJOURNAL = {Inventiones Mathematicae},
    VOLUME = {223},
      YEAR = {2021},
    NUMBER = {1},
     PAGES = {213--333},
      ISSN = {0020-9910},
   MRCLASS = {83C45 (30F99 58 60)},
  MRNUMBER = {4199443},
       DOI = {10.1007/s00222-020-00991-6},
       URL = {https://doi.org/10.1007/s00222-020-00991-6},
   eprint = {\arxiv{1905.00383}},
}

@article{dddf-lfpp,
    AUTHOR = {Ding, Jian and Dub\'{e}dat, Julien and Dunlap, Alexander and
              Falconet, Hugo},
     TITLE = {Tightness of {L}iouville first passage percolation for
              {$\gamma \in (0,2)$}},
   JOURNAL = {Publ. Math. Inst. Hautes \'{E}tudes Sci.},
  FJOURNAL = {Publications Math\'{e}matiques. Institut de Hautes \'{E}tudes
              Scientifiques},
    VOLUME = {132},
      YEAR = {2020},
     PAGES = {353--403},
      ISSN = {0073-8301},
   MRCLASS = {60K35 (60G15 83C45)},
  MRNUMBER = {4179836},
       DOI = {10.1007/s10240-020-00121-1},
       URL = {https://doi.org/10.1007/s10240-020-00121-1},
       eprint = {\arxiv{1904.08021}},
}

@ARTICLE{dg-lqg-dim,
   author = {{Ding}, J. and {Gwynne}, E.},
    title = "{The fractal dimension of {L}iouville quantum gravity: universality, monotonicity, and bounds}",
  journal = {{C}ommunications in {M}athematical {P}hysics}, 
  volume = {374},
  numver = {3},
  pages= {1877-1934},
   eprint = {\arxiv{1807.01072}},  
     year = 2020,
   adsurl = {http://adsabs.harvard.edu/abs/2018arXiv180701072D}
}

@article {pitt-positively-correlated,
    AUTHOR = {Pitt, Loren D.},
     TITLE = {Positively correlated normal variables are associated},
   JOURNAL = {Ann. Probab.},
  FJOURNAL = {The Annals of Probability},
    VOLUME = {10},
      YEAR = {1982},
    NUMBER = {2},
     PAGES = {496--499},
      ISSN = {0091-1798},
   MRCLASS = {62H20},
  MRNUMBER = {665603},
       URL ={http://links.jstor.org/sici?sici=0091-1798(198205)10:2<496:PCNVAA>2.0.CO;2-Q&origin=MSN},
}

@book {mandelbrot-book,
    AUTHOR = {Mandelbrot, Benoit B.},
     TITLE = {The fractal geometry of nature},
      NOTE = {Schriftenreihe f\"ur den Referenten. [Series for the Referee]},
 PUBLISHER = {W. H. Freeman and Co., San Francisco, Calif.},
      YEAR = {1982},
     PAGES = {v+460},
      ISBN = {0-7167-1186-9},
   MRCLASS = {00A69 (51-01 54F45 85A35 86A99)},
  MRNUMBER = {665254},
MRREVIEWER = {S. Dubuc},
}

@article {fgf-survey,
    AUTHOR = {Lodhia, Asad and Sheffield, Scott and Sun, Xin and Watson,
              Samuel S.},
     TITLE = {Fractional {G}aussian fields: a survey},
   JOURNAL = {Probab. Surv.},
  FJOURNAL = {Probability Surveys},
    VOLUME = {13},
      YEAR = {2016},
     PAGES = {1--56},
   MRCLASS = {60G15 (60G20 60G60)},
  MRNUMBER = {3466837},
MRREVIEWER = {Serge Cohen},
       DOI = {10.1214/14-PS243},
       URL = {https://doi.org/10.1214/14-PS243},
   eprint = {\arxiv{1407.5598}}
}

@incollection {lgf-survey,
    AUTHOR = {Duplantier, Bertrand and Rhodes, R\'{e}mi and Sheffield, Scott and
              Vargas, Vincent},
     TITLE = {Log-correlated {G}aussian fields: an overview},
 BOOKTITLE = {Geometry, analysis and probability},
    SERIES = {Progr. Math.},
    VOLUME = {310},
     PAGES = {191--216},
 PUBLISHER = {Birkh\"{a}user/Springer, Cham},
      YEAR = {2017},
   MRCLASS = {60D05 (60G22)},
  MRNUMBER = {3821928},
       DOI = {10.1007/978-3-319-49638-2\_9},
       URL = {https://doi.org/10.1007/978-3-319-49638-2_9},
   eprint = {\arxiv{1407.5605}}
}

@preamble{
   "\def\cprime{$'$} "
}

@book {falconer-fractal-geometry,
    AUTHOR = {Falconer, Kenneth},
     TITLE = {Fractal geometry},
   EDITION = {Third},
      NOTE = {Mathematical foundations and applications},
 PUBLISHER = {John Wiley \& Sons, Ltd., Chichester},
      YEAR = {2014},
     PAGES = {xxx+368},
      ISBN = {978-1-119-94239-9},
   MRCLASS = {28-01 (11K55 28A78 28A80 37C45 37F10)},
  MRNUMBER = {3236784},
MRREVIEWER = {Manuel Mor\'{a}n},
}

@book{durrett,
     AUTHOR = {Durrett, Rick},
     TITLE = {Probability: theory and examples},
    SERIES = {Cambridge Series in Statistical and Probabilistic Mathematics},
   EDITION = {Fourth},
 PUBLISHER = {Cambridge University Press},
   ADDRESS = {Cambridge},
      YEAR = {2010},
     PAGES = {x+428},
      ISBN = {978-0-521-76539-8},
   MRCLASS = {60-01},
  MRNUMBER = {2722836 (2011e:60001)},
       }

@book{lawler-limic-walks,     
    AUTHOR = {Lawler, Gregory F. and Limic, Vlada},
     TITLE = {Random walk: a modern introduction},
    SERIES = {Cambridge Studies in Advanced Mathematics},
    VOLUME = {123},
 PUBLISHER = {Cambridge University Press, Cambridge},
      YEAR = {2010},
     PAGES = {xii+364},
      ISBN = {978-0-521-51918-2},
   MRCLASS = {60G50 (60-02)},
  MRNUMBER = {2677157 (2012a:60132)},
MRREVIEWER = {Andrew R. Wade},
       DOI = {10.1017/CBO9780511750854},
       URL = {http://dx.doi.org/10.1017/CBO9780511750854},
}

@book {adler-taylor-fields,
    AUTHOR = {Adler, Robert J. and Taylor, Jonathan E.},
     TITLE = {Random fields and geometry},
    SERIES = {Springer Monographs in Mathematics},
 PUBLISHER = {Springer, New York},
      YEAR = {2007},
     PAGES = {xviii+448},
      ISBN = {978-0-387-48112-8},
   MRCLASS = {60G60 (58J65)},
  MRNUMBER = {2319516 (2008m:60090)},
MRREVIEWER = {Jos{\'e} Rafael Le{\'o}n},
}

@article {CH-thick-point,
    AUTHOR = {Cipriani, Alessandra and Hazra, Rajat Subhra},
     TITLE = {Thick points for {G}aussian free fields with different
              cut-offs},
   JOURNAL = {Ann. Inst. Henri Poincar\'e{} Probab. Stat.},
  FJOURNAL = {Annales de l'Institut Henri Poincar\'e{} Probabilit\'es et
              Statistiques},
    VOLUME = {53},
      YEAR = {2017},
    NUMBER = {1},
     PAGES = {79--97},
      ISSN = {0246-0203,1778-7017},
   MRCLASS = {60G60 (60G15)},
  MRNUMBER = {3606735},
MRREVIEWER = {Yoshihiro\ Abe},
       DOI = {10.1214/15-AIHP709},
       URL = {https://doi.org/10.1214/15-AIHP709},
}

@article {Chen-thick,
    AUTHOR = {Chen, Linan},
     TITLE = {Thick points of high-dimensional {G}aussian free fields},
   JOURNAL = {Ann. Inst. Henri Poincar\'e{} Probab. Stat.},
  FJOURNAL = {Annales de l'Institut Henri Poincar\'e{} Probabilit\'es et
              Statistiques},
    VOLUME = {54},
      YEAR = {2018},
    NUMBER = {3},
     PAGES = {1492--1526},
      ISSN = {0246-0203,1778-7017},
   MRCLASS = {60G60 (60G15 60G17 60H07)},
  MRNUMBER = {3825889},
       DOI = {10.1214/17-AIHP846},
       URL = {https://doi.org/10.1214/17-AIHP846},
}

@article {APP-disconnect,
    AUTHOR = {Aru, Juhan and Papon, L\'eonie and Powell, Ellen},
     TITLE = {Thick points of the planar {GFF} are totally disconnected for
              all {$\gamma\ne 0$}},
   JOURNAL = {Electron. J. Probab.},
  FJOURNAL = {Electronic Journal of Probability},
    VOLUME = {28},
      YEAR = {2023},
     PAGES = {Paper No. 85, 24},
      ISSN = {1083-6489},
   MRCLASS = {60J67 (60G15 60G60 60G70)},
  MRNUMBER = {4609448},
MRREVIEWER = {Anatoliy\ Malyarenko},
       DOI = {10.1214/23-ejp975},
       URL = {https://doi.org/10.1214/23-ejp975},
}

@book {Talagrand-book14,
    AUTHOR = {Talagrand, Michel},
     TITLE = {Upper and lower bounds for stochastic processes},
    SERIES = {Ergebnisse der Mathematik und ihrer Grenzgebiete. 3. Folge. A
              Series of Modern Surveys in Mathematics [Results in
              Mathematics and Related Areas. 3rd Series. A Series of Modern
              Surveys in Mathematics]},
    VOLUME = {60},
      NOTE = {Modern methods and classical problems},
 PUBLISHER = {Springer, Heidelberg},
      YEAR = {2014},
     PAGES = {xvi+626},
      ISBN = {978-3-642-54074-5; 978-3-642-54075-2},
   MRCLASS = {60G17 (46B99 60B11 60G15 60G52)},
  MRNUMBER = {3184689},
MRREVIEWER = {Sasha\ Sodin},
       DOI = {10.1007/978-3-642-54075-2},
       URL = {https://doi.org/10.1007/978-3-642-54075-2},
}

@article {CCD-fractal,
    AUTHOR = {Chayes, J. T. and Chayes, L. and Durrett, R.},
     TITLE = {Connectivity properties of {M}andelbrot's percolation process},
   JOURNAL = {Probab. Theory Related Fields},
  FJOURNAL = {Probability Theory and Related Fields},
    VOLUME = {77},
      YEAR = {1988},
    NUMBER = {3},
     PAGES = {307--324},
      ISSN = {0178-8051,1432-2064},
   MRCLASS = {60K35 (82A43)},
  MRNUMBER = {931500},
MRREVIEWER = {Massimo\ Campanino},
       DOI = {10.1007/BF00319291},
       URL = {https://doi.org/10.1007/BF00319291},
}

@article{Liggett-domination,
    AUTHOR = {Liggett, T. M. and Schonmann, R. H. and Stacey, A. M.},
     TITLE = {Domination by product measures},
   JOURNAL = {Ann. Probab.},
  FJOURNAL = {The Annals of Probability},
    VOLUME = {25},
      YEAR = {1997},
    NUMBER = {1},
     PAGES = {71--95},
      ISSN = {0091-1798,2168-894X},
   MRCLASS = {60G60 (60G10 60K35)},
  MRNUMBER = {1428500},
MRREVIEWER = {Vincent\ De Valk},
       DOI = {10.1214/aop/1024404279},
       URL = {https://doi.org/10.1214/aop/1024404279},
}

@article {Mad15,
    AUTHOR = {T.~Madaule},
     TITLE = {Maximum of a log-correlated {G}aussian field},
   JOURNAL = {Ann. Inst. Henri Poincar\'{e} Probab. Stat.},
    VOLUME = {51},
      YEAR = {2015},
    NUMBER = {4},
     PAGES = {1369--1431},
}

@article {LawlerSunWu-LERW,
    AUTHOR = {Lawler, Gregory and Sun, Xin and Wu, Wei},
     TITLE = {Four-dimensional loop-erased random walk},
   JOURNAL = {Ann. Probab.},
  FJOURNAL = {The Annals of Probability},
    VOLUME = {47},
      YEAR = {2019},
    NUMBER = {6},
     PAGES = {3866--3910},
      ISSN = {0091-1798,2168-894X},
   MRCLASS = {60G50 (60K35)},
  MRNUMBER = {4038044},
MRREVIEWER = {Thomas\ Polaski},
       DOI = {10.1214/19-aop1349},
       URL = {https://doi.org/10.1214/19-aop1349},
}

@incollection {Chayes-notes,
    AUTHOR = {Chayes, Lincoln},
     TITLE = {Aspects of the fractal percolation process},
 BOOKTITLE = {Fractal geometry and stochastics ({F}insterbergen, 1994)},
    SERIES = {Progr. Probab.},
    VOLUME = {37},
     PAGES = {113--143},
 PUBLISHER = {Birkh\"auser, Basel},
      YEAR = {1995},
      ISBN = {3-7643-5263-9},
   MRCLASS = {60K35},
  MRNUMBER = {1391973},
MRREVIEWER = {Yu\ Zhang},
       DOI = {10.1007/978-3-0348-7755-8\_6},
       URL = {https://doi.org/10.1007/978-3-0348-7755-8_6},
}

@article {Chayes-fractal-3d,
    AUTHOR = {Chayes, J. T. and Chayes, L. and Grannan, E. and Swindle, G.},
     TITLE = {Phase transitions in {M}andelbrot's percolation process in
              three dimensions},
   JOURNAL = {Probab. Theory Related Fields},
  FJOURNAL = {Probability Theory and Related Fields},
    VOLUME = {90},
      YEAR = {1991},
    NUMBER = {3},
     PAGES = {291--300},
      ISSN = {0178-8051,1432-2064},
   MRCLASS = {60K35 (60J80 82B26 82B43)},
  MRNUMBER = {1133368},
MRREVIEWER = {F.\ M.\ Dekking},
       DOI = {10.1007/BF01193747},
       URL = {https://doi.org/10.1007/BF01193747},
}

@incollection {Kesten-high,
    AUTHOR = {Kesten, Harry},
     TITLE = {Asymptotics in high dimensions for percolation},
 BOOKTITLE = {Disorder in physical systems},
    SERIES = {Oxford Sci. Publ.},
     PAGES = {219--240},
 PUBLISHER = {Oxford Univ. Press, New York},
      YEAR = {1990},
   MRCLASS = {60K35 (82B43)},
  MRNUMBER = {1064563},
MRREVIEWER = {J. Theodore Cox},
}

@article {JSW-decomposition,
    AUTHOR = {Junnila, Janne and Saksman, Eero and Webb, Christian},
     TITLE = {Decompositions of log-correlated fields with applications},
   JOURNAL = {Ann. Appl. Probab.},
  FJOURNAL = {The Annals of Applied Probability},
    VOLUME = {29},
      YEAR = {2019},
    NUMBER = {6},
     PAGES = {3786--3820},
      ISSN = {1050-5164},
   MRCLASS = {60G15 (60G57)},
  MRNUMBER = {4047992},
       DOI = {10.1214/19-AAP1492},
       URL = {https://doi.org/10.1214/19-AAP1492},
}

\end{document}